\documentclass[a4paper,reqno,11pt]{amsart}

\usepackage{amsmath, amsfonts, amssymb, amsthm, amscd}
\usepackage{graphicx}
\usepackage{psfrag}
\usepackage{perpage}
\usepackage{url}
\usepackage{color}
\usepackage{mathrsfs}
\usepackage{mathabx}
\usepackage{scalerel}
\usepackage{subdepth}

\usepackage{dsfont} 

\usepackage{bbm}
\usepackage[utf8]{inputenc}
\usepackage[T1]{fontenc}
\usepackage{lmodern}
\usepackage{microtype}

\usepackage{amsmath}
\usepackage{mathdots}
\usepackage{color}
\usepackage{array}
\usepackage{multirow}
\usepackage{amssymb}
\usepackage{tabularx}
\usepackage{booktabs}

\usepackage[a4paper,scale={0.72,0.74},marginratio={1:1},footskip=7mm,headsep=10mm]{geometry}

\usepackage{hyperref}

\makeatletter
\def\@secnumfont{\bfseries\scshape}

\def\section{\@startsection{section}{1}
  \z@{.9\linespacing\@plus\linespacing}{.5\linespacing}%
  {\normalfont\large\bfseries\scshape\centering}}

\def\subsection{\@startsection{subsection}{2}%
  \z@{.5\linespacing\@plus.7\linespacing}{-.5em}%
  {\normalfont\bfseries\scshape}}

\def\@secnumfont{\scshape}

\def\subsubsection{\@startsection{subsubsection}{3}%
  \z@{.5\linespacing\@plus.7\linespacing}{-.5em}%
  {\normalfont\scshape}}

\def\specialsection{\@startsection{section}{1}%
  \z@{\linespacing\@plus\linespacing}{.5\linespacing}%
  {\normalfont\centering\large\bfseries\scshape}}
\makeatother

\makeatletter

\renewenvironment{proof}[1][\proofname]{\par
\pushQED{\qed}%
\normalfont \topsep4\p@\@plus4\p@\relax
\trivlist
\item[\hskip\labelsep
\bfseries
#1\@addpunct{.}]\ignorespaces
}{%
\popQED\endtrivlist\@endpefalse
}
\makeatother

\makeatletter
\newcommand \Dotfill {\leavevmode \leaders \hb@xt@ 6pt{\hss .\hss }\hfill \kern \z@}
\makeatother

\makeatletter
\def\@tocline#1#2#3#4#5#6#7{\relax
  \ifnum #1>\c@tocdepth 
  \else
    \par \addpenalty\@secpenalty\addvspace{#2}%
    \begingroup \hyphenpenalty\@M
    \@ifempty{#4}{%
      \@tempdima\csname r@tocindent\number#1\endcsname\relax
    }{%
      \@tempdima#4\relax
    }%
    \parindent\z@ \leftskip#3\relax \advance\leftskip\@tempdima\relax
    \rightskip\@pnumwidth plus4em \parfillskip-\@pnumwidth
    #5\leavevmode\hskip-\@tempdima
      \ifcase #1
       \or\or \hskip 1.65em \or \hskip 3.3em \else \hskip 4.95em \fi%
      #6\nobreak\relax
    \Dotfill
    \hbox to\@pnumwidth{\@tocpagenum{#7}}\par
    \nobreak
    \endgroup
  \fi}
\makeatother

\makeatletter
\def\l@section{\@tocline{1}{0pt}{1pc}{}{\scshape}}
\renewcommand{\tocsection}[3]{%
\indentlabel{\@ifnotempty{#2}{\ignorespaces#1 #2.\hskip 0.7em}}#3}
\def\l@subsection{\@tocline{2}{0pt}{1pc}{5pc}{}}

\def\l@subsubsection{\@tocline{3}{0pt}{1pc}{7pc}{}}

\makeatother

\numberwithin{equation}{section}

\newtheoremstyle{mytheorem}{.7\linespacing\@plus.3\linespacing}{.7\linespacing\@plus.3\linespacing}%
     {\itshape}
     {}
     {\bfseries}
     {. }
     {0.3ex}
     {\thmname{{\bfseries #1}}\thmnumber{ {\bfseries #2}}\thmnote{ (#3)}}  

\theoremstyle{mytheorem}

\newtheorem{theorem}{Theorem}[section]
\newtheorem{lemma}[theorem]{Lemma}
\newtheorem{proposition}[theorem]{Proposition}

\newtheorem{remark}[theorem]{Remark}

\newtheorem{assumption}[theorem]{Assumption}

\newtheorem{example}[theorem]{Example}

\newcommand{\bbE}{{\ensuremath{\mathbb E}} }

\newcommand{\bbP}{{\ensuremath{\mathbb P}} }

\newcommand{\bbT}{{\ensuremath{\mathbb T}} }

\newcommand{\bx}{{\ensuremath{\bold x}} }
\newcommand{\by}{{\ensuremath{\bold y}} }
\newcommand{\bz}{{\ensuremath{\bold z}} }

\newcommand{\cB}{{\ensuremath{\mathcal B}} }

\newcommand{\cF}{{\ensuremath{\mathcal F}} }

\newcommand{\cM}{{\ensuremath{\mathcal M}} }
\newcommand{\cN}{{\ensuremath{\mathcal N}} }

\newcommand{\cV}{{\ensuremath{\mathcal V}} }
\newcommand{\cW}{{\ensuremath{\mathcal W}} }

\newcommand{\cZ}{{\ensuremath{\mathcal Z}} }

\DeclareMathSymbol{\leqslant}{\mathalpha}{AMSa}{"36} 
\DeclareMathSymbol{\geqslant}{\mathalpha}{AMSa}{"3E} 
\DeclareMathSymbol{\eset}{\mathalpha}{AMSb}{"3F}     
\newcommand{\dd}{\text{\rm d}}             

\newcommand{\be}{\begin{equation}}
\newcommand{\ee}{\end{equation}}

\newcommand{\R}{\mathbb{R}}

\newcommand{\Z}{\mathbb{Z}}
\newcommand{\N}{\mathbb{N}}

\newcommand{\T}{\mathbb{T}}

\newcommand{\PEfont}{\mathrm}

\newcommand{\p}{\ensuremath{\PEfont P}}

\newcommand{\E}{\PEfont E}
\renewcommand{\P}{\p}

\DeclareMathOperator{\bbvar}{\ensuremath{\mathbb{V}ar}}
\DeclareMathOperator{\bbcov}{\ensuremath{\mathbb{C}ov}}

\newcommand{\ind}{\mathds{1}}

\newcommand{\eps}{\varepsilon}
\renewcommand{\epsilon}{\varepsilon}
\renewcommand{\phi}{\varphi}
\renewcommand{\theta}{\vartheta}
\renewcommand{\rho}{\varrho}

\newenvironment{myenumerate}{%
\renewcommand{\theenumi}{\arabic{enumi}}%
\renewcommand{\labelenumi}{{\rm(\theenumi)}}%
\begin{list}{\labelenumi}
	{%
	\setlength{\itemsep}{0.4em}%
	\setlength{\topsep}{0.5em}%
	\setlength\leftmargin{2.45em}%
	\setlength\labelwidth{2.05em}%
	\setlength{\labelsep}{0.4em}%
	\usecounter{enumi}%
	}%
	}%
{\end{list}
}

{\end{list}
}

{\end{list}
}

\renewenvironment{enumerate}{
\begin{myenumerate}}%
{\end{myenumerate}}

\newenvironment{myitemize}{%
\begin{list}{$\bullet$}%
 	{%
	\setlength{\itemsep}{0.4em}%
	\setlength{\topsep}{0.5em}%
	\setlength\leftmargin{2.65em}%
	\setlength\labelwidth{2.65em}%
	\setlength{\labelsep}{0.4em}%
	}%
	}%
{\end{list}}

\renewenvironment{itemize}{
\begin{myitemize}}%
{\end{myitemize}}

\MakePerPage[2]{footnote} 

\newcommand{\rme}{\mathrm{e}}

\newcommand{\sff}{\mathsf{f}}
\newcommand{\sfg}{\mathsf{g}}
\newcommand{\sfq}{\mathsf{q}}
\newcommand{\sfA}{\mathsf{A}}
\newcommand{\sfC}{\mathsf{C}}
\newcommand{\sfQ}{\mathsf{Q}}
\newcommand{\sfU}{\mathsf{U}}
\newcommand{\sfc}{\mathsf{c}}
\newcommand{\scrC}{\mathscr{C}}
\newcommand{\bulk}{\mathrm{bulk}}
\newcommand{\pairs}{\mathrm{pairs}}
\newcommand{\others}{\mathrm{others}}

\title{Quasi-critical fluctuations for 2d directed polymers}


\date{\today}

\author[F. Caravenna]{Francesco Caravenna}

\address{Dipartimento di Matematica e Applicazioni\\
 Universit\`a degli Studi di Milano-Bicocca\\
 via Cozzi 55, 20125 Milano, Italy}

\email{francesco.caravenna@unimib.it}

\author[F. Cottini]{Francesca Cottini}

\address{Département Mathématiques\\
Université du Luxembourg\\
Maison du Nombre\\
6, Avenue de la Fonte\\
L-4364 Esch-sur-Alzette, Luxembourg}

\email{francesca.cottini@uni.lu}

\author[M. Rossi]{Maurizia Rossi}

\address{Dipartimento di Matematica e Applicazioni\\
 Universit\`a degli Studi di Milano-Bicocca\\
 via Cozzi 55, 20125 Milano, Italy}

\email{maurizia.rossi@unimib.it}

\begin{document}

\maketitle

\begin{abstract}
We study the 2d directed polymer in random
environment in a novel \emph{quasi-critical regime}, 
which interpolates between the much studied 
sub-critical and critical regimes. We prove
Edwards-Wilkinson fluctuations throughout the quasi-critical regime, showing that
the diffusively rescaled partition functions are asymptotically Gaussian.
We deduce a corresponding result
 for the critical 2d Stochastic Heat Flow.
A key challenge is the lack of hypercontractivity, which 
we overcome deriving new moment estimates.
\end{abstract}

\section{Introduction}

We study the 2d directed polymer in random environment,
a key model in statistical mechanics which has been the object of deep
mathematical investigation (see the recent monograph \cite{C17}).
More specifically, we focus on the \emph{partition functions} and their scaling limits, which
have close links to singular stochastic PDEs,
such as the Stochastic Heat Equation and the KPZ equation,
as we discuss in Subsection~\ref{sec:SPDEs}.

The partition functions of the 2d directed polymer in random environment are defined by
\begin{equation} \label{eq:Z}
	Z_{N,\beta}^{\omega}(z) := \E \big[ \rme^{\sum_{n=1}^{N} 
	\{ \beta  \omega(n,S_n)-\lambda(\beta)\}} \big| \, S_0 = z\big]\,,
\end{equation}
where $N\in\N$ is the system size, $\beta \ge 0$ is the disorder strength,
$z\in\Z^2$ is the starting point, and we have two independent sources of randomness:
\begin{itemize}
	\item $S=(S_n)_{n \ge 0}$ is the simple random walk on $\Z^2$
	with law $\P$ and expectation $\E$;
	\item $\omega = (\omega(n,z))_{n \in \N , \, z \in \Z^2}$ are i.i.d.\ random variables
	with law $\bbP$,
	independent of $S$, with
\begin{equation}\label{eq:omega}
		\bbE [\omega]=0\,, \qquad \bbE [\omega^2] = 1\,, 
		\qquad \lambda(\beta):= \log \bbE [\rme^{\beta \omega}] < \infty \quad \text{for }\beta>0\,.
\end{equation}
\end{itemize}
The factor $\lambda(\beta)$ in \eqref{eq:Z}
has the effect to normalise the expectation:
\begin{equation}\label{eq:Zmean}
	\bbE \big[ Z_{N,\beta}^{\omega}(z) \big] = 1 \,.
\end{equation}
Note that $(Z_{N,\beta}^{\omega}(z))_{z\in\Z^2}$ is a family of (correlated)
positive random variables, depending on the random variables~$\omega$
which play the role of \emph{disorder} (or \emph{random environment}).

\smallskip

In this paper we investigate 
the \emph{diffusively rescaled} partition functions
$Z_{N,\beta}^{\omega}(\lfloor \sqrt{N} x \rfloor)$, 
where $\lfloor \cdot \rfloor$ denotes
the integer part. For an integrable test function 
$\varphi : \R^2 \to \R$ we set
\begin{equation}\label{eq:ZNav}
	Z_{N,\beta}^\omega(\varphi) := 
	\int_{\R^2} Z_{N,\beta}^{\omega}(\lfloor \sqrt{N} x \rfloor) \, \varphi(x) \, \dd x
	= \frac{1 }{N} \sum_{z \in \Z^2} 
	Z_{N,\beta}^{\omega}(z)  \, \varphi_N (z) \,,
\end{equation}
where for $R>0$ we define $\varphi_R: \Z^2 \to \R$ by
\begin{equation}\label{eq:phiN}
	\varphi_R(z) := \int\limits_{[z_1, z_1 + 1) \times [z_2, z_2 + 1)} 
	\varphi \big ( \tfrac{y}{\sqrt R}\big ) \, \dd y 
	\qquad \text{for } z = (z_1, z_2) \in \Z^2 \,.
\end{equation}
(note that $\varphi_R(z) \approx \varphi \big ( \tfrac{z}{\sqrt R}\big )$ if $\varphi$ is continuous).
We look for the convergence in distribution of $Z_{N,\beta}^\omega(\varphi)$ as $N\to\infty$,
under an appropriate rescaling of the disorder strength $\beta = \beta_N$.

\subsubsection*{Notation}

We denote by $\varphi \in C_c(\R^2)$ the space of
functions $\varphi: \R^2 \to \R$ that are continuous and compactly supported.
We write $a_N \ll b_N$, $a_N \sim b_N$, $a_N \gg b_N$ to mean that 
the ratio $a_N/b_N$ converges respectively to $0, 1, \infty$
as $N\to\infty$.

\subsection{The phase transition}
It is known since \cite{CSZ17b} that the 
partition functions 
undergo a \emph{phase transition} on the scale $\beta^2 = \beta_N^2 = O(\frac{1}{\log N})$,
that we now recall.

Let $R_N$ be the \emph{expected replica overlap} of two independent
simple random walks $S,S'$:
\begin{equation}\label{eq:RN}
	R_N:= \E^{\otimes 2} \bigg[ \sum_{n=1}^N \ind_{\{S_n=S'_n\}} \bigg]
	=\sum_{n=1}^N \P (S_{2n}=0 ) 
	= \frac{\log N}{\pi}+O(1) \,,
\end{equation}
see the local limit theorem \eqref{eq:llt}.
Using the more convenient parameter
\begin{equation} \label{eq:sigma}
	\sigma_\beta^2 := \bbvar[\rme^{\beta\omega-\lambda(\beta)}] 
	= \rme^{\lambda(2\beta) - 2\lambda(\beta)}-1
\end{equation}
(note that $\sigma_\beta \sim \beta$ as $\beta \downarrow 0$,
since $\lambda(\beta) \sim \frac{1}{2}\beta^2$),
we can rescale $\beta = \beta_N$ as follows:
\begin{equation}\label{eq:hatbeta}
	\sigma_{\beta_N}^2 = \frac{\hat{\beta}^2}{R_N} \sim \frac{\hat{\beta}^2 \, \pi}{\log N}\,, 
	\qquad \text{with} \quad \hat\beta \in (0,\infty) \,.
\end{equation}
Let us recall some key results on the scaling limit of $Z_{N,\beta}^\omega(\varphi)$
from \eqref{eq:ZNav} for $\beta = \beta_N$.

\begin{itemize}
\item In the \emph{sub-critical regime} $\hat\beta < 1$,
after centering and rescaling by $\sqrt{\log N}$, the averaged partition function
$Z_{N,\beta_N}^\omega(\varphi)$ is asymptotically Gaussian, see
\cite{CSZ17b}:\footnote{The result proved 
in \cite[Theorem~2.13]{CSZ17b} actually involves a space-time average, but the same result
for the space average as in \eqref{eq:ZNav} follows by similar arguments, see \cite{CSZ20}.}
\begin{equation}\label{eq:EWsubcritical}
	\hat\beta \in (0,1): \qquad
		\sqrt{\log N}\, \big\{ Z_{N,\beta_N}^\omega(\varphi) - \bbE[Z_{N,\beta_N}^\omega(\varphi)] \big\}
		\xrightarrow[N\to \infty]{d} \, \mathcal{N} \big( 0 \,, 
		v_{\varphi, \, \hat{\beta}}\big)\,,
\end{equation}
for an explicit limiting variance 
$v_{\varphi, \, \hat{\beta}} \in (0, \infty)$
(which \emph{diverges} as $\hat\beta \uparrow 1$).

\item 
In the \emph{critical regime} $\hat\beta = 1$, actually
in the \emph{critical window} $\hat\beta^2 = 1 + \frac{\theta + o(1)}{\log N}$
with $\theta\in\R$,
the averaged partition function $Z_{N,\beta_N}^\omega(\varphi)$ 
is asymptotically \emph{non Gaussian}, see \cite{CSZ23}:
\begin{equation}\label{eq:EWcritical}
	\hat\beta = 1 + \tfrac{\theta + o(1)}{\log N}: \qquad
	Z_{N,\beta_N}^\omega(\varphi) \xrightarrow[N\to \infty]{d} \, 
	\mathscr{Z}^{\theta}(\varphi) = \int_{\R^2} \varphi(x) \, \mathscr{Z}^{\theta}(\dd x) \,,
\end{equation}
where $\mathscr{Z}^{\theta}(\dd x)$ is a non-trivial random measure on $\R^2$
called the Stochastic Heat Flow.
\end{itemize}

Note that the sub-critical convergence \eqref{eq:EWsubcritical} involves a rescaling factor
$\sqrt{\log N}$, while \emph{no rescaling
is needed for the critical convergence \eqref{eq:EWcritical}}.
In view of this discrepancy, it is natural to investigate the transition
between these regimes.

\subsection{Main result}
To interpolate between the sub-critical regime $\hat{\beta} < 1$
and the critical regime $\hat\beta = 1$,
we consider a \emph{quasi-critical regime} in which
\emph{$\hat\beta \uparrow 1$ but
slower than the critical window 
$\hat\beta^2 = 1 + O(\frac{1}{\log N})$}.
Recalling \eqref{eq:RN} and \eqref{eq:hatbeta}, we fix $\beta =\beta_N$ such that
\begin{equation}\label{eq:quasicrit}
	\sigma_{\beta_N}^2 = \frac{1}{R_N} \bigg( 1 - \frac{\theta_N}{\log N}\bigg)
	\qquad \text{for some} \quad 1 \ll \theta_N \ll \log N  \,.
\end{equation}
(Note that $\theta_N = O(1)$ would correspond to the critical window,
while $\theta_N = (1-\hat\beta^2)\log N$ with $\hat\beta \in (0,1)$ 
would correspond to the sub-critical regime.)

Our main result shows that the
averaged partition function $Z_{N,\beta_N}^\omega(\varphi)$ 
has Gaussian fluctuations \emph{throughout the quasi-critical regime \eqref{eq:quasicrit}},
after centering and rescaling by the factor $\sqrt{\theta_N}$,
whose rate of divergence can be arbitrarily slow.
This shows that \emph{non-Gaussian behavior does not appear before the critical regime}.
We call this result \emph{Edwards-Wilkinson fluctuations} in view of its link with stochastic
PDEs, that we discuss in Subsection~\ref{sec:SPDEs}.

\begin{theorem}[Quasi-critical Edwards-Wilkinson fluctuations]\label{th:EWquasicrit}
Let  $Z_{N,\beta}^\omega(\varphi)$ denote the diffusively rescaled and
averaged partition function of the 2d directed polymer model,
see \eqref{eq:Z} and \eqref{eq:ZNav}, for disorder variables $\omega$ which satisfy \eqref{eq:omega}.
Then, for $(\beta_N)_{N\in\N}$ in the quasi-critical regime, see \eqref{eq:sigma}
and \eqref{eq:quasicrit},
we have the convergence in distribution
\begin{equation}\label{eq:main}
	\forall \varphi \in C_c(\R^2): \qquad
	\sqrt{\theta_N} \,
	\big\{ Z_{N,\beta_N}^\omega(\varphi) - \bbE[Z_{N,\beta_N}^\omega(\varphi)] \big\}
	\xrightarrow[N\to \infty]{d} \, \mathcal{N} \big( 0 \,, v_{\varphi}\big)\,,
\end{equation}
where the limiting variance is given by
\begin{equation} \label{eq:sigmaphi}
	v_{\varphi} := 
	\int\limits_{\R^2 \times \R^2} \varphi(x) \, K(x,x') \, \varphi(x') \, \mathrm{d}x \, \mathrm{d}x'
	\qquad \text{with} \quad
	K(x,x') := \int_{0}^{1} \frac{1}{2u} \, \rme^{-\frac{|x-x'|^2}{2u}} \, \mathrm{d}u \,.
\end{equation}
\end{theorem}

The proof is given in Section~\ref{sec:proofmainth}.
An interesting feature of the quasi-critical regime \eqref{eq:sigma} is that it can be
used to approximate the Stochastic Heat Flow $\mathscr{Z}^{\theta}(\dd x)$ as $\theta \to -\infty$, 
see \eqref{eq:EWcritical}. As a consequence, we can
transfer our main result \eqref{eq:main} to the Stochastic Heat Flow, proving the following version of
Edwards-Wilkinson fluctuations as $\theta \to -\infty$.

\begin{theorem}[Edwards-Wilkinson fluctuations for the SHF]\label{th:EW-SHF}
Denoting by $\mathscr{Z}^{\theta}(\dd x)$ the Stochastic Heat Flow in \eqref{eq:EWcritical},
as $\theta\to-\infty$ we have the convergence in distribution
\begin{equation}\label{eq:main-SHF}
	\forall \varphi \in C_c(\R^2): \qquad
	\sqrt{|\theta|} \,
	\big\{ \mathscr{Z}^{\theta}(\varphi) - \bbE[\mathscr{Z}^{\theta}(\varphi)] \big\}
	\xrightarrow[\theta\to -\infty]{d} \, \mathcal{N} \big( 0 \,, v_{\varphi}\big)\,,
\end{equation}
where the limiting variance $v_{\varphi}$ is the same as in \eqref{eq:sigmaphi}.
\end{theorem}

In the rest of the introduction, we first describe the strategy of the proof of Theorem~\ref{th:EWquasicrit}
and we compare it with the literature,  notably with
the proof of the corresponding result \eqref{eq:EWsubcritical} in the sub-critical regime,
pointing out the novel challenges that we need to face. 
We then discuss the connection of our main result \eqref{eq:main}
with stochastic PDEs, in the framework of so-called Edwards-Wilkinson fluctuations,
highlighting future perspectives.

\subsection{Strategy of the proof and comparison with the literature}

We prove Theorem~\ref{th:EWquasicrit} by a Central Limit Theorem under a Lyapunov condition
(see Section~\ref{sec:proofmainth} for a detailed description),
which is close in spirit to the proof of \eqref{eq:EWsubcritical}
 in \cite{CC22} for the sub-critical regime.
On the other hand, the original proof of \eqref{eq:EWsubcritical}
in \cite{CSZ17b} exploited the \emph{Fourth Moment Theorem},
by analysing each term in the \emph{polynomial chaos expansion}
 of $Z_{N,\beta_N}^\omega(\varphi)$
(see Subsection~\ref{sec:poly}) 
and checking that second and fourth moments match the ones of a Gaussian.

Both the approaches in \cite{CC22,CSZ17b} require that
the main contribution to the variance comes from \emph{chaos of bounded order},
i.e.\ the tail of the chaos expansion must be small in~$L^2$
(c.f.\ hypotesis (d) in \cite[Theorem~6.3.1]{NP12} for the Fourth Moment Theorem).
This holds in the sub-critical regime $\hat\beta < 1$ but,
crucially, \emph{it fails in the quasi-critical regime \eqref{eq:quasicrit}
that we consider,
where each fixed order chaos has variance converging to zero}.
The tail of the chaos expansion thus gives the \emph{main} contribution 
to the variance in the quasi-critical regime,
which is one of the main technical challenges we face in this paper.

\smallskip

In our proof of Theorem~\ref{th:EWquasicrit}, we will need to bound moments of 
the partition function $Z_{N,\beta_N}^\omega(\varphi)$ of order higher than two
(to verify a Lyapunov condition). In the sub-critical regime,
such bounds can be obtained exploiting the hypercontractivity of polynomial chaos expansions,
as in \cite{CC22}. However, this property fails in the quasi-critical regime \eqref{eq:quasicrit}
for the same reason pointed out above, namely the tail of the chaos expansion is non negligible.

For this reason, we derive \emph{novel quantitative estimates}
on high moments of the partition function, see Sections~\ref{sec:moment3}
and~\ref{sec:4th-3}, extending the strategy developed in \cite{GQT21,CSZ23, LZ21+}.
We believe that these estimates will find several applications in future research.

\begin{remark}
An alternative approach to bounding moments of the partition function
was developed in \cite{CZ23} based on estimating the collision local time of multiple
independent random walks. This approach yields estimates on
\emph{very high moments}, whose order diverges as $N\to\infty$,
but they are restricted to (a strict subset of) the sub-critical
regime $\hat\beta < 1$, hence they do not cover the quasi-critical regime that we consider.
We also point out the recent paper \cite{LZ24+}, where bounds on very high moments are obtained
in the critical regime \eqref{eq:EWcritical}.
\end{remark}

Let us finally comment on the scaling factor $\sqrt{\theta_N}$ in our main result \eqref{eq:main}.
This can be determined by a variance computation: we show in 
Proposition~\ref{prop:approximationZbysum}, see \eqref{eq:approxbysum0-1}, that as $N\to\infty$
\begin{equation} \label{eq:var-comp}
	\bbvar\big[ Z_{N,\beta_N}^\omega(\varphi) \big] \sim \frac{v_\varphi^2}{\theta_N} \,,
\end{equation}
with $v_\varphi$ as in \eqref{eq:sigmaphi}. We can explain heuristically
the scaling in \eqref{eq:var-comp} as follows.
Due to the averaging on the diffusive scale $\sqrt{N}$ determined by
$\varphi_N(\cdot)$ in \eqref{eq:ZNav}, the
variance of $Z_{N,\beta_N}^\omega(\varphi)$ is essentially determined by $\bbcov[Z_{N,\beta_N}^\omega(x), Z_{N,\beta_N}^\omega(y)]$ for $|x-y| \approx \sqrt{N}$.
Such a covariance is approximately given by the product of \emph{three factors}
(see \eqref{eq:ub+} below):
\begin{itemize}
\item \emph{the expected number of times two independent random walks meet before time~$N$}
starting from $x$ and $y$ (see the term in brackets in \eqref{eq:ub+}), which is of order~$1$;
\item \emph{the factor $\sigma_{\beta_N}^2 \sim 1/\log N$} arising from the variance of
$\rme^{\beta \omega - \lambda(\beta)}$, see \eqref{eq:sigma};
\item \emph{the second moment of the partition function $Z_{N,\beta_N}^\omega(z)$ 
from  a single point} $(1-\sigma_{\beta_N}^2R_N)^{-1}$
(see the last fraction in \eqref{eq:ub+}), which is of order $\log_N/\theta_N$, see \eqref{eq:quasicrit}.
\end{itemize}
Combining these factors, we obtain $\bbvar[ Z_{N,\beta_N}^\omega(\varphi) ]
\approx 1/\theta_N$ in agreement with \eqref{eq:var-comp}.

\subsection{Relevant context and future perspectives}
\label{sec:SPDEs}

The Gaussian fluctuations for  $Z^\omega_{N,\beta}(\varphi)$
in Theorem~\ref{th:EWquasicrit}
are closely connected to a stochastic PDE, 
the \emph{Edwards-Wilkinson equation},
also known as Stochastic Heat Equation with \emph{additive} noise:
\begin{equation}
	\label{eq:EWeq}
	\partial_t v^{(\mathsf{s},\mathsf{c})}(t,x) \, = \, \frac{\mathsf{s}}{2}\, \Delta_xv^{(\mathsf{s},\mathsf{c})}(t,x) 
	\, + \, \mathsf{c} \, \dot{W}(t,x)\,,
\end{equation}
where $\mathsf{s}, \mathsf{c} >0$ are fixed parameters and
$\dot{W}(t,x)$ is space-time white noise.
This equation is well-posed in any spatial dimension $d \ge 1$:
its solution is the Gaussian process
\begin{equation*}
	v^{(\mathsf{s},\mathsf{c})}(t,x) \, = \, v^{(\mathsf{s},\mathsf{c})}(0,x) \,+\,
	\mathsf{c} \, \int_{0}^t \int_{\R^d} g_{\mathsf{s}(t-u)}(x-z) \, \dot{W}(u,z) \, \dd u \, \dd z\,, 
\end{equation*}
where $g_t(x) := (2 \pi t)^{-d/2} \, \rme^{-\frac{|x|^2}{2t}}$ is the heat kernel on $\R^d$.
It is known that $x \mapsto v^{(\mathsf{s},\mathsf{c})}(t,x)$ is a (random) function only for $d=1$,
while for $d \ge 2$ it is a genuine distribution.

Henceforth we focus on $d=2$.
The solution $v^{(\mathsf{s},\mathsf{c})}(t,\cdot)$ 
with initial condition $v^{(\mathsf{s},\mathsf{c})}(0,\cdot) \equiv 0$,
averaged on test functions $\varphi \in C_c(\R^2)$,
is the centered Gaussian process with covariance
\begin{equation*}
	\bbE \big[ v^{(\mathsf{s},\mathsf{c})} (t,\varphi)  \, v^{(\mathsf{s},\mathsf{c})}(t, \psi) \big] 
	= \int_{\R^2\times\R^2} \varphi(x) \, K^{(\mathsf{s},\mathsf{c})}_{t}(x,y) \, \psi(y) \, 
	\dd x\, \dd y \,,
\end{equation*}
where we set
\begin{equation}\label{eq:covEW}
	K^{(\mathsf{s},\mathsf{c})}_{t}(x,y) := \mathsf{c}^2 \int_0^t g_{2\mathsf{s}u}(x-y) \, \dd u =
	\frac{\mathsf{c}^2}{2\mathsf{s}} \int_{0}^{2\mathsf{s}t} 
	\frac{1}{2 \pi u} \rme^{-\frac{|x-y|^2}{2u}} \, \dd u\,.
\end{equation} 
Comparing with \eqref{eq:sigmaphi}, we can rephrase our main result 
\eqref{eq:main}: 
for any $\varphi \in C_c(\R^2)$
\begin{equation} \label{eq:EW}
	\sqrt{\theta_N} \,
	\big\{ Z_{N,\beta_N}^\omega(\varphi) - \bbE[Z_{N,\beta_N}^\omega(\varphi)] \big\}
	\xrightarrow[N\to \infty]{d} \, 
	v^{(\mathsf{s},\mathsf{c})}(1,\varphi)
	\qquad \text{with }
	\begin{cases}
	\mathsf{s} =\frac{1}{2} \,,\\
	\mathsf{c} = \sqrt{\pi} \,.
	\end{cases}
\end{equation}
In other term, \emph{the diffusively rescaled partition functions in the quasi-critical regime
converge, after centering and rescaling, to the solution
of the Edwards-Wilkinson equation}.

\begin{remark}\label{rem:EWsubcritical}
Also relation \eqref{eq:EWsubcritical},
in the sub-critical regime $\hat\beta \in (0,1)$, can be rephrased as a convergence
to the Edwards-Wilkinson solution $v^{(\mathsf{s},\hat{\mathsf{c}})}(1,\varphi)$
with $\hat{\mathsf{c}} = \sqrt{\pi} \, \hat{\beta}/\sqrt{1-\hat{\beta}^2}$.
\end{remark}

The reason why stochastic PDEs emerge naturally in the study of directed polymers is that,
by the Markov property of simple random walk, the diffusively rescaled 
partition function $u_N(t,x) := Z_{\lfloor Nt \rfloor,\beta}^\omega(\lfloor \sqrt{N} x \rfloor)$ solves 
(up to a time reversal) a \emph{discretized version} of the Stochastic Heat Equation 
with \emph{multiplicative noise}:
\begin{equation}\label{eq:SHEm}
	\partial_t u(t,x) \, = \, \frac{1}{2} \Delta_x u(t,x) \, + \, \beta\, \dot{W}(t,x) \,u(t,x)\,,
\end{equation}
with initial condition $u(0,x)=1$.
This gives a hint how the Edwards-Wilkinson equation \eqref{eq:EWeq}
may arise in the scaling limit of directed polymer partition functions: intuitively,
the singular product $\dot{W}(t,x) \,u(t,x)$ in \eqref{eq:SHEm}
for $u(t,x) = u_N(t,x)$ converges to an independent white noise as $N\to\infty$
(see \cite[Theorem~3.4]{CC22} in the sub-critical regime).

Edwards-Wilkinson fluctuations were recently proved also
for a \emph{non-linear} Stochastic Heat Equation,
see \cite{DG22,T22+}, always in the sub-critical regime. It would be interesting to extend
these results in the quasi-critical regime, generalizing our Theorem~\ref{th:EWquasicrit}.

\begin{remark}
The multiplicative Stochastic Heat Equation \eqref{eq:SHEm} 
in the continuum is well-posed
in one space dimension $d=1$, e.g.\ by classical Ito-Walsh stochastic integration, 
but \emph{it is ill-defined in higher dimensions $d \ge 2$}. 
For this reason, directed polymer partition functions can
provide precious insight on the equation \eqref{eq:SHEm}. In particular, for $d=2$,
their scaling limit in the critical regime was obtained in \cite{CSZ23}
and called the critical 2d Stochastic Heat Flow, see \eqref{eq:EWcritical}, 
as a natural
candidate for the ill-defined solution of \eqref{eq:SHEm}.
\end{remark}

In the same spirit, the log-partition function 
$h_N(t,x) := \log Z^\omega_{\lfloor Nt \rfloor,\beta}(\lfloor \sqrt{N} x \rfloor)$ provides a
discretized approximation for the \emph{Kardar-Parisi-Zhang (KPZ) equation} \cite{KPZ86}:
\begin{equation*}
	\partial_th(t,x) \, = \, \frac{1}{2} \Delta_x h(t,x) \, + \, \frac{1}{2} | \nabla h(t,x)|^2 \, 
	+ \, \beta \, \dot{W}(t,x) \,,
\end{equation*}
with initial condition $h(0,x)=0$.
This equation too, in the continuum,
is only fully understood in one space-dimension $d=1$,
via recent breakthrough techniques of regularity structures \cite{H14}
or paracontrolled distributions \cite{GIP15,GP17};
see also \cite{GJ14,K16}.
Similar to \eqref{eq:EWsubcritical},
Edwards-Wilkinson fluctuations
have been proved for $h_N(t,x)$  in the entire sub-critical regime \eqref{eq:hatbeta}
with $\hat{\beta} \in (0,1)$ \cite{CSZ20,G20,CD20}: for $\varphi \in C_c(\R^2)$
\begin{equation}\label{eq:EWsublogZ}
	\sqrt{\log N}\, \big\{ \log Z_{N,\beta_N}^\omega(\varphi) - 
	\bbE[\log Z_{N,\beta_N}^\omega(\varphi)] \big\}
	\xrightarrow[N\to \infty]{d} \, v^{(\mathsf{s},\hat{\mathsf{c}})}(1,\varphi) \,,
\end{equation} 
with $\mathsf{s}$, $\hat\sfc$ as in Remark~\ref{rem:EWsubcritical}.
This  was recently extended in \cite{NN23}, which focuses on a mollification
(rather than discretization) of
the Stochastic Heat Equation \eqref{eq:SHEm}: phrased
in our setting, the results of \cite{NN23} prove
Gaussian fluctuations in the sub-critical regime
for general transformations $F(Z_{N,\beta_N}^\omega)$,
besides $F(z) = \log z$, 
with general initial conditions.

\smallskip

It would be very interesting to extend \eqref{eq:EWsublogZ} to the quasi-critical regime
\eqref{eq:quasicrit},
namely to prove an analogue of our Theorem~\ref{eq:EWsublogZ} for 
$\log Z_{N,\beta_N}^\omega(\varphi)$,
which we expect to hold. A natural strategy would be to
generalize the linearization procedure established in \cite{CSZ20} 
to handle the logarithm. This requires estimating \emph{negative moments}
of the partition function, which is
a challenge in the quasi-critical regime (since $Z_{N,\beta_N}^\omega(z) \to 0$
for fixed~$z\in\Z^2$).

\smallskip

Local averages on \emph{sub-diffusive scales}
have also been investigated for the mollified KPZ solution in the sub-critical regime, see \cite{C23,T23+}.
Similar results can be expected for the mollified solution of the Stochastic Heat Equation
\eqref{eq:SHEm}, or for the directed polymer partition function, which
should be obtainable in the sub-critical regime 
as in \cite{CSZ17b}.
It would be natural to study
such local averages also in the quasi-critical regime.

\smallskip

We finally mention that
Edwards-Wilkinson fluctuations like \eqref{eq:EWsubcritical} 
and \eqref{eq:EWsublogZ} have also been obtained in higher dimensions $d \ge 3$,
in the so-called \emph{$L^2$-weak disorder phase}
where the partition function has bounded second moment
\cite{CN21, LZ22, CNN22,CCM21+}, see also the previous works 
\cite{MU18, GRZ18, CCM20, DGRZ20}. 
Unlike the two-dimensional setting, for $d \ge 3$ the partition
function admits a non-zero limit also \emph{beyond the $L^2$-weak disorder phase}:
see \cite{J22, J22+} for recent results in this challenging regime.
It would be natural to investigate whether our approach can also be applied
in higher dimensions $d \ge 3$, in order to prove Gaussian fluctuations
\emph{slightly beyond} the $L^2$-weak disorder phase.

\subsection{Organization of the paper}

The paper is structured as follows.
\begin{itemize}
\item In Section~\ref{sec:proofmainth} we present the structure of the proof of 
Theorem~\ref{th:EWquasicrit} based on two key
steps, formulated as Propositions~\ref{prop:approximationZbysum}
and~\ref{prop:stimamomento}, and we prove Theorem~\ref{th:EW-SHF}.

\item In Section~\ref{sec:moment2} we prove 
Proposition~\ref{prop:approximationZbysum}.

\item In Section~\ref{sec:moment3} we derive upper bounds on the moments
of the partition functions.

\item In Section~\ref{sec:4th-3} we prove Proposition~\ref{prop:stimamomento}.

\item Finally, some technical points are deferred to Appendix~\ref{sec:4thapp}.
\end{itemize}

\subsection*{Acknowledgements}
We gratefully acknowledge the support of INdAM/GNAMPA.

\section{Proof of Theorems~\ref{th:EWquasicrit}
and~\ref{th:EW-SHF}}\label{sec:proofmainth}

Let us call $X_N$ the LHS of \eqref{eq:main}: recalling \eqref{eq:ZNav} and \eqref{eq:Zmean}, we can write
\begin{equation}\label{eq:X}
\begin{split}
	X_N &:= \sqrt{\theta_N} \,
	\big\{ Z_{N,\beta_N}^\omega(\varphi) - \bbE[Z_{N,\beta_N}^\omega(\varphi)] \big\} \\
	&= \frac{\sqrt{\theta_N} }{N} \sum_{z \in \Z^2} 
	\big\{ Z_{N,\beta_N}^{\omega}(z) - 1 \big\}  \, \varphi_N(z) \,,
\end{split}
\end{equation}
with $\varphi_N$ as in \eqref{eq:phiN}.
In this section, we prove Theorem~\ref{th:EWquasicrit} via the following two main steps:
\begin{enumerate}
	\item we first approximate $X_N$ in $L^2$
	by a sum $\sum_{i=1}^M X_{N,M}^{(i)}$ of
	\emph{independent} random variables,
	for $M=M_N \to \infty$ slowly enough;
	\item we then show that the random variables $(X_{N,M}^{(i)})_{1 \le i \le M}$ 
	for $M=M_N$
	satisfy the assumptions of the classical \emph{Central Limit Theorem} for triangular arrays.
\end{enumerate}

\subsection{First step}

In order to define the random variables $X_{N,M}^{(i)}$, 
for $M\in\N$ and $1 \le i \le M$,
we introduce a variation
of \eqref{eq:Z}, for $-\infty < A < B < \infty$:
\begin{equation} \label{eq:ZAB}
	Z_{(A,B],\beta}^{\omega}(z):= \E \big[ \rme^{\sum_{n \in (A,B] \cap \N}
	\{ \beta  \omega(n,S_n)-\lambda(\beta)\}} \big| \, S_0 = z\big]\,.
\end{equation}
We then define $X_{N,M}^{(i)}$ 
replacing $Z_{N,\beta}^{\omega}$ by $Z_{(\frac{i-1}{M}N,\frac{i}{M}N],\beta}^{\omega}$
in the definition \eqref{eq:X} of $X_N$:
\begin{equation}\label{eq:Xi}
\begin{split}
	X_{N,M}^{(i)} 
	&= \frac{\sqrt{\theta_N} }{N} \sum_{z \in \Z^2} 
	\big\{ Z_{(\frac{i-1}{M}N,\frac{i}{M}N],\beta_N}^{\omega}(z) - 1 \big\}  \, 
	\varphi_N (z) \,.
\end{split}
\end{equation}
Note that $Z_{(A,B],\beta}^{\omega}(z)$ only depends on $\omega(n,x)$ for $A < n \le B$,
moreover $\bbE[Z_{(A,B],\beta}^{\omega}(z)]=1$. 
As a consequence,
$X_{N,M}^{(i)}$
for $1 \le i \le M$ are \emph{independent} and \emph{centered} random variables.

\smallskip

The core of this first step is the following approximation result, proved in Section~\ref{sec:moment2}.

\begin{proposition}[$L^2$ approximation]\label{prop:approximationZbysum}
For $(\beta_N)_{N\in\N}$ in the quasi-critical regime, see \eqref{eq:sigma} and \eqref{eq:quasicrit},
the following relations hold for any $\varphi \in C_c(\R^2)$, 
with $v_{\varphi}$ as in \eqref{eq:sigmaphi}:
\begin{gather}\label{eq:approxbysum0-1}
	\lim_{N\to\infty} \, \bbE\big[ X_N^2 \big] = v_{\varphi}\,, \\ 
	\label{eq:approxbysum0-2}
	\forall M\in\N: \quad \lim_{N\to \infty} \, \bigg\| \, X_N \, - \, 
	\sum_{i=1}^{M} X_{N,M}^{(i)} \, \bigg\|_{L^2} = 0 \,.
\end{gather}
\end{proposition}

By general arguments, see \cite[Remark~4.2]{CC22},
relation \eqref{eq:approxbysum0-2} still holds if
$M \to \infty$ slowly enough as $N\to\infty$. More precisely, 
there exists a sequence $\overline{M_N} \to \infty$
such that
	\begin{equation}
		\label{eq:approxbysum}
		\lim_{N\to \infty} \, \bigg\| \, X_N \, - \, 
		\sum_{i=1}^{M_N} X_{N,M_N}^{(i)} \, \bigg\|_{L^2} = 0 
		\qquad \text{for any $M_N \le \overline{M_N}$.}
	\end{equation}

{\footnotesize
\begin{proof}[Proof of \eqref{eq:approxbysum}]
If we set $\alpha_{\bar{M},N} := \max_{M \le \bar{M}} \| X_N - \sum_{i=1}^M X^{(i)}_{N,M}\|_{L^2}$,
it follows by \eqref{eq:approxbysum0-2} that for any $\bar{M} \in \N$ we have
$\lim_{N\to\infty} \alpha_{\bar{M},N} = 0$, 
hence we can find $\widehat{N}_{\bar{M}} \in \N$ such that 
$\alpha_{\bar{M},N} \le 1/\bar{M}$ (say)
for $N \ge \widehat{N}_{\bar{M}}$,
and we can take $\bar{M} \mapsto \widehat{N}_{\bar{M}}$ increasing. 
Given $N\in\N$, we call  $\overline{M_N}$
the largest $\bar{M}\in\N$ for which $N \ge \widehat{N}_{\bar{M}}$, 
that is $\overline{M_N}:= \max\{\bar{M} \in \N \colon \widehat{N}_{\bar{M}} \le N\}$. 
This ensures that 
$\alpha_{\overline{M_N},N} \le 1/\overline{M_N}$, hence $\alpha_{\overline{M_N},N} \to 0$ as
$N\to\infty$ because $\overline{M_N} \to \infty$. The
definition of $\alpha_{\bar{M},N}$ then directly implies \eqref{eq:approxbysum}.
\end{proof}
}

\smallskip

Relation \eqref{eq:approxbysum} shows that we can approximate $X_N$ in $L^2$ by a sum 
of independent and centered random variables.
We then obtain, by \eqref{eq:approxbysum0-1},
\begin{equation}\label{variance}
 	\lim_{N\to \infty}  \bbE \Bigg[
	\Bigg( \sum_{i=1}^{M_N} X_{N,M_N}^{(i)} \Bigg)^2 \Bigg] 
	= \lim_{N\to \infty}  
	\sum_{i=1}^{M_N} \bbE \Big[ \big(X_{N,M_N}^{(i)}\big)^2 \Big]
	= v_{\varphi} \,.
\end{equation}

	\begin{remark}\label{rem:CD}
		A decomposition of the partition function
		is employed in the recent paper
		\cite{CD} to give an alternative proof of the asymptotic log-normality 
		of the partition function in the sub-critical regime.
		In our decomposition \eqref{eq:approxbysum0-2}, 
		each individual piece
		$X_{N,M}^{(i)}$ for $i=1,\ldots, M$ contributes on the order of $\frac{1}{M}$ 
		to the total limiting variance $v_\varphi$ (see Lemma~\ref{lem:limitsecmomZ(i)}).
		The same holds for the decomposition in \cite{CD}.
		
		There are, however, key differences: 
		the decomposition in \cite{CD} is \emph{multiplicative}
		whereas ours is \emph{additive}, as seen in \eqref{eq:approxbysum0-2};
		additionally, the decomposition in \cite{CD} is based on the \emph{exponential time scale}
		$N^{\frac{i}{M}}$, while
		ours is defined on the \emph{linear time scale} $\frac{i}{M} N$, 
		reflecting the different limits that are obtained (log-normal vs.\ normal).
		
		We also point out that analogous decompositions ---both in linear and exponential time scales--- 
		had already been used in \cite{CC22}.
	\end{remark}

\subsection{Second step}

Recalling \eqref{eq:X}, we can rephrase our goal \eqref{eq:main} as
$X_N \overset{d}{\to} \cN(0,v_{\varphi})$. In view of \eqref{eq:approxbysum},
this follows if we prove the convergence in distribution
\begin{equation} \label{eq:CLT}
	\sum_{i=1}^{M_N} X_{N,M_N}^{(i)}
	\xrightarrow[N\to \infty]{d} \, \mathcal{N} \big( 0 \,, v_{\varphi}\big)\,.
\end{equation}
Since $(X_{N,M_N}^{(i)})_{1\le i \le M_N}$ are
independent and centered, we apply the classical Central Limit Theorem 
for triangular arrays, see e.g.\ \cite[Theorem 27.3]{Bil95}: 
since we have convergence of the variance by \eqref{variance}, it is enough to check
the Lyapunov condition
\begin{equation}\label{lyapunov}
    \text{for some } p> 2: \qquad 
    \lim_{N\to \infty} \sum_{i=1}^{M_N} \mathbb E \Big[ \big| X_{N, M_N}^{(i)} \big|^p \Big] = 0 \,.
\end{equation}
This follows from the next result, proved in Section~\ref{sec:moment3}, where we focus on the case $p=4$.

\begin{proposition}[Fourth moment bound]\label{prop:stimamomento}
For $(\beta_N)_{N\in\N}$ in the quasi-critical regime, see \eqref{eq:sigma} and \eqref{eq:quasicrit},
and for any $\varphi \in C_c(\R^2)$, there is a constant $C<\infty$ such that
\begin{equation}\label{eq:stimamomento}
	\forall M\in\N \,, \ \ \forall 1\le i\le M: \qquad
	\limsup_{N\to\infty} \bbE \Big[ \big(  X_{N, M}^{(i)} \big)^4 \Big] \le \frac{C}{M^2}  \,.
\end{equation}
\end{proposition}

Since the constant $C$ in  \eqref{eq:stimamomento} does not depend on~$M$,
we can let $M_N \to \infty$ slowly enough and the estimate will still hold if the RHS is doubled, say.
More precisely, there exists a sequence $\overline{M'_N} \to \infty$ such that
\begin{equation} \label{eq:estdouble2}
	\max_{1 \le i \le M_N} \, \bbE \Big[ \big(  X_{N, M_N}^{(i)} \big)^4 \Big] \le \frac{2C}{M_N^2}
	\qquad \text{for any $M_N \le \overline{M'_N}$.}
\end{equation}

{\footnotesize
\begin{proof}[Proof of \eqref{eq:estdouble2}]
If we call $\alpha_{M,N} := \max_{1 \le i \le M} \bbE [ (  X_{N, M}^{(i)} )^4 ]$,
then by \eqref{eq:stimamomento},
for any $M \in \N$, there is $\widehat{N}_M \in \N$ such that
$\alpha_{M,N} \le \frac{2C}{M^2}$ for all $N \ge \widehat{N}_M$. 
We can take $M \mapsto \widehat{N}_M$ increasing,
and setting $\overline{M'_N} := \max\{M \in \N \colon \widehat{N}_M \le N\}$
we see that $M \le \overline{M'_N}$ is the same as $N \ge \widehat{N}_M$,
and $\lim_{N\to\infty} \overline{M'}_N = \infty$.
\end{proof}
}

If we finally take $M_N = \min\{\overline{M_N}, \overline{M'_N}\}$, both estimates
\eqref{eq:approxbysum} and \eqref{eq:estdouble2} hold.
This shows that 
\eqref{lyapunov} holds with $p=4$
(the sum therein is $\le 2C/M_N \to 0$ as $N\to\infty$).

\smallskip

The proof of Theorem~\ref{th:EWquasicrit} is then completed once we 
prove Propositions~\ref{prop:approximationZbysum} and~\ref{prop:stimamomento}.
The next sections are devoted to these tasks.

\subsection{Proof of Theorem~\ref{th:EW-SHF}}
Recalling \eqref{eq:hatbeta}, we define for $\theta \in \R$
and $N\in\N$ the value $\beta_N^{\mathrm{crit}}(\theta)$ such that
\begin{equation*}
	\sigma^2_{\beta_N^{\mathrm{crit}}(\theta)} := \frac{1}{R_N} \big(1 + \frac{\theta}{\log N}\big) \,.
\end{equation*}
Then we can rephrase \eqref{eq:EWcritical} as follows:
\begin{equation}\label{eq:EWcriticalbis}
	\forall \varphi \in C_c(\R^2) \,, \ \forall \theta \in \R: \qquad
	Z_{N,\beta_N^{\mathrm{crit}}(\theta)}^\omega(\varphi) \xrightarrow[N\to \infty]{d} \, 
	\mathscr{Z}^{\theta}(\varphi) \,.
\end{equation}

Let us fix $\varphi \in C_c(\R^2)$
and an artbitrary negative sequence $\theta_k < 0$ such that $\theta_k \to -\infty$.
It is enough to prove \eqref{eq:main-SHF} along $\theta_k$,
that is, for any fixed continuous and bounded $f: \R \to \R$,
\begin{equation}\label{eq:main-SHF2}
	\bbE\Big[ f \Big( \sqrt{|\theta_k|} \,
	\big\{ \mathscr{Z}^{\theta_k}(\varphi) - {\textstyle \int \varphi} \big\}
	\Big) \Big]
	\xrightarrow[k\to\infty]{} \, \bbE\big[ f\big( \mathcal{N} \big( 0 \,, v_{\varphi}\big)
	\big) \big] \,,
\end{equation}
where we have replaced $\bbE[\mathscr{Z}^{\theta}(\varphi)] = \int \varphi$
by properties of the Stochastic Heat Flow, and we also note that
$\bbE[Z_{N,\beta}^\omega(\varphi)] = \sum_{z\in\Z^2} \varphi_N(z) = \int \varphi$ by construction, see 
\eqref{eq:phiN}.

The idea is, for any fixed $k\in\N$, to take $N_k \in \N$ large enough so that,
by \eqref{eq:EWcriticalbis}, we can approximate
$\mathscr{Z}^{\theta_k}(\varphi)$ with $Z_{N_k,\beta_{N_k}^{\mathrm{crit}}(\theta)}^\omega(\varphi)$
in the LHS of \eqref{eq:main-SHF}, more precisely
\begin{equation} \label{eq:step-appr}
	\Big| \bbE\Big[ f \Big( \sqrt{|\theta_k|} \,
	\big\{ \mathscr{Z}^{\theta_k}(\varphi) - {\textstyle \int \varphi} \big\}
	\Big) \Big] - 
	\bbE\Big[ f \Big( \sqrt{|\theta_k|} \,
	\big\{ Z_{N_k,\beta_{N_k}^{\mathrm{crit}}(\theta_k)}^\omega(\varphi) 
	- {\textstyle \int \varphi} \big\}
	\Big) \Big] \Big| \le \frac{1}{k} \,.
\end{equation}
By possibly enlarging $N_k$, we assume that 
$N_k \ge \rme^{k |\theta_k|}$ which ensures 
$|\theta_k| \le \frac{1}{k} \, \log N_k \ll \log N_k$ as $k\to\infty$. 
Writing $\theta_k = -|\theta_k|$ since $\theta_k < 0$, we have
\begin{equation*}
	\beta_{N_k}^{\mathrm{crit}}(\theta_k) =
	\tfrac{1}{R_{N_k}} \big(1 - \tfrac{|\theta_k|}{\log N_k}\big)
	\qquad \text{with} \quad 1 \ll |\theta_k| \ll \log N_k  \,. 
\end{equation*}
This means that $\beta_{N_k}^{\mathrm{crit}}(\theta_k)$ is in the quasi-critical
regime \eqref{eq:quasicrit}, hence we can apply our main result \eqref{eq:main}
and deduce that
\begin{equation*}
	\bbE\Big[ f \Big( \sqrt{|\theta_k|} \,
	\big\{ Z_{N_k,\beta_{N_k}^{\mathrm{crit}}(\theta_k)}^\omega(\varphi) 
	- {\textstyle \int \varphi} \big\}
	\Big)  \Big] \xrightarrow[k\to\infty]{} 
	\bbE\big[ f\big( \mathcal{N} \big( 0 \,, v_{\varphi}\big)
	\big) \big] \,.
\end{equation*}
Recalling \eqref{eq:step-appr}, we obtain our goal \eqref{eq:main-SHF2}.\qed

\section{Second moment bounds: proof of Proposition~\ref{prop:approximationZbysum}}
\label{sec:moment2}

In this section we prove Proposition~\ref{prop:approximationZbysum}
exploiting a polynomial chaos expansion of the partition function.
We fix $(\beta_N)_{N\in\N}$ in the quasi-critical regime, see \eqref{eq:sigma} 
and \eqref{eq:quasicrit}, and $\varphi \in C_c(\R^2)$.
We denote by $C, C', \ldots$ generic constants
that may vary from place to place.

\subsection{Polynomial chaos expansion}
\label{sec:poly}

The partition function admits a key polynomial chaos expansion \cite{CSZ17a}.
Let us define, for $\beta > 0$,
\begin{equation} \label{eq:eta}
	\begin{aligned}
		&\xi_\beta(n,x) := \rme^{\beta\omega(n,x)-\lambda(\beta)}-1\,, 
		\qquad \text{for }  n \in \N \,, \ x \in \Z^2 \,.
	\end{aligned}
\end{equation}
Recalling \eqref{eq:sigma}, we note that $(\xi_\beta(n,x))_{n\in\N, x\in\Z^2}$
are independent random variables with
\begin{equation} \label{eq:meanvarxi}
	\bbE[\xi_\beta] = 0 \,, \qquad \bbE[\xi_\beta^2] = \sigma_\beta^2 \,, \qquad
	\bbE[|\xi_\beta|^k] \le C_k \, \sigma_\beta^k \quad \forall k \ge 2 \,,
\end{equation}
for some $C_k < \infty$
(for the bound on $\bbE[|\xi_\beta|^k]$ see, e.g., \cite[eq. (6.7)]{CSZ17a}).

We denote by $q_{n}(x)$ the random walk transition kernel:
\begin{equation}\label{eq:defofq}
	q_n(x) := \P (S_n=x \,|\, S_0 = 0)\,.
\end{equation}
Then, writing $\rme^{\sum_n \{\beta \omega(n,x)-\lambda(\beta)\}} =
\prod_n (1 + \xi_\beta(n,x))$ and expanding the product,
we can write $Z_{(A,B], \beta}^\omega(z) $ in \eqref{eq:ZAB} as
the following polynomial chaos expansion: 
\begin{equation}\label{eq:partfunc}
\begin{split}
	Z_{(A,B], \beta}^\omega(z) 
	= 1 + \sum_{k=1}^\infty 
	\ \sum_{\substack{A<n_1<\ldots<n_k\le B\\
	x_1,\ldots,x_k \in \Z^2}} \, & q_{n_1}(x_1 - z) \, \xi_\beta(n_1, x_1) \times \\
	& \times \prod_{j=2}^{k} q_{n_j-n_{j-1}}(x_j-x_{j-1}) \, \xi_\beta(n_j,x_j)\,,
\end{split}
\end{equation}
where we agree that the time variables $n_1 < \ldots < n_k$ are summed in the set
$(A, B] \cap \Z$ (in particular, the seemingly infinite sum over $k$ can be stopped at $B-A$).

Plugging \eqref{eq:partfunc} into \eqref{eq:X}, we obtain 
a corresponding polynomial chaos expansion for $X_N$,
recall \eqref{eq:X} and \eqref{eq:phiN}: if we define the averaged random walk transition kernel
\begin{equation}\label{eq:kernel1}
	q_{n}^\varphi(x) := 
	\sum_{z \in \Z^2} q_{n}(x-z) \, \varphi(z) \,,
	\qquad \text{for } \ \varphi: \Z^2 \to \R \,,
\end{equation}
we obtain
\begin{equation} \label{eq:Zexplicit}
		X_N = \frac{\sqrt{\theta_N}}{N} \sum_{k=1}^\infty 
		\, \sum_{\rule{0pt}{1.3em}\substack{0<n_1<\ldots<n_k\le N\\
			 x_1,\ldots,x_k \in \Z^2}} \!\!\!\!\!\!\!
			 q_{n_1}^{\phi_N}(x_1) \, \xi_{\beta_N}(n_1,x_1) 
			 \prod_{j=2}^{k} q_{n_j-n_{j-1}}(x_j-x_{j-1}) \, \xi_{\beta_N}(n_j,x_j)\,.
\end{equation}
The analogous polynomial chaos expansion for the random variables $X_{N,M}^{(i)}$, see \eqref{eq:Xi},
is obtained from \eqref{eq:Zexplicit} restricting the sum
to $\frac{i-1}{M}N<n_1<\ldots<n_k\le \frac{i}{M}N$:
\begin{equation}\label{eq:Z(i)}
	\begin{aligned}
		X_{N,M}^{(i)}
		= \frac{\sqrt{\theta_N}}{N} \, \sum_{k=1}^\infty 
	\ \sum_{\substack{\frac{i-1}{M}N<n_1<\ldots<n_k\le \frac{i}{M}N\\
				x_1,\ldots,x_k \in \Z^2}} \, &
				q_{n_1}^{\phi_N}(x_1) \, \xi_{\beta_N}(n_1,x_1) \, \times \\ 
				&\times \prod_{j=2}^{k} q_{n_j-n_{j-1}}(x_j-x_{j-1}) \, \xi_{\beta_N}(n_j,x_j)\,.
	\end{aligned}
\end{equation}

\begin{remark}\label{rem:orthogonal}
Since the random variables $(\xi_\beta(n,x))_{n \in \N, x \in \Z^2}$ are independent
and centered, see \eqref{eq:eta}, \emph{the terms in the polynomial chaos
\eqref{eq:partfunc}, \eqref{eq:Zexplicit}, \eqref{eq:Z(i)} are orthogonal in $L^2$}.
\end{remark}

We finally recall the local limit theorem for the simple random walk on $\Z^2$, see
\cite[Theorem~2.1.3]{LL10}: as $n \to \infty$, uniformly for $x \in \Z^2$
we have\footnote{The scaling factor in \eqref{eq:llt} is $n/2$ because the covariance matrix of 
the simple random walk on $\Z^2$ is $\frac{1}{2}I$, 
while the factor $2 \ind_{(m,z) \in \Z^3_{\text{even}}}$ is due to periodicity.}
\begin{equation} \label{eq:llt}
	q_n(x) = \frac{1}{n/2} \Big( g\Big(\tfrac{x}{\sqrt{n/2}}\Big) + o(1) \Big) \,
	2 \, \ind_{(n,x) \in \Z^3_{\mathrm{even}}} \,, \qquad \text{where} \qquad
	g(y) := \frac{\rme^{-\frac{1}{2}|y|^2}}{2\pi} ,
\end{equation}
and we set $\Z^3_{\mathrm{even}}:= \big\{ \, y = \big( \, y_1, \, y_2, \, y_3 \, \big)\in \Z^3 
\colon y_1 + y_2 + y_3 \in 2\Z \, \big\}$.

\subsection{Proof of Proposition~\ref{prop:approximationZbysum}}

Note that $\sum_{i=1}^{M} X_{N,M}^{(i)}$ is a polynomial chaos 
where all time variables $n_1 < \ldots < n_k$ belong to 
\emph{one of the intervals} $( \frac{i-1}{M}N \,, \, \frac{i}{M}N ]$, 
see \eqref{eq:Z(i)}.
It follows that $X_N$ is a \emph{larger polynomial chaos}
than $\sum_{i=1}^{M} X_{N,M}^{(i)}$, i.e.\ it contains more terms, 
hence \emph{the difference $X_N - \sum_{i=1}^{M} X_{N,M}^{(i)}$ is orthogonal in $L^2$ 
to $\sum_{i=1}^{M} X_{N,M}^{(i)}$} (see Remark~\ref{rem:orthogonal}):
	\begin{equation*}
		\begin{aligned}
			\bigg\| \, X_N \, - \, \sum_{i=1}^{M} X_{N,M}^{(i)} \, \bigg\|_{L^2}^2 
			& = \big\| X_N \big\|_{L^2}^2\, - \, \bigg\|  \sum_{i=1}^{M} X_{N,M}^{(i)} \, 
			\bigg\|_{L^2}^2
			=\big\| X_N \big\|_{L^2}^2\, - \, \sum_{i=1}^{M} \big\| X_{N,M}^{(i)} \big\|_{L^2}^2 .
		\end{aligned}
	\end{equation*}
	
As a consequence, to prove our goals \eqref{eq:approxbysum0-1}
and \eqref{eq:approxbysum0-2} it is enough to show that
\begin{equation}\label{eq:limsecmom}
	\lim_{N\to\infty} \, \bbE\big[ X_N^2 \big] = v_{\varphi}\,, \qquad \quad
	\forall M\in\N: \quad
	\lim_{N\to \infty} \, \sum_{i=1}^{M} \bbE \Big[ \big(X_{N,M}^{(i)} \big)^2 \Big]
	= v_{\varphi} \,,
\end{equation}
where we recall that $v_{\varphi}$ is defined in \eqref{eq:sigmaphi}.
The first relation in \eqref{eq:limsecmom} follows from the second one,
because $X_N = X_{N,1}^{(1)}$. Then the proof is completed by the next result.
\qed

\begin{lemma}[Quasi-critical variance]\label{lem:limitsecmomZ(i)}
Fix $(\beta_N)_{N\in\N}$ in the quasi-critical regime, see \eqref{eq:sigma} and \eqref{eq:quasicrit},
and $\varphi \in C_c(\R^2)$.
For any $M \in \N$, the following holds for all $i=1,\ldots,M$:
 \begin{equation}
 	\label{eq:limsecmomZ(i)}
 	\lim_{N\to \infty} \bbE \big[ \big(X_{N,M}^{(i)} \big)^2 \big] 
	= v_{\varphi, \, (\frac{i-1}{M},\frac{i}{M}]} \, := \!
	\int\limits_{\R^2 \times \R^2} \! \varphi(x) \, \varphi(x')  \, 
	\bigg( \int_{\frac{i-1}{M}}^{\frac{i}{M}} \frac{1}{2u} \, \rme^{-\frac{|x-x'|^2}{2u}} \, 
	\mathrm{d}u \bigg)  \, \mathrm{d}x \, \mathrm{d}x' \,.
 \end{equation}
 \end{lemma}

\begin{proof}
Let us fix $M\in\N$ and $1 \le i \le M$. We
split the proof of \eqref{eq:limsecmomZ(i)} in the two bounds
	\begin{equation}
		\label{eq:ub}
		\limsup_{N\to \infty}\,
		\bbE \Big[ \big(X_{N,M}^{(i)} \big)^2 \Big] \le 
		v_{\varphi, \, (\frac{i-1}{M},\frac{i}{M}]}
	\end{equation}
and
	\begin{equation}
		\label{eq:lb}
		\liminf_{N\to \infty}\, \bbE 
		\Big[ \big(X_{N,M}^{(i)} \big)^2 \Big] \ge 
		v_{\varphi, \, (\frac{i-1}{M},\frac{i}{M}]} \,.
	\end{equation}

We first obtain an exact expression
for the second moment of $X_{N,M}^{(i)}$ by \eqref{eq:Z(i)}:
since the random variables $\xi_\beta(n,x)$ 
are independent with zero mean and variance $\sigma_\beta^2$, 
we have
\begin{equation*}
	\bbE \Big[ \big(X_{N,M}^{(i)} \big)^2 \Big]
	= \frac{\theta_N}{N^2}  \, \sum_{k=1}^\infty (\sigma_{\beta_N}^2)^k  
	\sum_{\substack{\frac{i-1}{M}N<n_1<\ldots<n_k\le \frac{i}{M}N\\
	x_1,\ldots,x_k \in \Z^2}} \, q_{n_1}^{\phi_N}(x_1)^2 \, 
	\prod_{j=2}^{k} q_{n_j-n_{j-1}}(x_j-x_{j-1})^2 \,.
\end{equation*}
We can sum the space variables $x_k, x_{k-1}, \ldots, x_2$ because
$\sum_{x\in\Z^2} q_n(x)^2 = q_{2n}(0)$, see \eqref{eq:defofq},
while to handle the sum over $x_1$ we note that, recalling \eqref{eq:kernel1},
\begin{equation} \label{eq:q2}
	\sum_{x\in\Z^2} q_{n}^{\varphi}(x)^2 = q_{2n}^{\varphi,\varphi} \qquad \text{where we set} \quad
	q_{m}^{\varphi,\varphi} := 
	\sum_{z, z' \in \Z^2} q_{m}(z-z') \, \varphi(z)
	\, \varphi(z') \,.
\end{equation}
We then obtain
\begin{equation} \label{eq:secmomforlb}
	\bbE \Big[ \big(X_{N,M}^{(i)} \big)^2 \Big]
	= \theta_N  \, \sum_{k=1}^\infty (\sigma_{\beta_N}^2)^k  
	\sum_{\frac{i-1}{M}N<n_1<\ldots<n_k\le \frac{i}{M}N} \, 
	\frac{q_{2n_1}^{\phi_N,\varphi_N}}{N^2} \, 
	\prod_{j=2}^{k} q_{2(n_j-n_{j-1})}(0) \,.
\end{equation}

\smallskip

We then prove the upper bound \eqref{eq:ub}.
We rename $n_1 = n$ and enlarge
the sum over the other time variables $n_2, \ldots, n_k$, by letting each increment $m_j:=n_j-n_{j-1}$
for $j=2,\ldots, k$ vary in the whole interval $(0,N]$: since $\sum_{m=1}^N q_{2m}(0) = R_N$, 
see \eqref{eq:RN}, we obtain
\begin{align}\notag
	\bbE \Big[ \big(X_{N,M}^{(i)} \big)^2 \Big]
	& \le \theta_N  \, 
	\sum_{\frac{i-1}{M}N<n \le \frac{i}{M}N} \, 
	\frac{q_{2n}^{\varphi_N,\varphi_N}}{N^2} \, 
	\sum_{k=1}^\infty (\sigma_{\beta_N}^2)^{k}
	( R_N)^{k-1} \\
	 \label{eq:ub+}
	& = \theta_N  \,
	\Bigg\{ \sum_{\frac{i-1}{M}N<n \le \frac{i}{M}N} \, 
	\frac{q_{2n}^{\varphi_N,\varphi_N}}{N^2} \Bigg\} \cdot \sigma_{\beta_N}^2 \cdot
	\frac{1}{1 \, - \, \sigma_{\beta_N}^2R_N} \,,
\end{align}
where we summed the geometric series since
$\sigma_{\beta_N}^2 R_N = 1- \frac{\theta_N}{\log N} < 1$ for large~$N$, by \eqref{eq:quasicrit}.
We will prove the following Riemann sum approximation, for any given $0 \le a < b \le 1$:
\begin{equation}\label{eq:Riemann}
	\lim_{N\to\infty} \,
	\sum_{aN <n \le bN} \, 
	\frac{q_{2n}^{\varphi_N,\varphi_N}}{N^2} \, = \!\!
	\int\limits_{\R^2 \times \R^2} \varphi(x) \, \varphi(x')  \, 
	\bigg( \int_{a}^{b} \frac{1}{u} \,
	g\bigg(\frac{x-x'}{\sqrt{u}}\bigg) \,
	\mathrm{d}u \bigg)  \, \mathrm{d}x \, \mathrm{d}x' \,,
\end{equation}
where $g(y) = \frac{1}{2\pi} \, \rme^{-\frac{1}{2}|y|^2}$ is the standard Gaussian density
on $\R^2$, see \eqref{eq:llt}. Plugging this into \eqref{eq:ub+}, 
since $1-\sigma_{\beta_N}^2 R_N = \frac{\theta_N}{\log N}$
and $\sigma_{\beta_N}^2 \sim \frac{1}{R_N} \sim \frac{\pi}{\log N}$ as $N\to\infty$
by \eqref{eq:quasicrit} and \eqref{eq:RN}, we obtain
precisely the upper bound \eqref{eq:ub}
(note that $\pi\, \frac{1}{u} \, g(\frac{x-x'}{\sqrt{u}}) = \frac{1}{2u} \, \exp(-\frac{|x-x'|^2}{2u})$).

\smallskip
Let us now prove \eqref{eq:Riemann}. This is based on the local limit theorem \eqref{eq:llt} as
$n\to\infty$, hence the case $a=0$ could be delicate, as the sum in \eqref{eq:Riemann}
starts from $n=1$ and, therefore, $n$ needs not be large. 
For this reason, we first show that small values of $n$ are negligible for \eqref{eq:Riemann}.
Since $\varphi$ is compactly supported, when we plug $f = \varphi_N$ into
$q_{2n}^{f,f}$, see \eqref{eq:q2},
we can restrict the sums
to $|z'| \le C \sqrt{N}$, which yields the following \emph{uniform bound}:
\begin{equation} \label{eq:ubq2}
	\forall m \in \N: \qquad
	|q_{m}^{\varphi_N,\varphi_N}| \le \|\varphi\|_\infty^2 \sum_{|z'| \le C \sqrt{N}}
	\sum_{z\in\Z^2} q_{m}(z-z') 
	\le C' \, \|\varphi\|_\infty^2 \, N \,.
\end{equation}
In particular, the contribution of $n \le \epsilon N$
to the LHS of \eqref{eq:Riemann} is $O(\epsilon)$. As a consequence,
it is enough to prove \eqref{eq:Riemann} when $a > 0$, which we assume henceforth.

Recalling \eqref{eq:q2} and applying \eqref{eq:llt}, we can write the LHS of \eqref{eq:Riemann} as follows:
\begin{equation*}
	\sum_{aN <n \le bN} \, 
	\frac{q_{2n}^{\varphi_N,\varphi_N}}{N^2} \, = \,
	\frac{1}{N^2}  \,
	\sum_{aN <n \le bN} \!\!
	\sum_{\substack{z, z' \in \Z^2 :\\ (n,z-z') \in \Z^3_{\mathrm{even}}}} \!\!
	\frac{2}{n} \, \Big( g\Big(\tfrac{z-z'}{\sqrt{n}}\Big) + o(1) \Big) 
	\, \phi \big( \tfrac{z }{\sqrt{N}} \big)
	\, \phi \big( \tfrac{z' }{\sqrt{N}} \big) \,,
\end{equation*}
where $o(1) \to 0$ as $N\to\infty$ (because $n > aN \to \infty$ and we assume $a>0$).
The additive term $o(1)$ gives a vanishing contribution as $N\to\infty$,
because we can bound $\frac{2}{n} \le \frac{2}{a N}$ and $|\varphi(\cdot)| \le \|\varphi\|_\infty$,
and the sums contain $O(N^3)$ terms (since $|z|, |z'| \le C \sqrt{N}$).
Introducing the rescaled variables $u := \frac{n}{N}$ and $x:=\frac{z}{\sqrt{N}}$,
$x':=\frac{z'}{\sqrt{N}}$, we can then rewrite the RHS as
\begin{equation*}
	\frac{1}{N^3}  \,
	\sum_{u \in (a,b] \cap \frac{\N}{N}} \!\!
	\sum_{\substack{x, x' \in \frac{\Z^2}{\sqrt{N}} :\\ 
	(Nu, \sqrt{N}(x-x')) \in \Z^3_{\mathrm{even}}}} \!\!
	\frac{2}{u} \, \Big( g\Big(\tfrac{x-x'}{\sqrt{u}}\Big) \Big) 
	\, \phi (x)
	\, \phi (x')
	 \, + \, o(1) \,,
\end{equation*}
which is a Riemann sum for the integral in the RHS of \eqref{eq:Riemann}.
Note that the restriction $(Nu, \sqrt{N}(x-x')) \in \Z^3_{\mathrm{even}}$
effectively \emph{halves} the range of the sum: indeed, for any given $u$ and $x$,
the sum over $x' = \frac{z'}{\sqrt{N}} \in \frac{\Z^2}{\sqrt{N}}$ is restricted to points
$z' \in \Z^2$ with a fixed parity (even or odd, depending on $u, x$).
This restriction is compensated by the multiplicative factor $2$, which disappears
as we let $N\to\infty$. This completes the proof of \eqref{eq:Riemann}.

\smallskip

We finally prove the lower bound \eqref{eq:lb}. 
We fix $\epsilon >0$ small enough and we bound
the RHS of \eqref{eq:secmomforlb} from below as follows:
\begin{itemize}
	\item we rename $n = n_1$ and we restrict its sum to the interval
	$\big( \, \frac{i-1}{M}N \,, \, (1-\epsilon)\frac{i}{M}N \big]$;
	\item for $k\ge 2$, we introduce the ``displacements'' $m_j := n_{j} - n_{1}$
	from $n_1$, for $j=2,\ldots, k$,
	and we restrict the sum over $n_2, \ldots, n_k$
	to the set $0<m_2<\ldots<m_{k}\le\epsilon \frac{i}{M}N$.
\end{itemize}
We thus obtain by \eqref{eq:secmomforlb}
\begin{equation}\label{eq:serieslb}
\begin{split}
	\bbE \Big[ \big(X_{N,M}^{(i)} \big)^2 \Big]
	\ge \, & \,\theta_N \,
	\sum_{\frac{i-1}{M}N<n \le (1-\epsilon) \frac{i}{M}N}
	\frac{q_{2n}^{\varphi_N,\varphi_N}}{N^2} \,  \times \\
	& \times  \Bigg( \sigma_{\beta_N}^2 +
	\sum_{k=2}^\infty (\sigma_{\beta_N}^2)^{k} \!\!
	\sum_{0< m_2 < \ldots < m_{k} \le \epsilon \frac{i}{M}N} \!
	q_{2m_2}(0) \prod_{j=3}^{k} q_{2(m_j-m_{j-1})}(0)  \Bigg) .
\end{split}
\end{equation}

We now give a probabilistic interpretation to the sum over $m_2,\ldots, m_k$:
following \cite{CSZ19a} and recalling \eqref{eq:RN}, given $N \in \N$ 
we define i.i.d.\ random variables $(T_i^{(N)})_{i \in \N}$ with distribution
\begin{equation}\label{eq:distrT}
	\P \big(T_i^{(N)}=n\big) \, = \, \frac{q_{2n}(0)}{R_N} \, \ind_{\{1,\ldots,N\}}(n)\,,
\end{equation}  
so that the second line of \eqref{eq:serieslb} can be written, renaming
$\ell = k-1$, as
\begin{equation} \label{eq:probaint}
\begin{split}
	& \sigma_{\beta_N}^2 \Bigg(1 \,+ \,  \sum_{\ell=1}^{\infty} 
	(\sigma_{\beta_N}^2  R_N)^\ell
	\,  \P \Big(T_1^{(N)} \, + \ldots + \, T_\ell^{(N)} \, \le \, \eps \, \tfrac{i}{M}N\Big) \Bigg) \\
	& \quad  = \sigma_{\beta_N}^2 \Bigg( \frac{1}{1-\sigma_{\beta_N}^2 R_N} - 
	\sum_{\ell=1}^{\infty} 
	(\sigma_{\beta_N}^2 R_N)^\ell
	\,  \P \Big(T_1^{(N)} \, + \ldots + \, T_\ell^{(N)} \, > \, \eps \, \tfrac{i}{M}N\Big) \Bigg) \,.
\end{split}
\end{equation}
Plugging this into \eqref{eq:serieslb} and recalling \eqref{eq:ubq2}, we obtain
\begin{equation}\label{eq:serieslb2}
\begin{split}
	\bbE \Big[ \big(X_{N,M}^{(i)} \big)^2 \Big]
	\ge  & \, \theta_N \, \Bigg\{ 
	\sum_{\frac{i-1}{M}N<n \le (1-\epsilon) \frac{i}{M}N}
	\frac{q_{2n}^{\varphi_N,\varphi_N}}{N^2} \Bigg\} \,  
	\frac{\sigma_{\beta_N}^2}{1-\sigma_{\beta_N}^2 R_N}  \\
	&  - \big( C' \, \|\varphi\|_\infty^2 \big) \, \theta_N \, \sigma_{\beta_N}^2 \, 
	\sum_{\ell=1}^{\infty} 
	(\sigma_{\beta_N}^2 R_N)^\ell
	\,  \P \Big(T_1^{(N)} \, + \ldots + \, T_\ell^{(N)} \, > \, \tfrac{\eps}{M} N\Big) \,.
\end{split}
\end{equation}
The first term in the RHS is similar to \eqref{eq:ub+}, 
just with $(1-\epsilon) \frac{i}{M}$ instead of $\frac{i}{M}$, therefore 
\emph{we already proved that it converges to 
$v_{\varphi, \, (\frac{i-1}{M},(1-\epsilon)\frac{i}{M}]}$ as
$N\to\infty$},
see \eqref{eq:Riemann} and the following lines (recall also \eqref{eq:limsecmomZ(i)}).
Letting $\epsilon \downarrow 0$ after $N\to\infty$ we recover 
$v_{\varphi, \, (\frac{i-1}{M},\frac{i}{M}]}$,
hence to prove \eqref{eq:lb} we just need
to show that the second term in the RHS of \eqref{eq:serieslb2}
is negligible:
\begin{equation}\label{eq:last-goal}
	\lim_{N\to\infty} \, \theta_N \, \sigma_{\beta_N}^2 \, 
	\sum_{\ell=1}^{\infty} 
	(\sigma_{\beta_N}^2 R_N)^\ell
	\,  \P \Big(T_1^{(N)} \, + \ldots + \, T_\ell^{(N)} \, > \, \tfrac{\eps}{M} N\Big) = 0 \,.
\end{equation}

Recall that the random variables $(T^{(N)}_i)_{i\in\N}$ are i.i.d.\
with distribution \eqref{eq:distrT}. Since $q_{2n}(0) \le \frac{C}{n}$
by the local limit theorem \eqref{eq:llt}, we have
$\E[T^{(N)}_i] = \frac{1}{R_N} \sum_{n=1}^N n \, q_{2n}(0) \le C \frac{N}{R_N}$
and, by Markov's inequality, we can bound
\begin{equation*}
	\P \Big(T_1^{(N)} \, + \ldots + \, T_\ell^{(N)} \, > \, \tfrac{\eps}{M} N\Big)
	\le \frac{\E \big[T_1^{(N)} \, + \ldots + \, T_\ell^{(N)} \big]}{\tfrac{\eps}{M} N}
	\le \frac{C \, \ell }{\frac{\epsilon}{M} \, R_N} \, .
\end{equation*}
Since $\sum_{\ell = 1}^\infty \ell \, x^\ell = \frac{x}{(1-x)^2}$, we obtain
\begin{equation*}
\begin{split}
	\theta_N \, \sigma_{\beta_N}^2 \, \sum_{\ell=1}^{\infty} 
	(\sigma_{\beta_N}^2 R_N)^\ell
	\,  \P \Big(T_1^{(N)} \, + \ldots + \, T_\ell^{(N)} \, > \, \tfrac{\eps}{M} N\Big)
	&\le \theta_N \, \sigma_{\beta_N}^2 \,
	\frac{C}{\frac{\epsilon}{M} \, R_N} \, 
	\frac{\sigma_{\beta_N}^2 R_N}{(1-\sigma_{\beta_N}^2 R_N)^2}  \\
	& = \frac{C \, M}{\epsilon} \, \frac{\theta_N \, (\sigma_{\beta_N}^2)^2}{(1-\sigma_{\beta_N}^2 R_N)^2} \,.
\end{split}
\end{equation*}
Note that $1-\sigma_{\beta_N}^2 R_N = \frac{\theta_N}{\log N}$
and $\sigma_{\beta_N}^2 \sim \frac{1}{R_N} \sim \frac{\pi}{\log N}$  by
\eqref{eq:quasicrit} and \eqref{eq:RN}, hence the last term is asymptotically equivalent
to $\frac{C M}{\epsilon} \, \frac{ \pi^2}{\theta_N} \to 0$ as $N\to\infty$,
since $\theta_N \to \infty$, see \eqref{eq:quasicrit}.
This shows that \eqref{eq:last-goal} holds and completes the proof of 
Proposition~\ref{prop:approximationZbysum}.
\end{proof}

\section{General moment bounds}
\label{sec:moment3}

In this section we estimate the \emph{moments of the partition function} $Z_{L,\beta}^{\omega}$
through a refinement of the operator approach from \cite[Theorem~6.1]{CSZ23}
and \cite[Theorem~1.3]{LZ21+} (inspired by \cite{GQT21}).
We point out that these papers deal with the critical and sub-critical regimes, 
while we are interested the quasi-critical regime~\eqref{eq:quasicrit}.

For transparency, and in view of future applications,
we develop in this section a \emph{non asymptotic approach
which  is independent of the regime of~$\beta$}: we obtain bounds 
with explicit constants 
which hold for any given system size~$L$ and disorder strength~$\beta$.
Some novelties with respect to \cite{CSZ23,LZ21+}
are described in Remarks~\ref{rem:novel1}, \ref{rem:novel2},~\ref{rem:novel3}.
These bounds will be crucially applied in Section~\ref{sec:4th-3} to prove
Proposition~\ref{prop:stimamomento}.

\smallskip

The section is organised as follows:
\begin{itemize}
\item in  Subsection~\ref{sec:4th-1}
we give an \emph{exact expansion} for the moments, see Theorem~\ref{th:moments-exact},
in terms of suitable operators linked to the random walk and the disorder;

\item Subsection~\ref{sec:4th-2} we deduce \emph{upper bounds} for the moments,
see Theorems~\ref{th:moments-bound1} and~\ref{th:moments-bound2}, which
depend on two pairs of quantities, that we call \emph{boundary terms}
and \emph{bulk terms};

\item in Subsection~\ref{sec:rw-bounds} we state some basic random walk bounds
needed in our analysis (we consider general symmetric random walks with sub-Gaussian tails);

\item in Subsections~\ref{sec:4th-boundaryL} and~\ref{sec:4th-bulk}
we obtain explicit estimates on the boundary terms and bulk terms, which plugged
in Theorem~\ref{th:moments-bound2} yield explicit bounds on the moments.
\end{itemize}

\smallskip
\subsection{Moment expansion}
\label{sec:4th-1}

The partition function $Z_{(A,B],\beta}^{\omega}(z)$ in \eqref{eq:ZAB} is called ``point-to-plane'',
since random walk paths start at~$S_0 = z$ but have no constrained endpoint. We
introduce a ``point-to-point'' version, for
simplicity when $(A,B] = (0,L]$ for $L \in \N$, restricting to random
walk paths with a fixed endpoint $S_L = w$:
\begin{equation}\label{eq:Zgen}
	\cZ_{L,\beta}^{\omega}(z,w):= \E \Big[ \rme^{\sum_{n =1}^{L-1}
	\{ \beta  \omega(n,S_n)-\lambda(\beta)\}} \, \ind_{\{S_L = w\}} 
	\, \Big| \, S_0 = z \Big]
\end{equation}
(we stop the sum at $n=L-1$ for later convenience).

Given two ``boundary conditions'' $f, g : \Z^2 \to \R$, we define the averaged version
\begin{equation} \label{eq:Zgenav}
	\cZ_{L,\beta}^{\omega}(f,g) :=
	\sum_{z,w \in \Z^2} f(z) \, \cZ_{L,\beta}^{\omega}(z,w) \, g(w) \,,
\end{equation}
where we use a different font to avoid confusions with the diffusively rescaled
average \eqref{eq:ZNav}.
We focus on the \emph{centred moments} 
of $Z_{L,\beta}^{\omega}(f,g)$, that we denote by
\begin{equation} \label{eq:Mkdef}
	\cM_{L,\beta}^{h}(f,g)  :=
	\bbE \Big[ \big( \cZ_{L,\beta}^{\omega}(f,g) - \bbE[\cZ_{L,\beta}^{\omega}(f,g)] \big)^h \Big]
	\qquad \text{for } h \in \N \,.
\end{equation}

\begin{remark}\label{rem:backtous}
Recalling the definition \eqref{eq:Xi} of $X_{N,M}^{(i)}$, we have the equality in law
\begin{equation} \label{eq:eqlaw}
	X_{N,M}^{(i)} \,\overset{d}{=}\,
	\frac{\sqrt{\theta_N}}{N} \, \cZ_{L,\beta_N}^{\omega}(f,g) \qquad
	\text{for suitable $L,f,g$} \,.
\end{equation}
More precisely, in view of the translated partition
function $Z_{(\frac{i-1}{M}N,\frac{i}{M}N],\beta_N}^{\omega}$
appearing in  \eqref{eq:Xi}, 
relation \eqref{eq:eqlaw} holds if we choose:
\begin{itemize}
\item $L = \frac{i}{M}N - \frac{i-1}{M}N
= \frac{N}{M}$ by translation invariance;
\item $f= q_{\frac{i-1}{M}N}^{\varphi_N}$, that is $f$
is the function $\varphi_N$ from  \eqref{eq:Xi}
``evolved from time~$0$ to time~$\frac{i-1}{M}N$
under the random walk'', i.e.
convolved with the random walk kernel $q_{\frac{i-1}{M}N}$ as in \eqref{eq:kernel1};
\item $g \equiv 1$.
\end{itemize}
We can thus write
\begin{equation}\label{eq:backtous}
	\mathbb E\left [(X_{N,M}^{(i)})^4\right ] = 
	\frac{\theta_N^2}{N^4} \, \cM_{\frac{N}{M},\beta_N}^{4}(f_i, g ) \,, \qquad
	\text{where} \quad 
	\begin{cases}
	f(z) := q_{\frac{i-1}{M}N}^{\varphi_N}(z) \,, \\
	g(w) := 1 \,.
	\end{cases}
\end{equation}
To prove Proposition~\ref{prop:stimamomento},
in Section~\ref{sec:4th-3} we will focus on $\cM_{L,\beta}^{4}(f,g)$.
\end{remark}

Henceforth we fix $h\in\N$ with $h\ge 2$ (the interesting case is $h \ge 3$).
We are going to give an \emph{exact expression} for $\cM_{L,\beta}^{h}(f,g)$,
see Theorem~\ref{th:moments-exact}.
We first need some notation.

\smallskip

We denote by $I \vdash\{1,\ldots,h\}$ a 
\emph{partition of $\{1,\ldots,h\}$}, i.e.\ a family $I = \{I^1,  \ldots, I^m\}$
of non empty disjoint subsets $I^j \subseteq \{1,\ldots, h\}$ with $I^1 \cup \ldots \cup 
I^m = \{1,\ldots, h\}$.
We single out:
\begin{itemize}
\item the unique partition $I = * := \{\{1\}, \{2\}, \ldots, \{h\}\}$ composed by all singletons;
\item the $\binom{h}{2}$ 
partitions of the form $I = \{\{a,b\}, \{c\} \colon c \ne a, c \ne b\}$, that we call \emph{pairs}.
\end{itemize}

\begin{example}[Cases $h =2,3, 4$]\label{ex:234}
All partitions $I \vdash \{1,2\}$ are $I=*$ and $I = \{\{1,2\}\}$.

All partitions $I \vdash \{1,2,3\}$ are $I=*$,  three \emph{pairs}
$I=\{\{a,b\},\{c\}\}$ and $I =\{\{1,2,3\}\}$.

All partitions $I \vdash \{1,2,3,4\}$ are 
 $I=*$,  six \emph{pairs} $I = \{\{a,b\},\{c\},\{d\}\}$, 
three \emph{double pairs} $I = \{\{a,b\}, \{c,d\}\}$, 
four \emph{triples} $I = \{\{a,b,c\}, \{d\}\}$ and
the \emph{quadruple} $I = \{\{1,2,3,4\}\}$.
\end{example}

Given a partition $I = \{I^1,  \ldots, I^m\} \vdash \{1,\ldots,h\}$,
we define for $\bx =(x^1, \ldots , x^h) \in (\Z^2)^h$
\begin{equation}\label{eq:xsimI}
\begin{split}
	\bx \sim I \quad \ \mbox{if and only if}
	\ \begin{cases}  x^a=x^b  
	& \text{if } a,b \in I^i \text{ for some } i \,, \\
	\rule{0pt}{1.2em}x^a \ne x^b & \text{if } a \in I^i, b \in I^j \text{ for some } i \ne j
	\text{ with } |I^i|, |I^j| \ge 2 \,.
	\end{cases}
\end{split}
\end{equation}
For instance $\bx \sim \{\{1,2\}, \{3\}, \{4\}\}$ means $x^1 = x^2$, while
$\bx \sim \{\{1,2\}, \{3,4\}\}$ means $x^1 = x^2$ and $x^3 = x^4$ with $x^1 \ne x^3$.
Note that $\bx \sim *$ imposes no constraint.
We also define 
\begin{equation}\label{ZI-space}
	(\Z^2)^h_I := \big\{ \bx\in (\Z^2)^h \colon \bx=(x^1, \ldots , x^h) \sim I \big\} \,,
\end{equation}
which is essentially a copy of $(\Z^2)^{m}$ embedded in $(\Z^2)^h$,
because $\bx \sim I = \{I^1,  \ldots, I^m\}$ means that we only have $m$ ``free'' variables,
one for each component $I^i$.

A family $I_1, \ldots, I_r$ of partitions $I_i = \{I_i^1, \ldots, I_i^{m_i}\} \vdash \{1,\ldots,h\}$
is said to have \emph{full support} if 
any $a\in\{1,\ldots,h\}$ belongs to some partition $I_i$ not as a singleton,
i.e.\ $a \in I_i^j$ with $|I_i^j| \ge 2$.

\begin{example}[Full support for $h=4$]
A single partition $I_1 \vdash \{1,2,3,4\}$ with full support
is either the quadruple $I_1 = \{\{1,2,3,4\}\}$
or a double pair $I_1 = \{\{a,b\},\{c,d\}\}$.
There are many families of two partitions $I_1, I_2 \vdash \{1,2,3,4\}$ with full support, 
for instance two non overlapping pairs such as 
$I_1 = \{\{1,3\},\{2\},\{4\}\}$, $I_2 = \{\{2,4\},\{1\},\{3\}\}$.
\end{example}

\smallskip

We now introduce $h$-fold analogues of the random walk transition kernel
\eqref{eq:defofq} and of its averaged version \eqref{eq:kernel1}:
given partitions $I, J \vdash \{1,\ldots,h\}$, we define for $\bx, \bz \in (\Z^2)^h$
\begin{equation}\label{eq:Qdef2}
	\sfQ_{n}^{I,J}(\bz, \bx)
	:= \ind_{\{\bz \sim I, \, \bx \sim J\}} \, \prod_{i=1}^h q_{n}(x^i - z^i)  \,,
	\qquad
	\sfq_{n}^{f,J}(\bx)
	:= \ind_{\{\bx \sim J\}} \, \prod_{i=1}^h q_{n}^f (x^i)  \,.
\end{equation}
Given $m \in \N_0$
and $J \vdash \{1,\ldots, h\}$ with $J \ne *$, we define for $\bx, \bz \in (\Z^2)^h$
the weighted \emph{Green's kernel}
\begin{equation}\label{eq:sfU}
	\sfU^J_{m,\beta}(\bz,\bx) 
	:= \begin{cases}
	\displaystyle \sum_{k=1}^\infty \! \big( \mathbb{E}[\xi_{\beta}^J] \big)^{k} \!\!\!
	\sum_{\substack{0 =: n_0 <n_1<\dots <n_k:= m\\ 
	\by_1, \ldots, \by_{k-1} \in (\Z^2)^h \\
	\by_0 :=\bz \,, \ \by_k:=\bx}} \
	\prod_{i=1}^k \sfQ^{J, J}_{n_i-n_{i-1}}(\by_{i-1}, \by_i)
	& \text{if } m \ge 1 \,, \\
	\ind_{\lbrace \bz = \bx \sim J \rbrace}
	& \text{if } m = 0 \,,
	\end{cases}
\end{equation}
where the outer sum is actually finite ($k \le m$ by the constraints
on the $n_i$'s) and we define
\begin{equation}\label{eq:Exi}
	\mathbb E[\xi_{\beta}^J] := \prod_{i: \, |J^i|\geq 2}\mathbb E[\xi_{\beta}^{|J^i|}]
	\qquad \text{for $J = \{J^1,  \ldots, J^\ell\}$ with $J \ne *$} \,.
\end{equation}
When $J$ is a pair, this reduces to 
$\mathbb E[\xi_{\beta}^J] = \mathbb E[\xi_{\beta}^2] = \sigma_{\beta}^2$, see \eqref{eq:meanvarxi}.

\begin{remark}[On the definition of $\sfU^J$]\label{rem:novel1}
We point out that $\sfU^J$ was only defined in \cite{CSZ23,LZ21+} when $J$ is a pair.
Defining $\sfU^J$ for any partition $J$ makes formulas simpler,
as it avoids to distinguish between pairs and non-pairs in
the sums \eqref{eq:MNh4} and \eqref{eq:Xir}.

For a pair $J = \{\{a,b\}, \{c\}: c \ne a,b\}$, 
since $\bx \sim J$ for $\bx = (x^1, \ldots, x^h) \in (\Z^2)^h$
simply means $x^a = x^b$,
by Chapman-Kolmogorov we can express
\begin{equation} \label{eq:sfUpair}
	\sfU^J_{m,\beta}(\bz,\bx) = U_{m,\beta}(x^a-z^a)
	\, \ind_{\{x^b = x^a, \, z^b = z^a\}} \, \prod_{c \ne a,b} \, q_m(x^c - z^c) \,,
\end{equation}
where we define $U_{m,\beta}(x)$ for $x \in \Z^2$ by
\begin{equation}\label{eq:Ux}
	U_{m,\beta}(x) := \sum_{k=1}^\infty (\sigma_\beta^2)^{k} \!
	\sum_{\substack{0 =: n_0 <n_1<\dots <n_k:= m \\
	x_0 := 0, \; x_1, \ldots, x_{k-1} \in \Z^2,\; x_k := x}} \
	\prod_{i=1}^k q_{n_i - n_{i-1}}(x_i - x_{i-1})^2 \,.
\end{equation}
(We denote a generic sequence of points $x_i \in \Z^2$ using subscripts, while we use
superscripts to denote the $h$ components $x^a \in \Z^2$ 
of a vector $\bx = (x^1, \ldots, x^h) \in (\Z^2)^h$.)
\end{remark}

Given the countable set $\bbT = (\Z^2)^h$,
for the one-variable functions $\sfq^f, \sfq^g : \bbT \to \R$ and the
two-variable functions $\sfU_i, \sfQ_i: \bbT \times \bbT \to \R$
we use the matrix-vector notation
\begin{equation*}
\begin{split}
	\bigg\langle \sfq^f ,\, \sfU_1 \, \bigg\{ \prod_{i=2}^r \sfQ_i \, \sfU_i \bigg\} 
	\,  \sfq^g\bigg\rangle
	 := \!\!\!\!\!\! \sum_{\substack{\bz_1, \ldots, \bz_r \in \bbT \\
	 \bz'_1, \ldots, \bz'_r \in \bbT}}
	 \sfq^f(\bz_1) \, \sfU_1(\bz_1, \bz'_1) \, \bigg\{ \prod_{i=2}^r \sfQ_i(\bz'_{i-1}, \bz_i) \,
	\sfU_i(\bz_{i}, \bz'_i) \bigg\} \,  \sfq^g(\bz'_r) \,.
\end{split}
\end{equation*}
We can now give the announced expansion for $\cM_{L,\beta}^{h}(f,g)$,
that we prove in Appendix~\ref{sec:4thapp}.

\begin{theorem}[Moment expansion]\label{th:moments-exact}
Let $\cZ_{L,\beta}^{\omega}(f,g)$ be the averaged partition function in \eqref{eq:Zgenav}
with centred moments $\cM_{L,\beta}^{h}(f,g)$, see \eqref{eq:Mkdef}.
For any $h\in\N$ with $h \ge 2$ we have
\begin{equation}\label{eq:MNh4}
\begin{aligned}
	\cM_{L,\beta}^{h}(f,g)
	= \ \sum_{r=1}^\infty \ &
	\sum_{0 < n_1 \le m_1 <\cdots <n_r \le m_r < L} 
	\sum_{\substack{I_1, \ldots, I_r \vdash\{1, \ldots, h\} \\
	\text{with full support} \\
	\text{and } I_i \ne I_{i-1}, \,\, I_i \ne * \ \forall i}} 
	\bigg\{ \prod_{i=1}^r \bbE[\xi_{\beta}^{I_i}] \bigg\} \,\times
	\\
	& \qquad \times \bigg\langle \sfq_{n_1}^{f,I_1} \,, \,\,
	\sfU^{I_1}_{m_1-n_1,\beta} \, 
	\bigg\{\prod_{i=2}^r \sfQ_{n_i - m_{i-1}}^{I_{i-1}, I_i} \, \,
	\sfU_{m_i - n_{i}, \beta}^{I_i} \, \bigg\} \, 
	\sfq_{L - m_r}^{g,I_r} \bigg\rangle \,.
\end{aligned}
\end{equation}
\end{theorem}

\begin{remark}[Sanity check]
In case $h=2$, the conditions $I_i \ne I_{i-1}$ and $I_i \ne *$ in \eqref{eq:MNh4}
force $r=1$ and $I_1 = \{\{1,2\}\}$. Then, recalling \eqref{eq:sfUpair}-\eqref{eq:Ux},
formula \eqref{eq:MNh4} reduces to
\begin{equation*}
	\cM_{L,\beta}^{2}(f,g)
	= \bbvar[\cZ_{L,\beta}^{\omega}(f,g)] =
	\sigma_\beta^2 \, \sum_{\substack{0 < n \le m < L\\
	z,x \in \Z^2}} q_n^f(z) \, U_{m-n,\beta}(x-z)
	\, q_{L-m}^g(x) \,,
\end{equation*}
which is a classical expansion for the variance, see e.g.\
\cite[eq.~(3.51)]{CSZ23}.
\end{remark}

\begin{remark}[Boundary conditions]\label{rem:novel2}
In \cite{CSZ23,LZ21+}, the quantity $\sfq_{n_1}^{f,I_1}$ in \eqref{eq:MNh4}
is expanded as $\sfQ_{n_1}^{I_1,*}  f^{\otimes h}$ (recall \eqref{eq:Qdef2} and \eqref{eq:kernel1});
similarly for $\sfq_{L - m_r}^{g,I_r}$.
We keep these quantities unexpanded
in order to derive tailored estimates, see Subsection~\ref{sec:4th-boundaryL},
which could not be derived by simply applying operator norm bounds on $\sfQ_{n_1}^{I_1,*}$
as in \cite{CSZ23,LZ21+}.
\end{remark}

\smallskip
\subsection{Moment upper bounds}
\label{sec:4th-2}
We next obtain upper bounds from \eqref{eq:MNh4}.
For $L\in\N$ we define the summed kernels
\begin{equation}
	\label{eq:QLap}
	\widehat \sfQ^{I,J}_{L}(\bz, \bx)
	:= \sum_{n=1}^L 
	\sfQ^{I,J}_{n}(\bz, \bx) 
	\,,  \qquad
	\widehat \sfq_{L}^{f,I} (\bx) := 
	\sum_{n=1}^L \sfq_{n}^{f,I}(\bx) \,.
\end{equation}
Recalling \eqref{eq:sfU} and \eqref{eq:Exi} we set, with some abuse of notation,
\begin{equation}\label{eq:sfUalt}
	|\sfU|^{J}_{m, \beta}(\bz, \bx) :=
	\sfU^{J}_{m, \beta}(\bz, \bx) \text{ \emph{from  (\ref{eq:sfU}) with $\bbE[\xi_{\beta}^J]$
	replaced by $|\bbE[\xi_{\beta}^J]|$}} \,.
\end{equation}
Then, for $L\in\N$ and $\lambda \ge 0$, we define the 
Laplace sum
\begin{equation} \label{eq:ULap}
	|\widehat \sfU|^{J}_{L, \lambda, \beta}(\bz, \bx) 
	:= \ind_{\{\bz = \bx \sim J\}} + \sum_{m=1}^L \rme^{-\lambda m} \,
	|\sfU|^{J}_{m, \beta}(\bz, \bx)  \,. 
\end{equation}
Finally, we introduce a \emph{uniform bound} on
the right boundary function $\sfq_{L - m_r}^{g,I_r}$ in \eqref{eq:MNh4}:
\begin{equation}\label{eq:Qright}
	\widebar{\sfq}_{L}^{g,J} (\bz) := 
	\max_{1 \le n \le L} \sfq_{n}^{g,J}(\bz) \,.
\end{equation}
We can now state our first moment upper bound.

\begin{theorem}[Moment upper bound, I]\label{th:moments-bound1}
Let $\cZ_{L,\beta}^{\omega}(f,g)$ denote the averaged partition function in \eqref{eq:Zgenav}
with centred moment $\cM_{L,\beta}^{h}(f,g)$, see \eqref{eq:Mkdef}, for $h\in\N$ with $h \ge 2$.
For any $\lambda \ge 0$ we have the upper bound
\begin{equation}\label{eq:MNh5}
\begin{gathered}
	\big| \cM_{L,\beta}^{h}(f,g) \big|
	\le  \,\rme^{\lambda L} \, \sum_{r=1}^\infty \, \Xi(r)
\end{gathered}
\end{equation}
with
\begin{equation}\label{eq:Xir}
	\Xi(r) := \!\!\!\!\!\! \sum_{\substack{I_1, \ldots, I_r \vdash\{1, \ldots, h\} \\
	\text{with full support} \\
	\text{and } I_i \ne I_{i-1} , \,\, I_i \ne * \,\forall i}}   
	\!\!\!\!\! \bigg\{
	\prod_{i=1}^r \big| \bbE[\xi_{\beta}^{I_i}] \big|  \bigg\} \,
	 \bigg\langle  \widehat \sfq_{L}^{|f|,I_1}  , 
	|\widehat \sfU|^{I_1}_{L, \lambda, \beta} \,
	\bigg\{\prod_{i=2}^r \, \widehat \sfQ_{L}^{I_{i-1}, I_i} \,
	|\widehat \sfU|_{L, \lambda, \beta}^{I_i}  \bigg\} \, 
	\widebar{\sfq}_{L}^{|g|,I_r}  \bigg\rangle \,.
\end{equation}
\end{theorem}

\begin{proof}
Replacing $\bbE[\xi_{\beta}^{I_i}]$, $f$, $g$, 
$\sfU$ in \eqref{eq:MNh4} respectively by $|\bbE[\xi_{\beta}^{I_i}]|$, $|f|$, $|g|$,  $|\sfU|$,
every term becomes non-negative. We next replace $\sfq_{L - m_r}^{|g|,I}$
by the uniform bound $\widebar{\sfq}_{L}^{|g|,I} $ and then enlarge the sum in \eqref{eq:MNh4},
allowing increments $n_i - m_{i-1}$ and $m_i - n_i$
to vary freely in $\{1,\ldots, L\}$.
Plugging $1 \le \rme^{\lambda L} \, \rme^{-\lambda m_r}
\le \rme^{\lambda L} \, \rme^{-\lambda \sum_{i=1}^r (m_i - n_i)}$,
we obtain \eqref{eq:MNh5}.
\end{proof}

\begin{remark}[On the right boundary condition]\label{rem:novel3}
The function $\sfq_{L - m_r}^{g,I_r}$
in \eqref{eq:MNh4} is controlled
in \cite{CSZ23,LZ21+} by introducing an average over~$L$, which forces
the function~$g$ to be estimated in $\ell^\infty$. Our approach avoids such averaging, via the
quantity $\widebar{\sfq}_{L}^{g,J}$ from \eqref{eq:Qright}:
this lets us estimate the function $g$ in $\ell^q$ also for $q < \infty$
(see Proposition~\ref{prop:right}).
\end{remark}

\smallskip
We next bound $\Xi(r)$ in \eqref{eq:Xir}, starting from the scalar product.
Let us recall some functional analysis:
given a countable set $\T$ and a function $f:\T \to \R$, we define
\begin{equation}\label{eq:defnormaellp}
	\Vert f \Vert_{\ell^p(\bbT)} 
	= \Vert f \Vert_{\ell^p} := \bigg( \sum_{z \in \T} |f(z)|^p\bigg)^{\frac{1}{p}} 
	\qquad \text{for }p \in [1,\infty)\,.
\end{equation}
For a linear operator ${\sfA: \ell^q(\bbT) \to \ell^q(\bbT')}$,
with $p, q \in (1,\infty)$ 
such that $\frac{1}{p}+\frac{1}{q}=1$, we have
\begin{equation}
	\Vert \sfA \Vert_{\ell^q\to \ell^q} \,:= \,
	\sup_{g \nequiv 0} \, \frac{\Vert \sfA\, g \Vert_{\ell^q(\bbT')}}{\Vert g \Vert_{\ell^q(\bbT)}}
	\,=\,  \sup_{ \Vert f\Vert_{\ell^p(\bbT')}\leq 1,
	\, \Vert g\Vert_{\ell^q(\bbT)}\leq 1} \langle f, \sfA\, g\rangle \,.
\end{equation}
By H\"older's inequality  $|\langle g, h \rangle| \le \|g\|_{\ell^p} \, \|h\|_{\ell^q}$,
so the scalar product in \eqref{eq:Xir} is bounded by
\begin{equation} \label{eq:factorbound}
	\big\Vert \widehat \sfq_{L}^{|f|,I_1} \big\Vert_{\ell^p} \;
	\big\Vert |\widehat \sfU|^{I_1}_{L,\lambda, \beta} \big\Vert_{\ell^q \to \ell^q} \,
	\bigg\{ \, \prod_{i=2}^r \, \big\Vert \widehat \sfQ_{L}^{I_{i-1}, I_i} 
	\big\Vert_{\ell^q \to \ell^q}
	\; \big\Vert |\widehat \sfU|_{L,\lambda, \beta}^{I_i} \big\Vert_{\ell^q \to \ell^q}
	\bigg\} \; \big\Vert \widebar{\sfq}_{L}^{|g|,I_r}  \big\Vert_{\ell^q} \,.
\end{equation}

\begin{remark}[Restricted $\ell^q$ spaces]
Due to the constraint $\ind_{\{\bz \sim I, \bx \sim J\}}$ in \eqref{eq:Qdef2},
we may  regard $\widehat \sfQ^{I,J}_{L}$
as a linear operator from $\ell^q((\Z^2)^h_J)$ to $\ell^q((\Z^2)^h_I)$, see \eqref{ZI-space}.
Similarly, we may view $|\widehat \sfU|^{J}_{L, \lambda,\beta}$ 
as a linear operator from $\ell^q((\Z^2)^h_J)$ to itself.
\end{remark}

To make the bound \eqref{eq:factorbound} more useful, we 
introduce a \emph{weight} $\cW: (\Z^2)^h \to (0,\infty)$,
that we also identify with the diagonal operator $\cW(\bx) \, \ind_{\{\bx=\by\}}$,
so that in particular
\begin{equation*}
	\big(\cW \, \sfA \, \tfrac{1}{\cW}\big)(\bx,\by) :=
	\cW(\bx) \, \sfA(\bx,\by) \, \frac{1}{\cW(\by)} \,.
\end{equation*}
Inserting $(\cW\, \frac{1}{\cW})$ between each pair of adjacent
operators in \eqref{eq:MNh5}, we improve \eqref{eq:factorbound} to
\begin{equation} \label{eq:factorbound+}
\begin{split}
	\big\Vert \widehat \sfq_{L}^{|f|,I_1} \tfrac{1}{\cW} \big\Vert_{\ell^p} \;
	& \big\Vert \cW \, |\widehat \sfU|^{I_1}_{L,\lambda, \beta}
	\, \tfrac{1}{\cW} \big\Vert_{\ell^q \to \ell^q} \, \times \\
	\times\, & \bigg\{ \, \prod_{i=2}^r \, \big\Vert \cW \, \widehat \sfQ_{L}^{I_{i-1}, I_i} 
	\, \tfrac{1}{\cW} \big\Vert_{\ell^q \to \ell^q}
	\; \big\Vert \cW \, |\widehat \sfU|_{L,\lambda, \beta}^{I_i}
	\, \tfrac{1}{\cW} \big\Vert_{\ell^q \to \ell^q}
	\bigg\} \; \big\Vert \cW \, \widebar{\sfq}_{L}^{|g|,I_r}  \big\Vert_{\ell^q} \,.
\end{split}
\end{equation}
In view of \eqref{eq:MNh5}-\eqref{eq:Xir}, this leads directly to our second moment upper bound.

\begin{theorem}[Moment upper bound, II]\label{th:moments-bound2}
Let $\cZ_{L,\beta}^{\omega}(f,g)$ be the averaged partition function in \eqref{eq:Zgenav},
whose centred moment are known to satisfy $\cM_{L,\beta}^{h}(f,g)
\le \rme^{\lambda L} \, \sum_{r=1}^\infty \, \Xi(r)$
for $h \ge 2$ and  $\lambda \ge 0$, see \eqref{eq:Mkdef} and \eqref{eq:MNh5}.
For any weight $\cW: (\Z^2)^h \to (0,\infty)$
and for $p,q \in (1,\infty)$ 
with $\frac{1}{p} + \frac{1}{q} = 1$, we have the following upper bound on $\Xi(r)$ from \eqref{eq:Xir}:
\begin{equation}\label{eq:MNh6}
\begin{split}
	\Xi(r)
	&\,\le\, 
	\Big( \max_{I \ne *} \big\Vert \widehat \sfq_{L}^{|f|,I} \tfrac{1}{\cW} \big\Vert_{\ell^p} \Big) 
	\, \Big( \max_{J \ne *} \big\Vert \cW \,
	\widebar{\sfq}_{L}^{|g|,J} \big\Vert_{\ell^q} \Big)
	\,  \Xi^\bulk(r)
\end{split}
\end{equation}
with
\begin{equation}\label{eq:Xibulk}
	\Xi^\bulk(r) \,:= \!\!\!
	\sum_{\substack{I_1, \ldots, I_r \vdash\{1, \ldots, h\} \\
	\text{with full support} \\
	\text{and } I_i \ne I_{i-1} , \,\, I_i \ne * \ \forall i}}   
	\!\!\! \bigg\{ \prod_{i=1}^r \big| \bbE[\xi_{\beta}^{I_i}] \big|  \bigg\} \,
	\big( \big\| \widehat \sfQ_{L}\big\|_{\ell^q \to \ell^q}^{\cW} \big)^{r-1} \,
	\big( \big\| |\widehat \sfU|_{L,\lambda, \beta}\big\|_{\ell^q \to \ell^q}^{\cW} \big)^{r} \,
\end{equation}
where we set for short
\begin{align}\label{eq:Crho1}
	\big\| \widehat \sfQ_{L} \big\|_{\ell^q \to \ell^q}^{\cW}
	& := \max_{\substack{I, J \ne * \\ I \ne J}} \, 
	\big\Vert \cW \, \widehat \sfQ_{L}^{I,J}  \, \tfrac{1}{\cW}
	\big\Vert_{\ell^q \to \ell^q} \,, \\
	\label{eq:Crho2}
	\big\| |\widehat \sfU|_{L,\lambda, \beta}\big\|_{\ell^q \to \ell^q}^{\cW}
	&:= \max_{I \ne *} \, 
	\big\Vert \cW \, |\widehat \sfU|_{L,\lambda, \beta}^{I}
	\, \tfrac{1}{\cW} \big\Vert_{\ell^q \to \ell^q} \,.
\end{align}
\end{theorem}

Note that the bound \eqref{eq:MNh6}-\eqref{eq:Xibulk} 
depends on two pairs of quantities, that we call
\begin{equation}\label{eq:terms}
	\text{\emph{boundary terms}} \ \begin{cases}
	\rule{0pt}{1em}\big\Vert \widehat \sfq_{L}^{|f|,I} 
	\tfrac{1}{\cW} \big\Vert_{\ell^p}  \\
	\rule{0pt}{1.4em} \big\Vert \cW \, \overline{\sfq}_{L}^{|g|,J} \big\Vert_{\ell^q}
	\end{cases}
	\text{and} \quad \ \
	\text{\emph{bulk terms}} \ \begin{cases}
	\rule{0pt}{1.1em} 
	\big\| \widehat \sfQ_{L} \big\|_{\ell^q \to \ell^q}^{\cW} \\
	\rule{0pt}{1.4em}
	\big\| |\widehat \sfU|_{L,\lambda, \beta}\big\|_{\ell^q \to \ell^q}^{\cW}
	\end{cases} \!\!\!\!\!\!\! \,.
\end{equation}
We will estimate these terms in Subsections~\ref{sec:4th-boundaryL} and~\ref{sec:4th-bulk}
respectively, exploiting some basic random walk bounds that we collect in Subsection~\ref{sec:rw-bounds}.

\begin{remark}[Choice of the parameters]
For our goals, we will later fix $p=q=2$
(in other contexts, such as \cite{LZ21+}, one needs to take
$p = p_L \to 1$, $q = q_L \to \infty$).
We will then choose an \emph{exponential
weight $\cW = \cW_t$ of rate $t \ge 0$}: for $\bx = (x^1, \ldots, x^h) \in (\Z^2)^h$
\begin{equation}\label{eq:weights}
	\cW_t(\bx) := \prod_{i=1}^h w_t(x^i) 
	\qquad \text{where} \qquad
	w_t(x) := \rme^{-t |x|} \ \text{for } x \in \Z^2 \,.
\end{equation}
The exponential decay ensures that
$\| g \, w_t \|_{\ell^q} < \infty$ for the ``flat'' boundary condition $g \equiv 1$,
see \eqref{eq:backtous}, and we will fix $t = 1/\sqrt{N}$ so that
$\|f \, w_t^{-1} \|_{\ell^p} < \infty$ for $f = \varphi_N \approx \varphi(\cdot/\sqrt{N})$.

Note that by the triangle inequality we can bound, for all $\bx,\bz \in (\Z^2)^h$,
\begin{equation} \label{eq:ratiocW}
	\frac{\cW_t(\bz)}{\cW_t(\bx)} \le \prod_{i=1}^h \rme^{t|z^i-x^i|} \,.
\end{equation}
We will later need to consider an additional weight $\cV_s^I$, see \eqref{eq:V} below.
\end{remark}

\smallskip
We finally bound the product
$\prod_{i=1}^r \big| \bbE[\xi_{\beta}^{I_i}] \big|$
in \eqref{eq:Xibulk}.
We assume that $\beta > 0$ is small enough so that (say)
$\sigma_\beta^2 \le 1$
(recall $\sigma_\beta$
from \eqref{eq:sigma} and \eqref{eq:meanvarxi}
and note that $\lim_{\beta \downarrow 0} \sigma_\beta = 0$).

\begin{proposition}[Moments of disorder]\label{lem:dismom}
Assume that $\sigma_\beta^2 \le 1$.
For any $h\in\N$ there is  
$C(h) < \infty$
(which depends on the disorder distribution)
such that 
\begin{equation}\label{eq:dismom}
	\ \text{for any } I \ne * \colon \quad
	\big| \bbE[\xi_{\beta}^{I}] \big|
	\le \begin{cases}
	\sigma_\beta^{2} & \text{if $I = \{\{a,b\}, \{c\} \colon c \ne a,b\}$ is a pair}\,, \\
	\rule{0pt}{1.2em}C(h) \, \sigma_\beta^{3} & \text{if $I \ne *$ is not a pair}\,.
	\end{cases}
\end{equation}
Moreover
\begin{equation}\label{eq:full-support}
	\text{if $I_1, \ldots, I_r \vdash \{1,\ldots, h\}$ have full support:} \qquad
	\prod_{i=1}^r \big| \bbE[\xi_{\beta}^{I_i}] \big| \le
	C(h)^r \, \sigma_\beta^{\max\{2r, h\}} \,.
\end{equation}
\end{proposition}

\begin{proof}
We have 
$|\bbE[\xi_\beta^I]| = \sigma_\beta^2$
if $I$ is a pair, see \eqref{eq:meanvarxi} and \eqref{eq:Exi}.
Consider now any partition $I = \{I^1, \ldots, I^m\} \vdash \{1,\ldots, h\}$
with $I \ne *$: denoting by
$\|I\| := \sum_{i=1}^m |I^i|  \, \ind_{\{|I^i| \ge 2\}}$
the number of $a \in \{1,\ldots, h\}$ which are not singletons
in~$I$, by \eqref{eq:meanvarxi} and \eqref{eq:Exi} we can bound
\begin{equation} \label{eq:mombo}
	\big| \bbE[\xi_{\beta}^{I}] \big|
	\le C(h) \, \sigma_\beta^{\|I\|}
	\qquad \text{with} \qquad
	C(h) := \max_{* \ne I \vdash \{1,\ldots, h\}} \, \prod_{i \colon
	k_i := |I^i|\ge 2} C_{k_i} \,.
\end{equation}
Since $\|I\|\ge 3$ if $I \ne *$ is not a pair, we obtain \eqref{eq:dismom} since $\sigma_\beta \le 1$.

Consider now $I_1, \ldots, I_r$ with full support. Each $a \in \{1,\ldots,h\}$
is a non-trivial element (not a singleton) of some partition $I_i$, 
hence $\|I_1\| + \ldots + \|I_r\| \ge h$
which yields $\prod_{i=1}^r \big| \bbE[\xi_{\beta}^{I_i}] \big| 
\le C(h)^r \sigma_\beta^h$ by \eqref{eq:mombo} and $\sigma_\beta \le 1$.
Since $\prod_{i=1}^r \big| \bbE[\xi_{\beta}^{I_i}] \big| \le (\sigma_\beta^2)^r$
by \eqref{eq:dismom}, we obtain \eqref{eq:full-support}.
\end{proof}

\smallskip

\subsection{Random walk bounds}
\label{sec:rw-bounds}

In this subsection we collect some useful random walk bounds,
stated in Lemmas~\ref{lem:weighted-rw}, \ref{lem:weighted-averaged-rw} and~\ref{lem:max-averaged-rw}.
The proofs are deferred to Appendix~\ref{sec:rwapp}.

Instead of sticking to the simple random walk on $\Z^2$,
we can allow for
\emph{any symmetric random walk with sub-Gaussian tails},
in the following sense.

\begin{assumption}[Random walk]\label{ass:rw}
We consider a random walk $S = (S_n)_{n\ge 0}$ on $\Z^2$ with
a symmetric distribution, i.e.\  $q_1(x) = \P(S_1 = x) = q_1(-x)$ for any $x\in\Z^2$,
and with sub-Gaussian tails, i.e.\ for some $c > 0$ we have, writing $x = (x^1, x^2)$,
\begin{equation}\label{eq:sub-Gaussian}
	\forall t \in \R, \ \forall a=1,2: \qquad
	\E\big[ \rme^{t \, S_1^a}\big] = \sum_{x\in\Z^2} \rme^{t x^a} q_1(x)
	\le \rme^{c \, \frac{t^2}{2}} \,.
\end{equation}
\end{assumption}

\begin{remark}
The simple random walk on $\Z^2$ satisfies \eqref{eq:sub-Gaussian} with $c=1$:
indeed, we can compute
$\sum_{x\in\Z^2} \rme^{t x^a} q_1(x) =
\frac{1}{2}(1+ \cosh(t)) \le \exp(t^2/2)$ (because $\cosh(t) \le \exp(t^2/2)$).
\end{remark}

We derive useful bounds
for the random walk transition kernel $q_n(x) = \P(S_n = x)$.

\begin{lemma}[Random walk bounds]\label{lem:weighted-rw}
Let Assumption~\ref{ass:rw} hold. There is $\sfc \in [1, \infty)$ such that
for all $t \ge 0$ and $n \in \N$
\begin{gather} \label{eq:rw-weight}
	\forall a=1,2: \qquad
	\sum_{x\in\Z^2} \rme^{t x^a} q_n(x)
	\le \rme^{\sfc \frac{t^2}{2} n} \,, \qquad
	\sum_{x\in\Z^2} \rme^{t x^a} \, \frac{q_n(x)^2}{q_{2n}(0)}
	\le \rme^{\sfc  \frac{t^2}{2} n} \,.
\end{gather}
Moreover, recalling $w_t(x) = \rme^{-t|x|}$ from \eqref{eq:weights}, we can bound
\begin{equation} \label{eq:rw-weight+}
	\bigg\| \frac{q_n}{w_{t}} \bigg\|_{\ell^1} \!\!=
	\sum_{x\in\Z^2} \rme^{t |x|} \, q_n(x)
	\le \sfc \, \rme^{2 \sfc t^2 n} \,, \qquad
	\bigg\| \frac{q_n}{w_{t}} \bigg\|_{\ell^\infty} \!\!=
	\sup_{x\in\Z^2} \Big\{ \rme^{t |x|} \, q_n(x) \Big\}
	\le \frac{\sfc \, \rme^{2\sfc  t^2 n}}{n} \,.
\end{equation}
\end{lemma}

We next extend the bounds in \eqref{eq:rw-weight+} to the 
\emph{averaged} random walk transition kernel $q_n^f(x)$, see \eqref{eq:kernel1},
for any $f: \Z^2 \to \R$. Let us agree that $a^{\frac{1}{\infty}} := 1$ for any $a > 0$.

\begin{lemma}[Averaged random walk bounds]\label{lem:weighted-averaged-rw}
Let Assumption~\ref{ass:rw} hold and let $\sfc$ be the constant from Lemma~\ref{lem:weighted-rw}.
For any $t\ge 0$ and $n\in\N$ we have, with $w_t(x) = \rme^{-t |x|}$,
\begin{gather}\label{eq:weighted-averaged-rw2}
	\forall p \in [1,\infty]: \qquad
	\bigg\| \frac{q_n^{f}}{w_{t}} \bigg\|_{\ell^p} \!\!\le
	\sfc \, \rme^{2\sfc \, t^2 n} \, \bigg\| \frac{f}{w_{t}} \bigg\|_{\ell^p} \,,
	\qquad \bigg\| \frac{q_n^{f}}{w_{t}} \bigg\|_{\ell^\infty} \!\!\le
	\frac{\sfc \, \rme^{2\sfc\, t^2 n}}{n^{\frac{1}{p}}} \,
	\bigg\| \frac{f}{w_{t}} \bigg\|_{\ell^p} \,.
\end{gather}
\end{lemma}

We finally prove a variant of the Hardy-Littlewood 
maximal inequality (see Appendix~\ref{sec:rwapp}).
Let us introduce a multi-dimensional generalisation of \eqref{eq:kernel1},
for $m\in\N$ and $F: (\Z^2)^m \to \R$:
\begin{equation}\label{eq:kernel1m}
	q_{n}^{\otimes m, F}(x_1, \ldots, x_m) := 
	\sum_{z_1, \ldots, z_m \in \Z^2} \bigg( \prod_{i=1}^m q_{n}(x_i-z_i) \bigg)
	\, F(z_1, \ldots, z_m) \,.
\end{equation}
We also use the standard notation $w_t^{\otimes m}(x_1, \ldots, x_m)
:= \prod_{i=1}^m w_t(x_i)$.

\begin{lemma}[Maximal random walk bounds]\label{lem:max-averaged-rw}
Let Assumption~\ref{ass:rw} hold and let $\sfc$ be the constant from Lemma~\ref{lem:weighted-rw}.
For any $m\in\N$, $t\ge 0$ and $L\in\N$ we have, with $w_t(x) = \rme^{-t |x|}$,
\begin{equation}\label{eq:max-averaged-rw}
\begin{split}
	\forall p \in (1,\infty] : \qquad
	\Big\| \max_{1 \le n \le L} \big| q_n^{\otimes m, F}\,
	w_{t}^{\otimes m} \, \big|
	\Big\|_{\ell^p} \,& \le \tfrac{p}{p-1} \, 
	\overline{\scrC}^{\,m} \, \,
	\big\| F \, w_{t}^{\otimes m} \big\|_{\ell^p} \\
	& \text{with} \quad
	\overline{\scrC} := 
	\, 5000 \, \pi \, \sfc^2 \, \rme^{4\sfc \, t^2 L}
\end{split}
\end{equation}
(we agree that $\frac{\infty}{\infty-1} := 1$).
\end{lemma}

\subsection{Boundary terms}
\label{sec:4th-boundaryL}

In this section we estimate the \emph{boundary terms} appearing
in \eqref{eq:MNh6}, see \eqref{eq:terms}.
The proofs are deferred to Appendix~\ref{app:boundary}.

\smallskip

We recall that the weight $\cW_t : (\Z^2)^h \to (0,\infty)$ is defined in \eqref{eq:weights}
for $t \ge 0$.
Our estimates contain the following constants
(with $\sfc$ from Lemma~\ref{lem:weighted-rw}):
\begin{equation}\label{eq:CbarC}
	\scrC := \sfc \, \rme^{2\sfc \, t^2 L} \,, \qquad
	\overline{\scrC} := 
	\, 5000 \, \pi \, \sfc^2 \, \rme^{4\sfc \, t^2 L} \,,
\end{equation}
where $L$ is the ``time horizon'' of the partition function $Z_{L,\beta}^{\omega}(f,g)$,
see \eqref{eq:Zgenav}.
We anticipate that we will take
\begin{equation} \label{eq:choiceoft}
	t = \tfrac{1}{\sqrt{N}} \qquad \text{with} \quad N \ge L \,.
\end{equation}
hence \emph{the constants $\scrC$ and $\overline{\scrC}$ are uniformly bounded} in this regime.

\smallskip

We start estimating the \emph{left boundary term}
which involves $\widehat \sfq_{L}^{|f|,I}$
(see \eqref{eq:QLap} and \eqref{eq:Qdef2}).
It was proved\footnote{The factor $q = \frac{p}{p-1}$ in the RHS of \eqref{eq:est1-old},
first identified in \cite{LZ21+}, is essential to allow for $p$ which can vary
with the system size~$L$.}
in \cite[Proposition~3.4]{LZ21+}, extending
\cite[Proposition~6.6]{CSZ23}, that for any $h \ge 2$
there is $C = C(h) < \infty$ such that,
for any $p \in (1,\infty)$,
\begin{equation}\label{eq:est1-old}
	\max_{I \ne *} \Big\Vert \widehat \sfq_{L}^{|f|,I} 
	\tfrac{1}{\cW_t} \Big\Vert_{\ell^p}
	\le \tfrac{p}{p-1} \, C  \, L^{1-\frac{1}{p}} \, \bigg\|\frac{f}{w_{t}}\bigg\|_{\ell^p}^{h} \,.
\end{equation}
For our goals it will be fundamental to have a \emph{linear dependence in~$L$}, which
would amount to take $p=\infty$ in \eqref{eq:est1-old}, but this is not allowed by our 
approach.
To solve this problem, we improve the estimate \eqref{eq:est1-old}, showing
that for $p \in (0,\infty)$ we can still have a linear dependence in $L$ in the RHS, provided
we replace one factor $\|\frac{f}{w_{t}}\|_{\ell^p}$ by $\|\frac{f}{w_{t}}\|_{\ell^\infty}$.

\begin{proposition}[Left boundary term, I]\label{prop:left}
Recall the weights $\cW_t$ and $w_t$ from \eqref{eq:weights}.
For any $h\ge 2$, $t \ge 0$, $L \in \N$ we have, for any 
$p \in (1,\infty)$ and $\mathscr{C}$ as in \eqref{eq:CbarC},
\begin{equation}\label{eq:est1inf}
\begin{split}
	\max_{I \ne *} \Big\Vert \widehat \sfq_{L}^{|f|,I} 
	\tfrac{1}{\cW_t} \Big\Vert_{\ell^p}
	&\le 4 \, \scrC^h \, L
	\, \bigg\|\frac{f}{w_{t}}\bigg\|_{\ell^\infty}
	\, \bigg\|\frac{f}{w_{t}}\bigg\|_{\ell^p}^{h-1} \,.
\end{split}
\end{equation}
More generally, for any $r \in [1,\infty]$ we have
(with $\frac{1}{0} := \infty$, $\frac{\infty}{\infty-1} := 1$)
\begin{equation}\label{eq:est1}
\begin{split}
	\max_{I \ne *} \Big\Vert \widehat \sfq_{L}^{|f|,I} 
	\tfrac{1}{\cW_t} \Big\Vert_{\ell^p}
	&\le 4\, \scrC^h \, \min\{\tfrac{r}{r - 1}, \tfrac{p}{p - 1}\}
	 \, L^{1-\frac{1}{r}}
	\, \bigg\|\frac{f}{w_{t}}\bigg\|_{\ell^r}
	\, \bigg\|\frac{f}{w_{t}}\bigg\|_{\ell^p}^{h-1} \,.
\end{split}
\end{equation}
\end{proposition}

We further improve the bound \eqref{eq:est1inf}
through a \emph{restricted weight} $\cV_s^I: (\Z^2)^h \to (0,\infty)$, defined
for a \emph{pair} $I \vdash \{1,\ldots, h\}$ and $s\ge 0$ by
\begin{equation}\label{eq:V}
	\cV^I_s(\bx) := w_s(x^a - x^b) = \rme^{-s|x^a-x^b|}
	\qquad \text{for } I = \{\{a,b\}, \{c\} \colon c \ne a,b\} \,.
\end{equation}
Note that $\big| |z^a-z^b| -  |x^a - x^b| \big| \le |z^a-x^a| + |x^b-z^b|$, therefore we can estimate
\begin{equation}\label{eq:ratioV}
	\frac{\cV^I_s(\bz)}{\cV^I_s(\bx)} \le \rme^{s|z^a-x^a| + s|z^b-x^b|} \,.
\end{equation}
In analogy with \eqref{eq:choiceoft}, we anticipate that we will take
\begin{equation} \label{eq:choiceofs}
	s = \tfrac{1}{\sqrt{L}} \,.
\end{equation}

\begin{proposition}[Left boundary term, II]\label{prop:left2}
For any $h\ge 3$, $t \ge 0$, $s \in (0,1]$, $L\in\N$
we have, for any $p \in (1,\infty)$
and $\mathscr{C}$ as in \eqref{eq:CbarC},
\begin{equation}\label{eq:est1+}
	\max_{\substack{J \, \text{pair} \\ I \ne *, \, I \not\supseteq J}} 
	\Big\Vert \widehat \sfq_{L}^{|f|,I} 
	\tfrac{\cV_s^J}{\cW_t} \Big\Vert_{\ell^p}
	\le 36^{\frac{1}{p}} \,\scrC^h \, \frac{L}{s^{\frac{2}{p}}} 
	\, \bigg\|\frac{f}{w_{t}}\bigg\|_{\ell^\infty}^2
	\, \bigg\|\frac{f}{w_{t}}\bigg\|_{\ell^p}^{h-2} \,,
\end{equation}
where $I \not\supseteq J$, for $I = \{I^1, \ldots, I^m\}$ and $J=\{\{a,b\},\{c\} : c \ne a, b\}$,
means $I^j \not\supseteq \{a,b\} \ \forall j$.
\end{proposition}

\smallskip

We next estimate the  \emph{right boundary term} which involves $\widebar \sfq_{L}^{|g|,J}$,
see \eqref{eq:Qright} and \eqref{eq:Qdef2}, obtaining estimates analogous to
\eqref{eq:est1} and \eqref{eq:est1+}.

\begin{proposition}[Right boundary term]\label{prop:right}
For any $h\ge 2$, $t \ge 0$, $L \in \N$ we have,
for any $q \in (1,\infty)$ and $\overline{\mathscr{C}}$ as in \eqref{eq:CbarC},
\begin{equation}\label{eq:est4-new}
\begin{split}
	\max_{J \ne *} \big\Vert \overline{\sfq}_{L}^{|g|,J} \, \cW_t \big\Vert_{\ell^q}
	& \le \tfrac{q}{q-1} \, \overline{\scrC}^h \,
	\, \|g \, w_t \|_{\ell^{2q}}^2 \, \|g \, w_t \|_{\ell^q}^{h-2} \\
	& \le \tfrac{q}{q-1} \, \overline{\scrC}^h \,
	\, \|g \, w_t \|_{\ell^\infty} \, \|g \, w_t \|_{\ell^q}^{h-1}  \,.
\end{split}
\end{equation}
Moreover, for any $h\ge 3$, $s \in (0,1]$
we have, for $\overline{\mathscr{C}}$ as in \eqref{eq:CbarC},
\begin{equation}\label{eq:est4+-new}
\begin{split}
	\max_{\substack{I \, \text{pair}\\J \ne *,\, J \not\supseteq I}}
	\big\Vert \overline{\sfq}_{L}^{|g|,J} \, \cW_t \, \cV_s^I \big\Vert_{\ell^q}
	& \le \tfrac{q}{q-1} \, \overline{\scrC}^h \,
	\, \frac{1}{s^{\frac{2}{q}}} \, \|g \, w_t \|_{\ell^\infty}^2 \, \|g \, w_t \|_{\ell^q}^{h-2} \,.
\end{split}
\end{equation}
where $J \not\supseteq I$, for $J = \{J^1, \ldots, J^m\}$ and $I=\{\{a,b\},\{c\} : c \ne a, b\}$,
means $J^i \not\supseteq \{a,b\}\ \forall i$.
\end{proposition}

\begin{remark}
We can bound $\|g \, w_t \|_{\ell^\infty} \le \|g \|_{\ell^\infty}
\, \|w_t \|_{\ell^\infty}$ and
$\|g \, w_t \|_{\ell^q} \le \|g \|_{\ell^\infty} \, \|w_t \|_{\ell^q}$.
By a direct computation, see \eqref{eq:easybound}, we have
\begin{equation} \label{eq:norms-wt}
	\| w_t \|_{\ell^\infty} = 1 \,, \qquad
	\| w_t \|_{\ell^q} = \bigg(\sum_{z\in\Z^2} \rme^{-q t |z|}\bigg)^{\frac{1}{q}} 
	\le \frac{36^{\frac{1}{q}}}{t^{\frac{2}{q}}} \,,
\end{equation}
therefore we obtain from \eqref{eq:est4-new}
\begin{equation}\label{eq:est4}
\begin{split}
	\max_{J \ne *} \big\Vert \widebar{\sfq}_{L}^{|g|,J}
	\, \cW_t \big\Vert_{\ell^q} \le  
	\tfrac{q}{q-1} \, \big( 36^{\frac{1}{q}} \, \overline{\scrC}\,\big)^h \, 
	\, \frac{\|g\|_{\ell^\infty}^h}{t^{\frac{2}{q}(h-1)}} \,.
\end{split}
\end{equation}
Similarly, from \eqref{eq:est4+-new} we deduce that
\begin{equation}\label{eq:est4+}
	\max_{\substack{I\,\text{pair}\\ J \ne *,\, J \not\supseteq I}}
	\big\Vert \overline{\sfq}_{L}^{|g|,J}
	\, \cW_t \, \cV_s^I\big\Vert_{\ell^q}\le 
	\,
	\tfrac{q}{q-1} \, \big( 36^{\frac{1}{q}} \, \overline{\scrC}\,\big)^h \, 
	\frac{\|g\|^h_{\ell^\infty}}{s^{\frac{2}{q}}\,t^{\frac{2}{q}(h-2)}} \,.
\end{equation}
\end{remark}

\smallskip
\subsection{Bulk terms}
\label{sec:4th-bulk}

In this section we estimate the the \emph{bulk terms} appearing in \eqref{eq:Xibulk},
i.e.\ 
$\big\| \widehat \sfQ_{L} \big\|_{\ell^q \to \ell^q}^{\cW}$
and
$\big\| |\widehat \sfU|_{L,\lambda, \beta}\big\|_{\ell^q \to \ell^q}^{\cW}$
from \eqref{eq:Crho1}-\eqref{eq:Crho2}.
The proofs are also given in Appendix~\ref{app:boundary}.

\smallskip

We recall the weights $\cW_t$ and $\cV_s^I$, see \eqref{eq:weights} and \eqref{eq:V}.
We will choose the parameters $t, s = O(\frac{1}{\sqrt{L}})$, see \eqref{eq:choiceoft}
and \eqref{eq:choiceofs}, hence
\emph{the following constants are uniformly bounded}:
\begin{equation}\label{eq:hatcheckC}
\begin{gathered}
	\widehat{\mathscr{C}\,} := 4000 \, \sfc^2 \, \rme^{8 \sfc\, t^2  L} \,, \qquad
	\widehat{\widehat{\mathscr{C}\,}} := 4000 \, \sfc^2 \, \rme^{8 \sfc\, (t+2s)^2  L} \,, \\
	\widecheck{\mathscr{C}\,} := 2 \, \rme^{4 \sfc \, t^2 L} \,, \qquad
	\widecheck{\widecheck{\mathscr{C}\,}} := 2 \, \rme^{4 \sfc \, (t+s)^2 L} \,.\
\end{gathered}
\end{equation}
We first estimate the ``bulk random walk term'' 
which involves  $\widehat \sfQ_{L}^{I,J}$,
see \eqref{eq:Crho1}.

\begin{proposition}[Bulk random walk term]\label{prop:bulk-rw}
For any $h\ge 2$, $t \ge 0$, $L\in\N$ we have, for any
$q \in (1,\infty)$ and $\widehat{\mathscr{C}\,}$
from \eqref{eq:hatcheckC}, 
\begin{equation}\label{eq:est2}
	\big\| \widehat \sfQ_{L} \big\|_{\ell^q \to \ell^q}^{\cW_t} :=
	\max_{I, J \ne *, \, I \ne J} \, \big\Vert \cW_t \, \widehat \sfQ_{L}^{I,J} 
	\tfrac{1}{\cW_t} \big\Vert_{\ell^q \to \ell^q}
	\le h! \, \widehat{\mathscr{C}\,}^h \, q \, \tfrac{q}{q-1} \,.
\end{equation}
Moreover, for $s \ge 0$  and $\widehat{\widehat{\mathscr{C}\,}}$ from \eqref{eq:hatcheckC},
\begin{equation}\label{eq:est2V}
	\max_{I, J \, \text{pairs}, \, I \ne J} \, \big\Vert \tfrac{\cW_t}{\cV_s^J} 
	\, \widehat \sfQ_{L}^{I,J} 
	\tfrac{1}{\cW_t \, \cV_s^I} \big\Vert_{\ell^q \to \ell^q}
	\le h! \, \widehat{\widehat{\mathscr{C}\,}}^h \, q \, \tfrac{q}{q-1} \,.
\end{equation}
(note that the weights $\cV_s^J, \cV_s^I$ appear \emph{in the denominator on both sides}).
\end{proposition}

\medskip

We next focus on the 
term $\big\| |\widehat \sfU|_{L,\lambda, \beta}\big\|_{\ell^q \to \ell^q}^{\cW}$
from \eqref{eq:Crho2}
which depends on the operator 
$|\widehat \sfU|_{L,\lambda, \beta}^{I}$, see \eqref{eq:sfU} and \eqref{eq:sfUalt}.
Recalling $R_N$ from \eqref{eq:RN} and $q_n(x)$ from \eqref{eq:defofq}, we define
\begin{equation}\label{eq:RNlambda}
	R_N^{(\lambda)} := \sum_{n=1}^N \rme^{-\lambda n} \, q_{2n}(0) \,,
\end{equation}
which reduces to $R_N$ for $\lambda = 0$.
In the next result we are going to assume that $|\bbE[\xi_\beta^I]| \le \sigma_\beta^2$
for any partition $I \ne *$,
which holds for $\beta > 0$ small enough (see Proposition~\ref{lem:dismom}).

\begin{proposition}[Bulk interacting term]\label{prop:bulk-interacting}
Let $\beta > 0$ satisfy $\max_{I\ne *} |\bbE[\xi_\beta^I]| \le \sigma_\beta^2$.
For any $h\ge 2$, $t \ge 0$, $L \in \N$, $\lambda \ge 0$
such that $\sigma_\beta^2 \, R_L^{(\lambda)} < 1$ we have,
for any $q \in (1,\infty)$ and $\widecheck{\mathscr{C}\,}$
from \eqref{eq:hatcheckC},
\begin{equation}\label{eq:est3}
\begin{split}
	\big\| |\widehat \sfU|_{L,\lambda, \beta}\big\|_{\ell^q \to \ell^q}^{\cW_t}
	:= \max_{I\ne *} \, \big\Vert \, \cW_t \, |\widehat \sfU|_{L,\lambda, \beta}^{I}
	\, \tfrac{1}{\cW_t} \, \big\Vert_{\ell^q \to \ell^q}
	&\le 1 + \widecheck{\mathscr{C}\,}^h
	\, \frac{\sigma_\beta^2 \, R_L^{(\lambda)}}
	{1- \sigma_\beta^2 \, R_L^{(\lambda)}} \,.
\end{split}
\end{equation}
Moreover, 
for any $s \ge 0$ we have,
for $\tau \in \{+1,-1\}$and $\widecheck{\widecheck{\mathscr{C}\,}}$
from \eqref{eq:hatcheckC},
\begin{equation}\label{eq:est3+}
	\begin{split}
	\max_{\substack{J \, \text{pair}\\ I\ne *}} \,
		\big\Vert \, (\cV^J_s)^{\tau} \, \cW_t \, |\widehat \sfU|_{L,\lambda, \beta}^{I}
	\, \tfrac{1}{\cW_t \, (\cV^J_s)^{\tau} } \, \big\Vert_{\ell^q \to \ell^q}
		&\le 1 + \widecheck{\widecheck{\mathscr{C}\,}}^h
		\, \frac{\sigma_\beta^2 \, R_L^{(\lambda)}}
		{1- \sigma_\beta^2 \, R_L^{(\lambda)}}  \,.
	\end{split}
\end{equation}
\end{proposition}

\smallskip

\section{Proof of Proposition~\ref{prop:stimamomento}} 
\label{sec:4th-3}

In this section we prove Proposition~\ref{prop:stimamomento}.
The key difficulty is that our goal \eqref{eq:stimamomento}
involves 
\emph{the (optimal) $1/M^2$ dependence} on
the width of the time interval $(\frac{i-1}{M}N, \frac{i}{M}N]$
(recall the definition \eqref{eq:Z(i)} of the random variable $X_{N, M}^{(i)}$).
This requires sharp ad hoc estimates.

\subsection{Setup}
By formula \eqref{eq:backtous} from Remark~\ref{rem:backtous},
for $l = 1,\ldots, M$ we can write
\begin{equation}\label{eq:backtous2}
	\mathbb E\left [(X_{N,M}^{(l)})^4\right ] = 
	\frac{\theta_N^2}{N^4} \, \cM_{L,\beta}^{4}(f,g)
\end{equation}
where $L, \beta, f, g$ are given as follows:
\begin{equation}\label{eq:parameters}
	L = \frac{N}{M} \,, \qquad \beta = \beta_N \ \text{ in } (\ref{eq:quasicrit}) \,, \qquad
	f(\cdot) = q_{\frac{l-1}{M}N}^{\varphi_N}(\cdot) \ \text{ in  (\ref{eq:phiN})-(\ref{eq:kernel1})} \,, 
	\qquad g(\cdot) \equiv 1 \,.
\end{equation}
We can bound $\cM^4_{\frac{N}{M},\beta_N}(f,g)$ 
exploiting \eqref{eq:MNh5} for $h=4$ and $\lambda = 0$, which yields
\begin{equation}
	\label{eq:startfromthis}
	\begin{aligned}
		\mathbb E\big [(X_{N,M}^{(l)})^4\big ] 
		\le \frac{\theta_N^2}{N^4} \, \bigg(\,
		\Xi(1) + \Xi(2) + \sum_{r=3}^\infty \Xi(r) \, \bigg) \,,
	\end{aligned}
\end{equation}
where $\Xi(r)$ is defined in \eqref{eq:Xir}.
We show that the only non-negligible term in \eqref{eq:startfromthis} 
is $\Xi(2)$: 
more precisely, we will prove that 
there is $C < \infty$ such that, for any  $M \in \N$, 
\begin{equation}\label{eq:XI-2}
	\begin{split}
		\limsup_{N\to\infty} \, \frac{\theta_N^2}{N^4} \,\Xi(2) \le \frac{C}{M^2} \,,
	\end{split}
\end{equation}
while
\begin{equation}\label{eq:XI-13}
	\begin{aligned}
		\lim_{N \to \infty}\, \frac{\theta_N^2}{N^4} \,\Xi(1) = 0 
		\qquad \text{and} \qquad\lim_{N \to \infty} \, \frac{\theta_N^2}{N^4} \, 
		\sum_{r=3}^\infty \Xi(r) = 0 \,.
	\end{aligned}
\end{equation}
This will complete the proof of Proposition~\ref{prop:stimamomento}.

We estimate $\Xi(r)$ exploiting the bound \eqref{eq:MNh6}-\eqref{eq:Xibulk} with the choice
\begin{equation*}
	p = q = 2 \,.
\end{equation*}
We need to control the \emph{boundary terms} and the \emph{bulk terms}, see \eqref{eq:terms}.
We recall that the weights $\cW_t$ and $\cV_s^I$ are defined in
\eqref{eq:weights} and \eqref{eq:V}, and we fix
\begin{equation}\label{eq:st}
	t = \tfrac{1}{\sqrt{N}} \,, \qquad s = \tfrac{1}{\sqrt{L}} = \sqrt{\tfrac{M}{N}} \,.
\end{equation}

For notational lightness,
we write $a \lesssim b$ whenever $a \le C \, b$ for some constant $0<C < \infty$.
We also denote by $\|\varphi\|_p := (\int_{\R^2} \varphi(x)^p \, \dd x)^{1/p}$ the usual
$L^p$ norm of a function $\varphi: \R^2 \to \R$.

\smallskip

\subsection{Boundary terms}
We estimate the \emph{left boundary term}
$\big\Vert \widehat \sfq_{L}^{|f|,I} \frac{1}{\cW_t} \big\Vert_{\ell^2}$
applying \eqref{eq:est1inf}.
We recall from \eqref{eq:parameters} that
$f(\cdot)=q^{\varphi_N}_{\frac{l-1}{M}N}(\cdot)$ for $1 \le l \le M$.
Let us estimate $\|\frac{f}{w_{t}}\|_{\ell^\infty}$ and
$\|\frac{f}{w_{t}}\|_{\ell^2}$, starting from the former.
By \eqref{eq:weighted-averaged-rw2},
for $l \le M$ and $t = \frac{1}{\sqrt{N}}$ we have
\begin{equation*}
	\bigg\|\frac{f}{w_{t}}\bigg\|_{\ell^\infty} 
	\le \sfc \, \rme^{2\sfc \, t^2 \frac{l-1}{M}N}
	\, \bigg\|\frac{\varphi_N}{w_{t}}\bigg\|_{\ell^\infty} 
	\le \sfc \, \rme^{2\sfc}
	\, \bigg\|\frac{\varphi_N}{w_{t}}\bigg\|_{\ell^\infty} \,.
\end{equation*}
Since $\varphi$ is compactly supported, say in a ball $B(0,R)$,
we have that $\varphi_N$ is supported in $B(0,R\sqrt{N} + \sqrt{2})
\subseteq B(0,2R\sqrt{N})$, see \eqref{eq:phiN}.
By $w_t(x) = \rme^{-t|x|}$, we then obtain
\begin{equation}\label{eq:stimetta::}
\begin{split}
	\bigg\|\frac{\varphi_N}{w_{t}}\bigg\|_{\ell^\infty} \le 
	\rme^{t \, 2 R \sqrt{N}} \, \big\|\varphi_N \big\|_{\ell^\infty} 
	\le \rme^{2R} \, \|\varphi\|_{\infty} \lesssim 1 \,,
	\qquad \text{hence} \qquad
	\bigg\|\frac{f}{w_{t}}\bigg\|_{\ell^\infty}  \lesssim 1 \,,
	\end{split}
\end{equation}
because $\|\varphi_N \|_{\ell^\infty} \le \| \varphi \|_{\infty}$.
We next estimate $\|\frac{f}{w_{t}}\|_{\ell^2}$.
By a Riemann sum approximation, we see from  \eqref{eq:phiN} that $\|\varphi_N\|_{\ell^2}
\lesssim \sqrt{N} \, \|\varphi\|_2$, hence by \eqref{eq:weighted-averaged-rw2} 
we obtain
\begin{equation}\label{eq:stimetta::2}
	\bigg\|\frac{f}{w_{t}}\bigg\|_{\ell^2} 
	\le \sfc \, \rme^{2\sfc}
	\, \bigg\|\frac{\varphi_N}{w_{t}}\bigg\|_{\ell^2} 
	\le \sfc \, \rme^{2\sfc} \, \rme^{2R}
	\, \big\|\varphi_N \big\|_{\ell^2}
	\lesssim \sqrt{N} \,.
\end{equation}
We can finally apply the estimate \eqref{eq:est1inf} for $p=2$ and $h=4$ to get,
since $L = \frac{N}{M}$,
\begin{equation}\label{eq:stimaperLB}
	\max_{I \ne *} \Big\Vert \widehat \sfq_{L}^{|f|,I}  \tfrac{1}{\cW_t} \Big\Vert_{\ell^2}
	\le 4 \, \scrC^h \, L \, \bigg\|\frac{f}{w_{t}}\bigg\|_{\ell^\infty}
	\, \bigg\|\frac{f}{w_{t}}\bigg\|_{\ell^2}^{3}
	\lesssim \frac{N^{\frac{5}{2}}}{M} \,.
\end{equation}

We now estimate the \emph{right boundary term}
$\big\| \overline{\sfq}_{L}^{|g|,J} \, \cW_t \big\|_{\ell^2}$:
applying \eqref{eq:est4} for $q=2$ and $h=4$,
since $g \equiv 1$ and $t = \frac{1}{\sqrt{N}}$, we obtain
\begin{equation}
	\label{eq:stimaperRB}
	\max_{J \ne *} \big\| \overline{\sfq}_{L}^{|g|,J}
	\, \cW_t \big\|_{\ell^2} 
	\le \big( 12 \, \overline{\mathscr{C}} \big)^4 \, \frac{\|g\|_\infty^4}{t^3}
	\lesssim \, N^{\frac{3}{2}}\,.
\end{equation}
Overall, 
we have shown that
\begin{equation}\label{eq:mezza}
	\Big( \max_{I \ne *} \big\Vert \widehat \sfq_{L}^{|f|,I} \tfrac{1}{\cW} \big\Vert_{\ell^p} \Big) 
	\, \Big( \max_{J \ne *} \big\Vert \cW \,
	\widebar{\sfq}_{L}^{|g|,J} \big\Vert_{\ell^q} \Big)
	\lesssim \frac{N^4}{M} \,.
\end{equation}
In view of \eqref{eq:MNh6}, it remains to estimate $\Xi^\bulk(r)$ defined in \eqref{eq:Xibulk}.

\smallskip

\subsection{Bulk terms}
We next estimate the \emph{bulk terms}, see \eqref{eq:Crho1}-\eqref{eq:Crho2}.
For the first bulk term, see \eqref{eq:Crho1}, we apply
directly the estimate \eqref{eq:est2} with $q=2$ and $h=4$ to get
\begin{equation}
	\label{eq:stimaperQIJ} 
	\big\| \widehat \sfQ_{L} \big\|_{\ell^q \to \ell^q}^{\cW_t} =
	\max_{I, J \ne *, \, I \ne J} \, \big\Vert \cW_t \, \widehat \sfQ_{L}^{I,J} 
	\tfrac{1}{\cW_t} \big\Vert_{\ell^2 \to \ell^2}
	\le 4! \, \widehat{\mathscr{C}\,}^4 \, 4 \lesssim 1 \,.
\end{equation} 
(Also note that $\big\| \widehat \sfQ_{L} \big\|_{\ell^q \to \ell^q}^{\cW_t}
\ge \cW_t(0) \, \widehat \sfQ_{L}^{I,J}(0,0)
\tfrac{1}{\cW_t(0)} \ge \sfQ_2(0,0) \gtrsim 1$.)

\smallskip

We then focus on the second term, see \eqref{eq:Crho2}. 
For $L = \frac{N}{M} \le N$ and $\beta = \beta_N$ as in \eqref{eq:quasicrit}
\begin{equation}\label{eq:conticonsigma}
	1-\sigma_{\beta_N}^2 \, R_L \ge 1- \sigma_{\beta_N}^2 \, R_N \ge \frac{\theta_N}{\log N}
	> 0 \,, \qquad \text{in particular} \quad
	\sigma_{\beta_N}^2 \, R_L < 1 \,.
\end{equation}
Then by \eqref{eq:est3}  with $\lambda = 0$ (so that $R_N^{(\lambda)} = R_N$) we obtain,
recalling that $\theta_N \ll \log N$,
\begin{equation}\label{eq:est3-us}
	\big\| |\widehat \sfU|_{L,\lambda, \beta}\big\|_{\ell^q \to \ell^q}^{\cW_t}
	= \max_{I\ne *} \big\Vert \, \cW_t \, |\widehat \sfU|_{L,\lambda, \beta}^{I}
	\, \tfrac{1}{\cW_t} \, \big\Vert_{\ell^2 \to \ell^2}
	\le 1 + \widecheck{\mathscr{C}\,}^4  \frac{ \sigma_{\beta_N}^2 \, R_L}
	{1- \sigma_{\beta_N}^2 \, R_L} 
	\lesssim \frac{\log N}{\theta_N}  \,.
\end{equation}
Since $\beta_N \to 0$, the bound \eqref{eq:dismom} ensures that
$|\bbE[\xi_{\beta_N}^I]| = O(\sigma_{\beta_N}^2) \le O(\frac{1}{R_N}) = O( \frac{1}{\log N})$ 
for any $I\ne *$ and $N$ large,
therefore there is $C < \infty$ such that
\begin{equation}\label{eq:geom}
	\Big( \max_{I \ne *} \, \big|\bbE[\xi_{\beta_N}^I]\big|  \Big)
	\,\, \big\| \widehat \sfQ_{L} \big\|_{\ell^q \to \ell^q}^{\cW_t}
	\, \, \big\| |\widehat \sfU|_{L,\lambda, \beta}\big\|_{\ell^q \to \ell^q}^{\cW_t}
	\,\le\, \frac{C}{\theta_N} \,.
\end{equation}

\smallskip

\subsection{Terms $r \ge 3$}

We are ready to prove the second relation in \eqref{eq:XI-13}, 
which shows that the terms 
$r \ge 3$
give a negligible contributions to $\mathbb E\big [(X_{N,M}^{(l)})^4\big ]$,
recall \eqref{eq:startfromthis}.

Let us denote by $c(h) \in \N$ the number 
of partitions $I \vdash \{1,\ldots, h\}$ with $I \ne *$. Then by \eqref{eq:Xibulk}
we have the geometric bound
\begin{equation*}
	\Xi^\bulk(r) \le  \big( 
	\big\| \widehat \sfQ_{L} \big\|_{\ell^q \to \ell^q}^{\cW_t}
	\big)^{-1}
	\, \Big\{ c(h) \, \Big( \max_{I \ne *} \, \big|\bbE[\xi_{\beta_N}^I]\big|  \Big)
	\,\, \big\| \widehat \sfQ_{L} \big\|_{\ell^q \to \ell^q}^{\cW_t}
	\, \, \big\| |\widehat \sfU|_{L,\lambda, \beta}\big\|_{\ell^q \to \ell^q}^{\cW_t}
	\Big\}^r \,,
\end{equation*}
and note that the term in brackets is $< \frac{1}{2}$ for large~$N$, by \eqref{eq:geom}
and $\theta_N \to \infty$, therefore
\begin{equation*}
	\sum_{r=3}^\infty \Xi^\bulk(r) \lesssim \, \Xi^\bulk(3)
	\lesssim \frac{1}{\theta_N^3} \,.
\end{equation*}
Applying \eqref{eq:MNh6} and \eqref{eq:mezza}, we then obtain
the second relation in \eqref{eq:XI-13}:
\begin{equation*}
	\frac{\theta_N^2}{N^4} \, \sum_{r=3}^\infty \Xi(r) \le
	\frac{\theta_N^2}{M} \, \sum_{r=3}^\infty \Xi^\bulk(r) \lesssim
	\frac{1}{M \, \theta_N} \, \xrightarrow[\,N\to\infty\,]{} 0 \,.
\end{equation*}

\begin{remark}
The same arguments can be applied to show that  in the quasi-critical regime,
the contribution of the terms $r > \big \lfloor \frac{h}{2} \big \rfloor$
for the $h$-th moment of $X^{(l)}_{N,M}$
is negligible as $N\to\infty$.
\end{remark}

\subsection{Term $r = 1$}

We now prove the first relation in \eqref{eq:XI-13}.
A partition $I \vdash \{1,2,3,4\}$ with full support is either 
a double pair $I = \{\{a,b\}, \{c,d\} \}$ or the quadruple $I = \{1,2,3,4\}$,
hence $\bbE[\xi_{\beta_N}^{I}] \lesssim \sigma_{\beta_N}^4$ for large~$N$,
by \eqref{eq:Exi} and \eqref{eq:meanvarxi} (see also Proposition~\ref{lem:dismom}).
Then, by \eqref{eq:Xibulk},
\begin{equation*}
	\Xi^\bulk(1) \,=\, \!\!
	\sum_{\substack{I \vdash\{1, \ldots, h\} \\
	\text{with full support}}}
	\!\! \big| \bbE[\xi_{\beta_N}^{I}] \big| \, 
	\big\| |\widehat \sfU|_{L,\lambda, \beta}\big\|_{\ell^q \to \ell^q}^{\cW_t}
	\,\lesssim\, \sigma_{\beta_N}^4 \, 
	\big\| |\widehat \sfU|_{L,\lambda, \beta}\big\|_{\ell^q \to \ell^q}^{\cW_t}
	\,\lesssim \, \frac{1}{(\log N) \, \theta_N} \,,
\end{equation*}
where we applied \eqref{eq:est3-us} and $\sigma_{\beta_N}^2 \le \frac{1}{R_N} = O( \frac{1}{\log N})$.
Applying \eqref{eq:MNh6} and \eqref{eq:mezza}, and recalling that $\theta_N \ll \log N$, we obtain
the first relation in \eqref{eq:XI-13}:
\begin{equation*}
	\frac{\theta_N^2}{N^4} \, \Xi(1) \le
	\frac{\theta_N^2}{M} \, \Xi^\bulk(1) \lesssim
	\frac{\theta_N}{M \, \log N} \, \xrightarrow[\,N\to\infty\,]{} 0 \,.
\end{equation*}

\smallskip

\subsection{Term $r=2$} 

We finally prove \eqref{eq:XI-2}, which completes the proof of Proposition~\ref{prop:stimamomento}.
We recall that $\Xi(2)$, defined by \eqref{eq:Xir}, is a sum over two partitions
$I_1, I_2 \vdash \{1,\ldots, h\}$ with $I_1 \ne *$, $I_2 \ne *$ and $I_1 \ne I_2$.
We then split $\Xi(2) = \Xi_\pairs(2) + \Xi_\others(2)$ where:
\begin{itemize}
\item $\Xi_\pairs(2)$ is the contribution to \eqref{eq:Xir} when both $I_1, I_2$ are pairs;
\item $\Xi_\others(2)$ is the complementary contribution when $I_1$ and/or $I_2$ is not a pair.
\end{itemize}

We first focus on $\Xi_\others(2)$ and on the corresponding
quantity $\Xi_\others^\bulk(2)$, see \eqref{eq:Xibulk}. If either $I_1$ or $I_2$ is not a pair, 
by Proposition~\ref{lem:dismom} we can bound
$|\bbE[\xi_{\beta_N}^{I_1}] \, \bbE[\xi_{\beta_N}^{I_2}]| \lesssim \sigma_{\beta_N}^5$,
hence
\begin{equation*}
	\Xi_\others^\bulk(2) \lesssim  \sigma_{\beta_N}^5 \, 
	\big\| \widehat \sfQ_{L} \big\|_{\ell^q \to \ell^q}^{\cW_t}
	\, \big( 
	\big\| |\widehat \sfU|_{L,\lambda, \beta}\big\|_{\ell^q \to \ell^q}^{\cW_t}
	\big)^{2}
	\lesssim \frac{1}{(\log N)^{5/2}} \, \bigg( \frac{\log N}{\theta_N} \bigg)^2
	\lesssim \frac{1}{\theta_N^2 \, \sqrt{\log N}} \,,
\end{equation*}
where we applied \eqref{eq:stimaperQIJ}, \eqref{eq:est3-us} and
$\sigma_{\beta_N}^2 \le \frac{1}{R_N} = O( \frac{1}{\log N})$.
Then, by \eqref{eq:MNh6} and \eqref{eq:mezza},
\begin{equation*}
	\frac{\theta_N^2}{N^4} \, \Xi_\others(2) \le \frac{\theta_N^2}{M} \,  \Xi_\others^\bulk(2) 
	\lesssim \frac{1}{M \, \sqrt{\log N}} \xrightarrow[\,N\to\infty\,]{} 0 \,,
\end{equation*}
which shows that the contribution of $\Xi_\others(2)$ to \eqref{eq:XI-2} is negligible.

\smallskip

It only remains to focus on $\Xi_\pairs(2)$: since 
$\bbE[\xi_{\beta}^{I}] = \sigma_\beta^2$ when $I$ is a pair, we can write
\begin{equation*}
	\Xi_\pairs(2) := \!\! \sum_{\substack{I_1 \ne I_2 \vdash\{1, \ldots, h\} \\
	\text{pairs with full support} } }  
	\sigma_{\beta}^4 \,\,
	 \Big\langle  \widehat \sfq_{L}^{|f|,I_1} \,  , \,
	|\widehat \sfU|^{I_1}_{L, \lambda, \beta} \,\,
	\widehat \sfQ_{L}^{I_{1}, I_2} \,\,
	|\widehat \sfU|_{L, \lambda, \beta}^{I_2} \,\,
	\widebar{\sfq}_{L}^{|g|,I_r}  \Big\rangle \,.
\end{equation*}
Besides inserting $\frac{1}{\cW_t} \cW_t$ as above, 
we also insert $\cV_s^{I_2}\frac{1}{\cV_s^{I_2}}$ on the left of $\widehat \sfQ_{L}^{I_{1}, I_2}$
and $|\widehat \sfU|^{I_1}_{L, \lambda, \beta}$, while we insert
$\frac{1}{\cV_s^{I_1}}\cV_s^{I_1}$ on the right of $\widehat \sfQ_{L}^{I_{1}, I_2}$
and $|\widehat \sfU|^{I_2}_{L, \lambda, \beta}$ (recall \eqref{eq:V}): we thus obtain
\begin{equation} \label{eq:cinquepezzi}
\begin{split}
	\Xi_\pairs(2) \le & \! \sum_{\substack{I_1 \ne I_2 \vdash\{1, \ldots, h\} \\
	\text{pairs with full support} } }  \!\!\!
	\sigma_{\beta}^4 \,\,
	\Big\|  \widehat \sfq_{L}^{|f|,I_1} \tfrac{\cV_s^{I_2}}{\cW_t} \Big\|_{\ell^p} \,
	\Big\| \tfrac{\cW_t}{\cV_s^{I_2}}
	\, |\widehat \sfU|^{I_1}_{L, \lambda, \beta} \, \tfrac{\cV_s^{I_2}}{\cW_t} 
	\Big\|_{\ell^q \to \ell^q} \cdot \\
	& \, \cdot \,  \Big\| \tfrac{\cW_t}{\cV_s^{I_2}} \,
	\widehat \sfQ_{L}^{I_{1}, I_2} \, \tfrac{1}{\cW_t\,\cV_s^{I_1}} \Big\|_{\ell^q \to \ell^q} 
	\, \Big\| \cW_t \, \cV_s^{I_1} \,
	|\widehat \sfU|_{L, \lambda, \beta}^{I_2} \, \tfrac{1}{\cW_t\,\cV_s^{I_1}} \Big\|_{\ell^q \to \ell^q}
	\, \Big\| \cW_t\,\cV_s^{I_1} \, \widebar{\sfq}_{L}^{|g|,I_r}  \Big\|_{\ell^q} \,.
\end{split}
\end{equation}

It remains to estimate these norms. Let us recall that $h=4$,  $p=q=2$
and $t = \frac{1}{\sqrt{N}}$, $s = \frac{1}{\sqrt{L}}$, where $L = \frac{M}{N}$.
We start with the boundary terms:
\begin{itemize}
\item applying the estimate \eqref{eq:est1+}, in view of \eqref{eq:stimetta::}-\eqref{eq:stimetta::2},
we improve the estimate \eqref{eq:stimaperLB}:
\begin{equation}\label{eq:stimaperLB+}
	\max_{\substack{I,J \, \text{pairs} \\ I \ne J}} \,
	\Big\Vert \widehat \sfq_{L}^{|f|,I} 
	\tfrac{\cV_s^J}{\cW_t} \Big\Vert_{\ell^2}
	\le 6 \, \scrC^4 \, \frac{L}{s} \, \bigg\|\frac{f}{w_{t}}\bigg\|_{\ell^\infty}^2
	\, \bigg\|\frac{f}{w_{t}}\bigg\|_{\ell^2}^{2}
	\lesssim L^{\frac{3}{2}} \, N
	\lesssim \frac{N^{\frac{5}{2}}}{M^{\frac{3}{2}}} \,;
\end{equation}

\item applying the estimate \eqref{eq:est4+}, since $g \equiv 1$,
we improve the estimate \eqref{eq:stimaperRB}:
\begin{equation}\label{eq:stimaperRB+}
	\max_{\substack{I,J\,\text{pairs}\\ I \ne J}} \,
	\big\Vert\, \cW_t \, \cV_s^I \, \overline{\sfq}_{L}^{|g|,J}
	\big\Vert_{\ell^2}\le \big( 12
	\, \overline{\scrC}\,\big)^4  \,
	\frac{\|g\|^4_{\ell^\infty}}{s \,t^{2}} 
	\lesssim \sqrt{L} \, N
	\lesssim \frac{N^{\frac{3}{2}}}{\sqrt{M}} \,.
\end{equation}
\end{itemize}
Overall, the product of the two boundary terms is $\lesssim \frac{N^4}{M^2}$, which
improves on the previous estimates by an essential factor $\frac{1}{M}$,
thanks to the use of the restricted weight $\cV_s^I$.

We next estimate the bulk terms:
\begin{itemize}
\item applying \eqref{eq:est2V} with $p=q=2$ and $h=4$, we obtain
an analogue of \eqref{eq:stimaperQIJ}:
\begin{equation} \label{eq:stimaperQIJ+} 	
	\max_{\substack{I, J \, \text{pairs}\\ I \ne J}} \, \big\Vert \tfrac{\cW_t}{\cV_s^J} 
	\, \widehat \sfQ_{L}^{I,J} \, \tfrac{1}{\cW_t \, \cV_s^I} \big\Vert_{\ell^2 \to \ell^2}
	\le 4! \, \widehat{\widehat{\mathscr{C}\,}}^4 \, 4 \lesssim 1 \,;
\end{equation} 

\item applying \eqref{eq:est3+} for both $\tau = +1$ and $\tau = -1$,
we obtain an analogue of \eqref{eq:est3-us}:
\begin{equation}\label{eq:est3-us+}
	\max_{I, J \, \text{pairs}} \, \big\Vert \,
	(\cV_s^J)^{\tau} \, \cW_t \, |\widehat \sfU|_{L,\lambda, \beta}^{I}
	\, \tfrac{1}{\cW_t \, (\cV_s^J)^{\tau}} \, \big\Vert_{\ell^2 \to \ell^2}
	\le 1 + \widecheck{\widecheck{\mathscr{C}\,}}^4  \frac{ \sigma_{\beta_N}^2 \, R_L}
	{1- \sigma_{\beta_N}^2 \, R_L} 
	\lesssim \frac{\log N}{\theta_N}  \,.
\end{equation}
\end{itemize}
Plugging the previous estimates into \eqref{eq:cinquepezzi},
since $\sigma_{\beta_N}^2 \le \frac{1}{R_N} = O(\frac{1}{\log N})$, we finally obtain
\begin{equation*}
	\Xi_\pairs(2) \lesssim \frac{1}{(\log N)^2} \, 
	\frac{N^{\frac{5}{2}}}{M^{\frac{3}{2}}}
	\, \bigg(\frac{\log N}{\theta_N}\bigg)^2
	\, \frac{N^{\frac{3}{2}}}{\sqrt{M}} = \frac{N^4}{M^2 \, \theta_N^2} \,,
\end{equation*}
which completes the proof of \eqref{eq:XI-2}, hence of Proposition~\ref{prop:stimamomento}.\qed

\appendix

\section{Some technical proofs}
\label{sec:4thapp}

We give the proof of Theorem~\ref{th:moments-exact}.
We recall that the averaged partition function $\cZ_{L,\beta}^{\omega}(f,g)$ is defined
in \eqref{eq:Zgen}-\eqref{eq:Zgenav}.
In analogy with \eqref{eq:partfunc} and \eqref{eq:Zexplicit},
by \eqref{eq:Zgen}-\eqref{eq:Zgenav} we can write
\begin{equation}\label{eq:partfuncgen}
\begin{split}
	\cZ_{L,\beta}^{\omega}(f,g) - \bbE[\cZ_{L,\beta}^{\omega}(f,g)]
	= \, & \sum_{k=1}^\infty 
	\ \sum_{\substack{0<n_1<\ldots<n_k < L\\
	x_1,\ldots,x_k \in \Z^2}} \, q_{n_1}^{f}(x_1) \, \xi_\beta(n_1, x_1) \times \\
	& \times \bigg\{ \prod_{j=2}^{k} q_{n_j-n_{j-1}}(x_j-x_{j-1}) \, \xi_\beta(n_j,x_j)
	\bigg\} \, q_{L-n_k}^{g}(x_k)\,,
\end{split}
\end{equation}
where we recall the random walk kernels \eqref{eq:defofq} and 
\eqref{eq:kernel1}. Recalling \eqref{eq:Mkdef}, we obtain
\begin{equation}\label{eq:fold}
	\begin{aligned}
	\cM_{L,\beta}^{h}(f,g) = \bbE \Bigg [ \Bigg ( \sum_{k=1}^\infty 
	& \ \sum_{\substack{0<n_1<\ldots<n_k < L\\
	x_1,\ldots,x_k \in \Z^2}} \, q_{n_1}^{f}(x_1)\, \xi_\beta(n_1, x_1) \times \\
	& \qquad \times \bigg\{ \prod_{j=2}^{k} q_{n_j-n_{j-1}}(x_j-x_{j-1}) \, \xi_\beta(n_j,x_j)
	\bigg\} \, q_{L-n_k}^{g}(x_k) \Bigg)^{\!\!h} \, \Bigg] \,.
	\end{aligned}
\end{equation}

When we expand the $h$-th power, we obtain 
a sum over $h$ families of space-time points
$A_i := \{(n^i_1, x^i_1), \ldots, (n^i_{k_i}, x^i_{k_i})\}$ for $i=1, \ldots, h$.
These points must \emph{match at least in pairs}, 
i.e.\ any point $(n^i_\ell,x^i_\ell)$ in any family $A_i$
must coincide with at least another point $(n^j_m, x^j_m)$ in a different family
$A_j$ for $j \ne i$,
otherwise the expectation
vanishes (since $\xi_{\beta}(n,x)$ are independent and centered).
In order to handle this constraint, following \cite[Theorem 6.1]{CSZ23},
we rewrite \eqref{eq:fold} by first \emph{summing over the set of all space-time points}
\begin{equation*}
	A := \bigcup_{i=1}^h A_i = \bigcup_{i=1}^h \{(n^i_1, x^i_1), \ldots, (n^i_{k_i}, x^i_{k_i})\}
	\subseteq \N \times \Z^2
\end{equation*}
and then specifying \emph{which families} each point $(n,x) \in A$ belongs to.

Let us fix the \emph{time coordinates} $n_1 < \ldots < n_r$ of the points in $A$. 
For each such time $n \in \{n_1, \ldots, n_r\}$, we have $(n,x) \in A$
for one or more $x \in \Z^2$ (there are at most $h/2$ such~$x$,
by the matching constraint described above).
We then make the following observations:
\begin{itemize}
\item if $(n,x) = (n^i_j, x^i_j)$ belongs to the family~$A_i$, then
we have in \eqref{eq:fold} the product of a random walk kernel ``entering'' $(n,x)$
and another one ``exiting'' $(n,x)$:
\begin{equation*}
	q_{n - n^i_{j-1}}(x - x_{j-1}^i) \cdot q_{n^i_{j+1} - n}(x_{j+1}^i - x) \,;
\end{equation*}

\item if $(n,x)$ does \emph{not} belong to the family $A_i$, then
we have in \eqref{eq:fold} a random walk kernel ``jumping over time~$n$'',
say $q_{n^i_j - n^i_{j-1}}(x_j - x_{j-1})$ with $n^i_{j-1} < n < n^i_j$: we
can split this kernel at time~$n$ by Chapman-Kolmogorov, writing
\begin{equation} \label{eq:CK}
	q_{n^i_j - n^i_{j-1}}(x^i_j - x^i_{j-1}) = \sum_{z\in\Z^2}
	q_{n - n^i_{j-1}}(z - x^i_{j-1}) \cdot q_{n^i_j - n}(x^i_j - z) \,.
\end{equation}
\end{itemize}
Then, to each time $n \in \{n_1, \ldots, n_r\}$, we can associate
a vector $\by = (y^1, \ldots, y^h) \in (\Z^2)^h$ with $h$ space coordinates,
where
$y^i = x$ if the family $A^i$ contains $(n,x)$ and $y^i = z$ from \eqref{eq:CK} otherwise.
The constraint that a point $(n,x) \in A$ belongs to two families
$A^i$ and $A^{i'}$
means that the corresponding coordinates of the vector $\by$ must coincide:
$y^i = y^{i'}$.  In order to specify which families $A^i$
share the same points, we assign a \emph{partition} $I 
\vdash \{1,\ldots, h\}$ to each time $n \in \{n_1, \ldots, n_r\}$
and we require that $\by \sim I$, see \eqref{eq:xsimI}.

We are now ready to provide a convenient rewriting of \eqref{eq:fold}
by first summing over the number $r \ge 1$ and the time coordinates
$n_1 < \ldots < n_r$, then on the corresponding space coordinates $\by_1, \ldots, \by_r$
and partitions $I_1, \ldots, I_r \vdash \{1,\ldots, h\}$ with $\by_i \sim I_i$.
Recalling the definitions of $\sfQ_n^{I,J}$ and $\sfq_n^{f,J}$ from \eqref{eq:Qdef2},
we can rewrite \eqref{eq:fold} as follows:
\begin{equation}\label{eq:MNh3}
\begin{aligned}
	\cM_{L,\beta}^{h}(f,g)
	= \ \sum_{r=1}^\infty \
	\sum_{\substack{0 < n_1<\cdots <n_r < L \\
	\by_1, \ldots, \by_r \in (\Z^2)^h}} \ &
	\sum_{\substack{I_1, \ldots, I_r \vdash\{1, \ldots, h\} \\
	\text{with full support} \\
	\text{and } I_i \ne * \ \forall i}}  \
	\sfq_{n_1}^{f,I_1} (\by_1) \, \bbE[\xi_{\beta}^{I_1}] \, \times \\
	& \times \bigg\{ \prod_{i=2}^r \sfQ_{n_i - n_{i-1}}^{I_{i-1}, I_i}(\by_{i-1}, \by_i)
	\,  \mathbb E[\xi_{\beta}^{I_i}] \bigg\}
	\, \sfq_{L - n_r}^{g,I_r} (\by_r) \,.
\end{aligned}
\end{equation}

Finally, formula \eqref{eq:MNh4} follows from \eqref{eq:MNh3}
grouping together stretches of \emph{consecutive repeated partitions},
i.e.\ when $I_i = J$ for consecutive indexes~$i$. 
The kernel $\sfU_{m-n,\beta}^J(\bz,\bx)$ from \eqref{eq:sfU}
does exactly this job, which leads to \eqref{eq:MNh4}.\qed

\begin{remark}
Formula \eqref{eq:MNh4} still contains the product of $\bbE[\xi_\beta^{I_i}]$ because
these factors from \eqref{eq:MNh3} are only partially absorbed in $\sfU_{m-n,\beta}^J(\bz, \bx)$:
indeed, in \eqref{eq:sfU} we have $k+1$ points 
$n_0 < n_1 < \ldots < n_k$, but the factor $\bbE[\xi_\beta^J]$ therein is only raised to the power~$k$.
\end{remark}

\section{Random walk bounds}
\label{sec:rwapp}

In this section we prove the random walk bounds from
Lemmas~\ref{lem:weighted-rw}, \ref{lem:weighted-averaged-rw} and~\ref{lem:max-averaged-rw}.
We also prove a heat kernel bound, see Lemma~\ref{lem:heat-kernel-bound} below.

\subsection{Proof of Lemma~\ref{lem:weighted-rw}}
We prove each of the four bounds in \eqref{eq:rw-weight}-\eqref{eq:rw-weight+}
for a different constant $\sfc$ (it then suffices to take the
maximal value).

The first bound in \eqref{eq:rw-weight} with $\sfc = c$ follows by \eqref{eq:sub-Gaussian},
thanks to the independence of the increments of the random walk.
This directly implies the first bound in \eqref{eq:rw-weight+}: it suffices to estimate 
$\sum_{x\in\Z^2} \rme^{t |x|} \, q_n(x) \le \sum_{x\in\Z^2} \rme^{2t |x^1|} \, q_n(x)$
(by $|x| \le |x^1| + |x^2|$, Cauchy-Scwharz and symmetry) and then
$\rme^{|z|} \le \rme^{z} + \rme^{-z}$, hence
$\sum_{x\in\Z^2} \rme^{t |x|} \, q_n(x) \le 2 \, \rme^{2 \sfc t^2 n}$.

To get the second bound in \eqref{eq:rw-weight+}, we fix $\ell < n$ and write 
$q_{n}(x) = \sum_{y\in\Z^2} q_\ell(y) \, q_{n-\ell}(x-y)$ by Chapman-Kolmogorov.
We next decompose the sum in the two parts $\langle y,x\rangle > \frac{1}{2}|x|^2$
and $\langle y,x\rangle \le \frac{1}{2}|x|^2$: renaming $y$ as $x-y$ in the second 
part, we obtain
\begin{equation} \label{eq:usedeco}
	q_{n}(x) \le \sum_{y\in\Z^2: \, \langle y,x\rangle \ge \frac{1}{2}|x|^2} 
	\big\{ q_\ell(y) \, q_{n-\ell}(x-y) + q_{n-\ell}(y) \, q_\ell(x-y) \big\} \,.
\end{equation}
We can bound $q_k(x-y) \le \sup_{z\in\Z^2} q_k(z) \le \frac{\sfc}{k}$ by the local limit theorem
(any random walk satysfying Assumption~\ref{ass:rw}
is in $L^2$ with zero mean).
We next observe that $\langle y,x\rangle \ge \frac{1}{2}|x|^2$ implies $|x| \le 2 |y|$
by Cauchy-Schwarz, therefore the first bound in \eqref{eq:rw-weight+} yields
\begin{equation*}
	\forall x \in \Z^2: \qquad
	\rme^{t |x|} \, q_n(x)  \le \sfc \sum_{y\in\Z^2} 
	\rme^{2t |y|} \bigg\{ \frac{q_\ell(y)}{n-\ell} + \frac{q_{n-\ell}(y)}{\ell} \bigg\}
	\le \frac{2\sfc \, \rme^{8\sfc t^2 n}}{\min\{n-\ell,\ell\}} \,.
\end{equation*}
If we choose $\ell = \lfloor \frac{n}{2}\rfloor$, we obtain
the second bound in \eqref{eq:rw-weight+} renaming~$\sfc$.

It remains to prove the second bound in \eqref{eq:rw-weight}.
We first note that $q_n(x)^2/q_{2n}(0) \le \sfc \, q_n(x)$ 
for some $\sfc \in [1, \infty)$, because $q_n(x)^2 \le \|q_n\|_{\ell^\infty} \, q_n(x)$ and
$\|q_n\|_{\ell^\infty} \le \sfc \, q_{2n}(0)$ by the local limit theorem. 
Since $q_n(x) = q_n(-x)$, we get
\begin{equation*}
\begin{split}
	\sum_{x\in\Z^2} \rme^{t x^a} \, \frac{q_n(x)^2}{q_{2n}(0)} - 1
	&= \sum_{x\in\Z^2} \Big(\tfrac{\rme^{t x^a} + \rme^{-t x^a}}{2} - 1\Big)
	\, \frac{q_n(x)^2}{q_{2n}(0)} 
	\le \sfc \sum_{x\in\Z^2} \Big(\tfrac{\rme^{t x^a} + \rme^{-t x^a}}{2} - 1\Big) \, q_n(x) \\
	&\le \sfc \, \big( \rme^{\sfc \frac{t^2}{2} n}-1 \big)
	= \sfc \sum_{k=1}^\infty \tfrac{1}{k!} \, \big(\sfc \,\tfrac{t^2}{2}n\big)^k
	\le \sum_{k=1}^\infty \tfrac{1}{k!} \, \big(\sfc^2\,\tfrac{t^2}{2}n\big)^k
	= \rme^{\sfc^2 \frac{t^2}{2} n}-1 \,,
\end{split}
\end{equation*}
which proves the second bound in \eqref{eq:rw-weight}
if we rename $\sfc^2$ as $\sfc$.\qed

\smallskip
\subsection{Proof of Lemma~\ref{lem:weighted-averaged-rw}}

For any $y\in\Z^2$ and $p \in [1,\infty]$ we can write, recalling \eqref{eq:kernel1},
\begin{equation*}
	\frac{q_n^{f}(y)}{w_{t}(y)}
	= q_n^f(y) \, \rme^{t|y|} \le \sum_{z\in\Z^2} \rme^{t  |z|} \, |f(z)| \,
	\big\{ \rme^{t  |y-z|} \, q_n(y-z) \big\}
	\le \bigg\| \frac{f}{w_{t}} \bigg\|_{\ell^p} \,
	\bigg\| \frac{q_n}{w_{t}} \bigg\|_{\ell^q}  \,,
\end{equation*}
where $q \in [1,\infty]$ is such that $\frac{1}{p}+\frac{1}{q}=1$. Since
$\| \frac{q_n}{w_{t}} \|_{\ell^q}^q 
\le \| \frac{q_n}{w_{t}} \|_{\ell^\infty}^{q-1} \, \| \frac{q_n}{w_{t}} \|_{\ell^1}$,
it suffices to apply the bounds in \eqref{eq:rw-weight+}
to obtain the second bound in \eqref{eq:weighted-averaged-rw2}.

We next prove the first bound in \eqref{eq:weighted-averaged-rw2},
assuming $p \in [1,\infty)$: we have, by Hölder,
\begin{equation}\label{eq:qbound}
\begin{split}
	\bigg| \frac{q_n^{f}(x)}{w_{t}(x)} \bigg|^p \!
	& = 
	\bigg| \sum_{z\in\Z^2} \frac{f(z)}{w_t(z)} \, q_n(x-z) \, \frac{w_t(z)}{w_t(x)} \bigg|^p \\
	& \le \Bigg\{
	\sum_{z\in\Z^2} \frac{|f(z)|^p}{w_t(z)^p} \, q_n(x-z) \, \frac{w_t(z)}{w_t(x)} \Bigg\} \,
	\Bigg\{ \sum_{z\in\Z^2} q_n(x-z) \, \frac{w_t(z)}{w_t(x)}  \Bigg\}^{p-1} \\
	& \le \big( \sfc \, \rme^{2 \sfc t^2 n} \big)^{p-1} \, 
	\sum_{z\in\Z^2} \frac{|f(z)|^p}{w_t(z)^p} \, q_n(x-z) \, \frac{w_t(z)}{w_t(x)} \,,
\end{split}
\end{equation}
where the last inequality holds by the first bound in \eqref{eq:rw-weight+},
since $\frac{w_t(z)}{w_t(x)} \le \rme^{t|x-z|}$.
Summing over $x$ and applying again \eqref{eq:rw-weight+}, 
we obtain the first bound in \eqref{eq:weighted-averaged-rw2}.\qed

\subsection{Proof of Lemma~\ref{lem:max-averaged-rw}}\label{sec:max-averaged-rw}
Given a real function $G$, we set
$\{G > \lambda\} := \{y \colon G(y) > \lambda\}$
for $\lambda \in \R$,
and we denote by $|A|$ the cardinality of a set~$A$.
Let us define the constant
\begin{equation}\label{eq:sfC}
	\sfC := 200 \pi \, \sfc^2 \, \rme^{4\sfc \, t^2 L} \,.
\end{equation}
We are going to show that
\begin{align}
	\label{eq:toMa1}
	\Big\| \max_{1\le n \le L} 
	\, \big| q^{\otimes m, F}_n \, 
	w_t^{\otimes m} \big|\, \Big\|_{\ell^\infty} 
	&\le \sfC^m \, \big\| F \, w_t^{\otimes m} \big\|_{\ell^\infty} \,, \\
	 \forall \lambda > 0 : \qquad
	\Big| \Big\{  \max_{1\le n \le L} |q^{\otimes m, F}_n \, 
	w_t^{\otimes m} | > \lambda \Big\} \Big| 
	\label{eq:toMa2}
	& \le (25 \, \sfC)^m \, \frac{\|F \, w_t^{\otimes m}\|_{\ell^1}}{\lambda} \,.
\end{align}
Note that \eqref{eq:toMa1} implies our goal \eqref{eq:max-averaged-rw} for $p=\infty$,
while \eqref{eq:toMa2} means that the sub-linear operator
$F \mapsto \max_{1\le n \le L} \big| q^{\otimes m, F}_n \, w_t^{\otimes m} \big|$
is of \emph{weak type $(1,1)$}, see \cite{cf:Gra}. Then, for every $1 < p < \infty$,
our goal \eqref{eq:max-averaged-rw} where  $\overline{\scrC} = 25 \,\sfC$ follows
by \emph{Marcinkiewicz's Interpolation Theorem}, see
\cite[Theorem~1.3.2 and Exercise~1.3.3(a)]{cf:Gra}.

We now prove \eqref{eq:toMa1} and \eqref{eq:toMa2}.
For any dimension $d\in\N$,
we denote  by $\cB^{d}(\bx,r)$ the set of integer points in the Euclidean ball in $\R^{d}$
with center $\bx \in \Z^d$ and radius $r > 0$:
\begin{equation} \label{eq:BcB}
	\cB^{d}(\bx,r) := \big\{\by\in \Z^d: \, |\by - \bx| = \sqrt{(y_1 - x_1)^2 +
	\ldots + (y_d - x_d)^2} \le r \big\} \,.
\end{equation}
We focus on the case $d = 2m$ and we write $\bx = (x_1, \ldots, x_m)$
with $x_i \in \Z^2$.
Given a function $F: (\Z^2)^m \to \R$, we define the
\emph{maximal function} $\cM^F: (\Z^2)^m \to [0,\infty]$ by
\begin{equation}\label{eq:Mg}
	\cM^F(\bx) := \sup_{0 < r < \infty}
	\Bigg\{ \frac{1}{|\cB^{2 m}(\bx,r)|}
	\sum_{\by \in \cB^{2m}(\bx,r)} |F(\by)| \Bigg\} \,.
\end{equation}
We are going to prove the following discrete
version of \emph{Hardy-Littlewood maximal inequality}:
\begin{equation} \label{eq:HL}
	\forall \lambda > 0: \qquad
	| \{ \cM^{F} > \lambda \} | \le 25^m \, \frac{\|F\|_{\ell^1}}{\lambda} \,.
\end{equation}
We are also going to prove the following upper bound:
for any $m\in\N$, $L\in\N$, $x \in \Z^2$,
\begin{equation}\label{eq:ub-maximal}
	\max_{1\le n \le L} |q^{\otimes m, F}_n(\bx) \, 
	w_t^{\otimes m}(\bx)|
	\le \sfC^m \,  \cM^{F w_t^{\otimes m}} (\bx) 
\end{equation}
Since clearly $\|\cM^G\|_{\ell^\infty} \le \|G\|_{\ell^\infty}$,
this directly implies \eqref{eq:toMa1} and, coupled to \eqref{eq:HL},
also \eqref{eq:toMa2}.
To complete the proof, it only remains to prove \eqref{eq:HL} and \eqref{eq:ub-maximal}.

\smallskip
\paragraph{\it Proof of \eqref{eq:HL}}
We follow closely the classical proof of the Hardy-Littlewood
maximal inequality, see \cite[Theorem~2.1.6]{cf:Gra}, which is stated on $\R^d$ instead of $\Z^d$. 
By definition of $\cM^F$,
see \eqref{eq:Mg}, for every point $\bx \in \{\cM^F > \lambda\}$ there
is $r_\bx > 0$ such that
\begin{equation}\label{eq:tobo}
	\sum_{\by\in \cB^{2m}(\bx,r_\bx)} |F(\by)| > \lambda \, 
	|\cB^{2m}(\bx,r_\bx)| \,.
\end{equation}
It suffices to fix any \emph{finite} set $K \subseteq \{\cM^F > \lambda\}$ and prove
that \eqref{eq:HL} holds with the LHS replaced by $|K|$. From the family of
balls $\cF := \{ \cB^{2m}(\bx,r_\bx): \, \bx \in K\}$ we extract a \emph{disjoint sub-family}
$\cF' := \{ \cB^{2m}(\bz,r_\bz): 
\, \bz \in K'\}$ with $K' \subseteq K$ by the greedy algorithm,
see \cite[Lemma~2.1.5]{cf:Gra}:
we first pick the ball of largest radius,
then we select the ball of largest radius among the remaining ones \emph{which do not intersect 
the balls that have already been picked}, and so on. By construction, if
a ball $\cB^{2m}(\bx,r_\bx)$ 
is \emph{not} included in $\cF'$, then it must overlap with some ball $\cB^{2m}(\bz,r_\bz)$ 
of larger radius $r_\bz \ge r_\bx$, therefore $\cB^{2m}(\bx,r_\bx) 
\subseteq \cB^{2m}(\bz, 3r_\bz)$.
In other terms, \emph{tripling the radii of the balls in $\cF'$ we cover all the balls
in $\cF$}, hence
\begin{equation*}
	|K| \le \sum_{\bz \in K'} |\cB^{2m}(\bz, 3r_\bz)| \,.
\end{equation*}
We prove below that, for any dimension $d\in\N$, $\bz \in \Z^{d}$ and $r > 0$,
\begin{equation}\label{eq:provebelow}
	|\cB^{d}(\bz,3r)| \le 5^d \, |\cB^{d}(\bz,r)| \,.
\end{equation}
Setting $d=2m$ and applying \eqref{eq:tobo}, we then obtain \eqref{eq:HL}:
\begin{equation*}
	|K| \le 25^m \sum_{z \in K'} |\cB^{2m}(\bz, r_\bz)|
	\le \frac{25^m}{\lambda} \sum_{z \in K'} \sum_{\by\in \cB^{2m}(\bz, r_\bz)} |F(\by)| 
	\le \frac{25^m}{\lambda} \, \|F\|_{\ell^1} \,,
\end{equation*}
where the last inequality holds because the balls
$\cB^{2m}(\bz,r_\bz)$ for $\bz \in K'$ are disjoint.

It remains to prove \eqref{eq:provebelow}. We fix $\bz = 0$
and we proceed by induction on~$d\in\N$.
\begin{itemize}
\item The case $d=1$ is proved by direct computation.
Note that $\cB^1(0,r) = \{-\lfloor r \rfloor,  \ldots, \lfloor r \rfloor\}$, hence
$|\cB^1(0,r)| = 2\lfloor r \rfloor+1$. For $0 \le r < 1$
we have $|\cB^1(0,r)| = 1$ while $|\cB^1(0,3r)| \le 5$, 
therefore \eqref{eq:provebelow} holds (as an \emph{equality} for $\frac{2}{3} \le r < 1$). 
More generally, given $k \in \N_0$,
for $k \le r < k+1$ we have $\lfloor r \rfloor = k$
and $\lfloor 3r \rfloor \le 3k+2$, therefore
$|\cB^1(0,r)| = 2k+1$
while $|\cB^1(0,3r)|  \le 2 \, (3k+2) + 1 = 6k+5$, which yields
\begin{equation*}
	\frac{|\cB^1(0,3r)|}{|\cB^1(0,r)|} \le \frac{6k+5}{2k+1}
	= 3 + \frac{2}{2k+1} \le 3 + 2 = 5 \,.
\end{equation*}
\item We next assume that \eqref{eq:provebelow} is proved
for some $d\in\N$ and we prove it for $d+1$. Recalling \eqref{eq:BcB}
and writing $\by = (y_1, \ldots, y_d)$, we sum over the possible values
of $y := y_1$ to write
\begin{equation} \label{eq:Bd}
	|\cB^{d+1}(0,r)| = \sum_{y \in \{-\lfloor r \rfloor, \ldots, \lfloor r \rfloor\}}
	|\cB^{d}(0,\sqrt{r^2 - y^2})| \,.
\end{equation}
In particular, replacing $r$ by $3r$ and applying the induction assumption \eqref{eq:provebelow}, we get
\begin{equation*}
\begin{split}
	|\cB^{d+1}(0,3r)| 
	& \le 5^d \sum_{y \in \{-\lfloor 3r \rfloor, \ldots, \lfloor 3r \rfloor\}}
	|\cB^{d}(0,\sqrt{r^2 - (y/3)^2})| \\
	& \le 5^d \sum_{y \in \{-\lfloor 3r \rfloor, \ldots, \lfloor 3r \rfloor\}}
	|\cB^{d}(0,\sqrt{r^2 - [ y/3]^2})| \,,
\end{split}
\end{equation*}
where in the last inequality we increased
the radius $\sqrt{r^2 - (y/3)^2}$ replacing $y/3$ by $[ y/3 ]$
\emph{defined as $\lfloor y/3 \rfloor$ for $y \ge 0$ and 
as $\lceil y/3 \rceil$ for $y < 0$}, so that $|[y/3]| \le |y/3|$.
We finally note that, as $y$ ranges in $\{-\lfloor 3r \rfloor, \ldots, \lfloor 3r \rfloor\}$,
the variable $\tilde y := [y/3]$ ranges in $\{-\lfloor r \rfloor, \ldots, \lfloor r \rfloor\}$,
and each value of $\tilde y$ comes either~$3$ or~$5$ values of $y$.\footnote{Indeed, 
$\tilde y = 0$ comes from $y \in \{-2,-1,0,1,2\}$, while
$\tilde y = \ell > 0$ comes from $y \in \{3\ell, 3\ell+1, 3\ell+2\}$,
and similarly for $\tilde y = \ell < 0$.}
We thus obtain, recalling \eqref{eq:Bd},
\begin{equation*}
	|\cB^{d+1}(0,3r)| 
	\le 5^d \cdot 5 \sum_{\tilde y \in \{-\lfloor r \rfloor, \ldots, \lfloor r \rfloor\}}
	|\cB^{d}(0,\sqrt{r^2 - \tilde y^2})|
	= 5^{d+1} \, |\cB^{d+1}(0,r)|  \,,
\end{equation*}
which completes the proof of \eqref{eq:provebelow}.
\end{itemize}

\smallskip

\paragraph{\it Proof of \eqref{eq:ub-maximal}}
We claim that for all $1 \le n \le L$ and $x\in\Z^2$
\begin{equation} \label{eq:densitybound}
	q_n(x) \, \rme^{t |x|} \le
	\frac{\sfC'}{n} \, \rme^{-\frac{|x|^2}{16 \sfc \, n}} \qquad
	\text{where} \quad \sfC' := 6\sfc \, \rme^{4\sfc \, t^2 L} \,.
\end{equation}
Indeed, we prove in Lemma~\ref{lem:heat-kernel-bound} below that
$q_n(x) \le \frac{6\sfc}{n} \, \rme^{-\frac{|x|^2}{8\sfc\, n}}$, see \eqref{eq:heat-kernel-bound}, therefore
\begin{equation*}
	q_{n}(x) \, \rme^{t|x|} \le \frac{6\sfc}{n} \, \rme^{t|x| -\frac{|x|^2}{8\sfc\, n}}
	\le \frac{6\sfc}{n} \, \rme^{-\frac{|x|^2}{16\sfc\, n}}
	\cdot \Big( \sup_{\gamma \ge 0} \rme^{t\gamma -\frac{\gamma^2}{16\sfc\, n}} \Big)
	= \frac{6\sfc}{n} \, \rme^{-\frac{|x|^2}{16\sfc\, n}} \, \rme^{4\sfc\, t^2 n} \,,
\end{equation*}
which shows that \eqref{eq:densitybound} holds for $n \le L$.

Let us now deduce \eqref{eq:ub-maximal} from \eqref{eq:densitybound}.
Since $\frac{w_t(x)}{w_t(z)} \le \rme^{t|x-z|}$,
by \eqref{eq:kernel1m} and \eqref{eq:densitybound} we get
\begin{equation*}
\begin{split}
	|q^{\otimes m, F}_n(\bx) \, w_t^{\otimes m}(\bx)|
	& \le \sum_{\bz \in (\Z^2)^m} |F(\bz)|
	\, w_t^{\otimes m}(\bz) \, \prod_{i=1}^m 
	q_n(x_i-z_i) \, \rme^{t|x_i-z_i|} \\
	& \le \bigg(\frac{\sfC'}{n}\bigg)^m \, \sum_{\bz \in (\Z^2)^m} |F(\bz)|
	\, w_t^{\otimes m}(\bz) \, 
	\rme^{-\frac{|\bx - \bz|^2}{16 \sfc \, n}} \,,
\end{split}
\end{equation*}
where $|\bx - \bz|^2 = \sum_{i=1}^m |x_i - z_i|^2$ is the Euclidean norm
on $(\R^2)^m$. Recalling \eqref{eq:BcB}, we write
\begin{equation*}
	\rme^{-\frac{|\bx - \bz|^2}{16 \sfc \, n}}
	= \int_0^1 \dd s \, \ind_{\{s \le \rme^{-\frac{|\bx - \bz|^2}{16 \sfc \, n}}\}} 
	= \int_0^1 \dd s \, \ind_{\{\bz \in \cB^{\otimes m}(\bx, r_{n,s})\}} \qquad
	\text{with} \quad r_{n,s} := 16\sfc \, n \, \log \tfrac{1}{s} \,,
\end{equation*}
therefore, recalling \eqref{eq:Mg}, we get
\begin{equation*}
\begin{split}
	|q^{\otimes m, F}_n(\bx) \, w_t^{\otimes m}(\bx)|
	&\le \bigg(\frac{\sfC'}{n}\bigg)^m
	\, \cM^{F w_t^{\otimes m}}(\bx) \cdot
	\int_0^1 \dd s \, \big| \cB^{\otimes m}(\bx, r_{n,s}) \big| \,.
\end{split}
\end{equation*}
Since
$\big| \cB^{\otimes m}(\bx, r_{n,s}) \big| = \sum_{\bz \in (\Z^2)^m}
\ind_{\{s \le \rme^{-\frac{|\bx - \bz|^2}{16 \sfc \, n}}\}} = \sum_{\by \in (\Z^2)^m}
\ind_{\{s \le \rme^{-\frac{|\by|^2}{16 \sfc \, n}}\}} $, we finally obtain
\begin{equation*}
	|q^{\otimes m, F}_n(\bx) \, w_t^{\otimes m}(\bx)|
	\le \bigg(\frac{\sfC'}{n} \, \sum_{y \in \Z^2} \rme^{-\frac{|y|^2}{16 \sfc \, n}} \bigg)^m
	\, \cM^{F w_t^{\otimes m}}(\bx) \,,
\end{equation*}
and it remains to show that the term in parenthesis is
at most $\sfC$, see \eqref{eq:ub-maximal}.
By monotonicity
\begin{equation*}
	\sum_{a\in\Z} \rme^{-\frac{a^2}{16\sfc \, n}} \le 1 +
	\int_\R \rme^{-\frac{x^2}{16\sfc \, n}} \, \dd x = 1 + \sqrt{16 \pi \, \sfc \, n} \,,
\end{equation*}
hence writing $y = (a,b)$, so that $|y|^2 = a^2 + b^2$, we obtain
\begin{equation*}
	\frac{\sfC'}{n} \, \sum_{y \in \Z^2} \rme^{-\frac{|y|^2}{16 \sfc \, n}}
	= \frac{\sfC'}{n} \bigg( \sum_{a\in\Z} \rme^{-\frac{a^2}{16\sfc\, n}} \bigg)^2
	\le \sfC' \, \frac{2(1 + 16 \pi \, \sfc \, n)}{n}
	\le (2 + 32\pi) \,  \sfc \,  \sfC' \le 33\pi \,  \sfc \,  \sfC' \,,
\end{equation*}
where the second last inequality holds by $n\ge 1 $ and $\sfc \ge 1$.
Since $33\pi \, \sfc \, \sfC' \le \sfC$, see \eqref{eq:sfC} and \eqref{eq:densitybound},
the proof is completed.\qed

\begin{lemma}[Heat kernel bound]\label{lem:heat-kernel-bound}
Let Assumption~\ref{ass:rw} hold and let $\sfc$ be the constant from Lemma~\ref{lem:weighted-rw}.
Then for every $n\in\N$ and $x \in \Z^2$ we have
\begin{equation}\label{eq:heat-kernel-bound}
	q_n(x) \le \frac{6\sfc}{n} \, \rme^{-\frac{|x|^2}{8\sfc n}} \,.
\end{equation}
\end{lemma}

\begin{proof}
We assume that $n \ge 2$, since the case $n=1$ is easier.
Let us apply the formula \eqref{eq:usedeco}
with $\ell = \lfloor \frac{n}{2}\rfloor$, so that
$\frac{n}{3} \le \ell \le \frac{n}{2}$:
by \eqref{eq:rw-weight+} (with $t=0$) we have
$q_k(x-y) \le \frac{\sfc}{k} \le \frac{3\sfc}{n}$
for both $k=\ell$ and $k = n-\ell$, therefore for any $\rho \ge 0$
\begin{equation} \label{eq:usedeco2}
	q_{n}(x) \le \frac{3\sfc}{n} \, \rme^{-\rho|x|}
	\sum_{y\in\Z^2: \,  \langle y,x\rangle \ge \frac{1}{2} |x|^2} 
	\rme^{2 \rho \,\langle y,\frac{x}{|x|}\rangle} \big\{ q_\ell(y) + q_{n-\ell}(y) \big\} \,,
\end{equation}
where we bounded $1 \le
\rme^{-\rho|x|} \rme^{2\rho \, \langle y,\frac{x}{|x|}\rangle}$
because $\langle y,x\rangle \ge \frac{1}{2}|x|^2$
(with $\frac{x}{|x|} := 0$ for $x=0$). 
For any $w = (w^1, w^2) \in \R^2$,
by \eqref{eq:rw-weight} and Cauchy-Schwarz we can bound
\begin{equation*}
	\sum_{y\in\Z^2} \rme^{\langle y, w\rangle} \, q_\ell(y)
	\le \sqrt{\sum_{y\in\Z^2} \rme^{2 y^1 w^1} \, q_\ell(y)
	\cdot \sum_{y\in\Z^2} \rme^{2 y^2 w^2} \, q_\ell(y)}
	\le \rme^{\sfc \, |w|^2 \ell} \,,
\end{equation*}
and similarly for $q_{n-\ell}(\cdot)$, therefore for $\max\{\ell,n-\ell\} \le \frac{n}{2}$
we obtain by \eqref{eq:usedeco2}
\begin{equation*}
	q_{n}(x) \le \frac{6\sfc}{n} \, \rme^{-\rho|x| + 2\sfc \, \rho^2 \, n} \,.
\end{equation*}
Optimising over $\rho$ leads us to choose $\rho = \frac{|x|}{4\sfc n}$, 
which yields \eqref{eq:heat-kernel-bound}.
\end{proof}

\section{Estimates on boundary and bulk terms}
\label{app:boundary}

In this section we prove the estimates on the boundary terms
(Propositions~\ref{prop:left} and~\ref{prop:left2} for the left boundary,
Proposition~\ref{prop:right} for the right boundary)
and on the bulk terms
(Proposition~\ref{prop:bulk-rw} and Proposition~\ref{prop:bulk-interacting}).

\subsection{Proof of Propositions~\ref{prop:left}}

By the triangle inequality we can bound
\begin{equation} \label{eq:tsw}
	\bigg\|  \frac{\widehat \sfq_{L}^{|f|,I}}{\cW_t} \bigg\|_{\ell^p}
	\le \sum_{n=1}^L  \bigg\| \frac{\sfq_{n}^{|f|,I}}{\cW_t} \bigg\|_{\ell^p} \,.
\end{equation}
Writing $I = \{I^1, \ldots, I^m\}$ we can write,
recalling \eqref{eq:Qdef2}, \eqref{eq:QLap} and \eqref{eq:weights},
\begin{equation} \label{eq:ina00}
	\bigg\| \frac{\sfq_{n}^{|f|,I}}{\cW_t} \bigg\|_{\ell^p}^p = 
	 \sum_{\bx \in (\Z^2)^h} \frac{\sfq_{n}^{|f|,I}(\bx)^p}{\cW_t(\bx)^p}
	 \le \prod_{j=1}^m
	\Bigg\{ \sum_{y\in\Z^2} \frac{q_n^{|f|}(y)^{p|I^j|}}{w_t(y)^{p|I^j|}} \, 
	\Bigg\}
	 = \prod_{j=1}^m \bigg\| \frac{q_n^{|f|}}{w_t} \bigg\|_{\ell^{p|I^j|}}^{p|I^j|} \,.
\end{equation}
Since $\big\| \cdot \big\|_{\ell^{pk}}^{pk}
\le \big\| \cdot \big\|_{\ell^\infty}^{p(k-1)} \, \big\| \cdot \big\|_{\ell^{p}}^{p}$,
from $\sum_{j=1}^m |I^j| = h$ we get
(raising to $1/p$)
\begin{equation} \label{eq:yb}
	\bigg\| \frac{\sfq_{n}^{|f|,I}}{\cW_t} \bigg\|_{\ell^p}
	\le \bigg\| \frac{q_n^{|f|}}{w_t} \bigg\|_{\ell^\infty}^{h-m} \,
	\bigg\| \frac{q_n^{|f|}}{w_t} \bigg\|_{\ell^{p}}^{m} 
	\le \bigg\| \frac{q_n^{|f|}}{w_t} \bigg\|_{\ell^\infty} \,
	\bigg\| \frac{q_n^{|f|}}{w_t} \bigg\|_{\ell^{p}}^{h-1} \,,
\end{equation}
where the last inequality holds since $m \le h-1$ for $I \ne *$
(note that $\|\cdot\|_{\ell^\infty} \le \|\cdot\|_{\ell^p}$).
By \eqref{eq:weighted-averaged-rw2},
for any $r \in [1,\infty]$,
\begin{equation} \label{eq:reco-ave}
	\bigg\| \frac{q_n^{|f|}}{w_t} \bigg\|_{\ell^\infty} \le
	\frac{\sfc \, \rme^{2\sfc \, t^2 n}}{n^{\frac{1}{r}}} \,
	\bigg\| \frac{f}{w_t} \bigg\|_{\ell^{r}} \,, \qquad
	\bigg\| \frac{q_n^{|f|}}{w_t} \bigg\|_{\ell^{p}} \le
	\sfc \, \rme^{2\sfc \, t^2 n}\, \bigg\| \frac{f}{w_t} \bigg\|_{\ell^{p}} \,,
\end{equation}
hence we obtain for $n \le L$, recalling the definition of $\scrC$ in \eqref{eq:CbarC},
\begin{equation} \label{eq:oh}
	\bigg\| \frac{\sfq_{n}^{|f|,I}}{\cW_t} \bigg\|_{\ell^p}
	\le \frac{\scrC^h}{n^{\frac{1}{r}}} \,
	\bigg\| \frac{f}{w_t} \bigg\|_{\ell^{r}}
	\, \bigg\| \frac{f}{w_t} \bigg\|_{\ell^{p}}^{h-1} \,.
\end{equation}
Plugging this into \eqref{eq:tsw}, since
$\sum_{n=1}^L \frac{1}{n^a} \le \int_0^L \frac{1}{x^a} \, \dd x = \frac{L^{1-a}}{1-a}$,
we obtain 
\begin{equation}\label{eq:est1-prel}
	\max_{I \ne *} \bigg\Vert \frac{\widehat \sfq_{L}^{|f|,I}}{\cW_t} \bigg\Vert_{\ell^p}
	\le 
	\tfrac{r}{r-1} \, \scrC^h \, L^{1-\frac{1}{r}}
	\, \bigg\|\frac{f}{w_{t}}\bigg\|_{\ell^r}
	\, \bigg\|\frac{f}{w_{t}}\bigg\|_{\ell^p}^{h-1} \, ,
\end{equation}
which proves \eqref{eq:est1} for $r \ge p$ (so that $\min\{\frac{r}{r-1}, \frac{p}{p-1}\} = \frac{r}{r-1}$).
More generally, if $r \ge \frac{3p}{1+2p}$, then
$\tfrac{r}{r-1} \le 3\frac{p}{p-1}$ hence \eqref{eq:est1-prel}
still proves \eqref{eq:est1}.

It remains to prove \eqref{eq:est1}
for $r \in [1,\frac{3p}{1+2p}] \subseteq [1,p)$. Let us obtain
an estimate alternative to \eqref{eq:oh}.
Since $\|\cdot\|_{\ell^p}^p \le \|\cdot\|_{\ell^\infty}^{p-r} \, \|\cdot\|_{\ell^r}^r$
for $r < p$, by \eqref{eq:weighted-averaged-rw2} we obtain
\begin{equation} \label{eq:l2lp}
	\bigg\| \frac{q_n^{|f|}}{w_t} \bigg\|_{\ell^{p}} \le
	\bigg\| \frac{q_n^{|f|}}{w_t} \bigg\|_{\ell^\infty}^{1-\frac{r}{p}} \,
	\bigg\| \frac{q_n^{|f|}}{w_t} \bigg\|_{\ell^r}^{\frac{r}{p}} \le
	\frac{\sfc \, \rme^{2\sfc \, t^2 n}}{n^{\frac{1}{r}-\frac{1}{p}}}
	\, \bigg\| \frac{f}{w_t} \bigg\|_{\ell^{r}} \,,
\end{equation}
which we can use to estimate one factor of $\| \frac{q_n^{|f|}}{w_t} \|_{\ell^{p}}$
appearing in \eqref{eq:yb} (recall that $h \ge 2$):
applying again the first bound in \eqref{eq:reco-ave}, for $n \le L$ we obtain from \eqref{eq:yb}
\begin{equation} \label{eq:alpha}
	\bigg\| \frac{\sfq_{n}^{|f|,I}}{\cW_t} \bigg\|_{\ell^p}
	\le \frac{\scrC^h}{n^{\gamma}} \,
	\bigg\| \frac{f}{w_t} \bigg\|_{\ell^{r}}^{2} \,
	\bigg\| \frac{f}{w_t} \bigg\|_{\ell^{p}}^{h-2} 
	\qquad
	\text{with} \quad \gamma := \tfrac{2}{r} -\tfrac{1}{p} = \tfrac{1}{r}+ \tfrac{p-r}{pr} \,.
\end{equation}
The RHS of \eqref{eq:alpha} is smaller than the RHS of \eqref{eq:oh} if and only if
\begin{equation} \label{eq:tilden}
	\frac{1}{n^{\gamma}} \, \bigg\| \frac{f}{w_t} \bigg\|_{\ell^{r}}
	< \frac{1}{n^{\frac{1}{r}}} \, \bigg\| \frac{f}{w_t} \bigg\|_{\ell^{p}}
	\qquad \iff \qquad
	n > \tilde{n} := 
	\Bigg(\frac{\|\frac{f}{w_{t}}\|_{\ell^r}}{\|\frac{f}{w_{t}}\|_{\ell^p}}\Bigg)^{\frac{pr}{p-r}} \,.
\end{equation}
Note that for $r \in [1, \frac{3p}{1+2p}]$
we have $\gamma-1 \ge \tfrac{2(1+2p)}{3p} - \tfrac{1}{p} - 1 = \tfrac{p-1}{3p} > 0$,
hence  $\gamma > 1$.
Then $\sum_{n > \tilde{n}}^\infty \frac{1}{n^{\gamma}} 
\le \int_{\tilde{n}}^\infty \frac{1}{x^{\gamma}} \, \dd x
= \frac{1}{\gamma-1}\, \tilde{n}^{1-\gamma}
\le \frac{3p}{p-1} \, \tilde{n}^{1-\gamma}$,
hence by \eqref{eq:alpha} we can bound
\begin{equation*}
	\sum_{n > \tilde{n}}  \bigg\| \frac{\sfq_{n}^{|f|,I}}{\cW_t} \bigg\|_{\ell^p}
	\le \tfrac{3p}{p-1} \, \scrC^h \, \tilde{n}^{1-\gamma} \, 
	\bigg\| \frac{f}{w_t} \bigg\|_{\ell^{r}}^{2} \,
	\bigg\| \frac{f}{w_t} \bigg\|_{\ell^{p}}^{h-2}
	= \tfrac{3p}{p-1} \, \scrC^h \, \bigg\|\frac{f}{w_{t}}\bigg\|_{\ell^r}^{\frac{r(p-1)}{p-r}}
	\, \bigg\|\frac{f}{w_{t}}\bigg\|_{\ell^p}^{h-\frac{r(p-1)}{p-r}} \,,
\end{equation*}
where the equality follows by the definitions
of $\tilde{n}$ in \eqref{eq:tilden} and $\gamma$ in \eqref{eq:alpha}.
For the contribution of $n \le \tilde{n}$, the previous bound
\eqref{eq:oh} with $r=p$ yields, as in \eqref{eq:est1-prel},
\begin{equation*}
	\sum_{n=1}^{\tilde{n}}  \bigg\| \frac{\sfq_{n}^{|f|,I}}{\cW_t} \bigg\|_{\ell^2}
	\le \tfrac{p}{p-1}
	\, \scrC^h \,\tilde{n}^{1-\frac{1}{p}} \, \bigg\|\frac{f}{w_{t}}\bigg\|_{\ell^p}^{h}  \, 
	= \tfrac{p}{p-1} \,\scrC^h \, \bigg\|\frac{f}{w_{t}}\bigg\|_{\ell^r}^{\frac{r(p-1)}{p-r}}
	\, \bigg\|\frac{f}{w_{t}}\bigg\|_{\ell^p}^{h-\frac{r(p-1)}{p-r}} \,,
\end{equation*}
having used the definition
of $\tilde{n}$ in \eqref{eq:tilden}. Overall, see  \eqref{eq:tsw}, for $r \in [1,\frac{3p}{1+2p}]$
we have
\begin{equation}\label{eq:est12bis}
	\max_{I \ne *} \bigg\Vert \frac{\widehat \sfq_{L}^{|f|,I}}{\cW_t} \bigg\Vert_{\ell^p}
	\le \underbrace{\tfrac{4p}{p-1} \,\scrC^h \, \bigg\|\frac{f}{w_{t}}\bigg\|_{\ell^r}^{\frac{1}{\alpha}}
	\, \bigg\|\frac{f}{w_{t}}\bigg\|_{\ell^p}^{h-\frac{1}{\alpha}}}_{A}
	\qquad \text{with } \alpha:= \tfrac{p-r}{r(p-1)}  \in (0,1] \,.
\end{equation}
At the same time, we can apply again the previous bound
\eqref{eq:est1-prel} with $r=p$ to estimate
\begin{equation}\label{eq:est12}
	\max_{I \ne *} \bigg\Vert \frac{\widehat \sfq_{L}^{|f|,I}}{\cW_t} \bigg\Vert_{\ell^p}
	\le \underbrace{ \tfrac{p}{p-1} \,\scrC^h \, L^{1-\frac{1}{p}}
	\, \bigg\|\frac{f}{w_{t}}\bigg\|_{\ell^p}^{h} }_{B} \,.
\end{equation}
Combining these bounds we get
$\max_{I \ne *} \big\Vert \frac{\widehat \sfq_{L}^{|f|,I}}{\cW_t} \big\Vert_{\ell^2} \le A^{\alpha}
B^{1-\alpha}$, hence 
\begin{equation*}
	\forall r \in [1,\tfrac{3p}{1+2p}]: \qquad
	\max_{I \ne *} \bigg\Vert \frac{\widehat \sfq_{L}^{|f|,I}}{\cW_t} \bigg\Vert_{\ell^p}
	\le \tfrac{4p}{p-1} \,\scrC^h \, L^{1-\frac{1}{r}}
	\, \bigg\|\frac{f}{w_{t}}\bigg\|_{\ell^r}
	\, \bigg\|\frac{f}{w_{t}}\bigg\|_{\ell^p}^{h-1}  \,,
\end{equation*}
which coincides with our goal \eqref{eq:est1}, since
$\min\{\frac{r}{r-1}, \frac{p}{p-1}\} = \frac{p}{p-1}$ for $r < p$.\qed

\smallskip
\subsection{Proof of Proposition~\ref{prop:left2}}

We follow the proof of Proposition~\ref{prop:left}.
By the triangle inequality, as in \eqref{eq:tsw},
it is enough to show that
\begin{equation}\label{eq:oh2}
	\bigg\| \frac{\sfq_{n}^{|f|,I}}{\cW_t} \cV_s^J \bigg\|_{\ell^p} 
	\le \frac{36^{\frac{1}{p}} \, \scrC^h}{s^{2/p}} \,
	\bigg\| \frac{f}{w_t} \bigg\|_{\ell^{\infty}}^2
	\, \bigg\| \frac{f}{w_t} \bigg\|_{\ell^{p}}^{h-2} \,.
\end{equation}

We assume for ease of notation that 
$J=\{\{1,2\}, \{3\},\ldots,\{h\}\}$.
Let us fix a partition $I = \{I^1, \ldots, I^{m}\}$ such that $I \not\supseteq J$, 
say $1 \in I^1$ and $2 \in I^2$.
In analogy with \eqref{eq:ina00}, we have
\begin{equation} \label{eq:overa}
	\bigg\| \frac{\sfq_{n}^{|f|,I}}{\cW_t} \cV_s^J \bigg\|_{\ell^p}^p \le
	\widehat\Sigma^{(1,2)}_n \cdot
	\prod_{j=3}^m \bigg\| \frac{q_n^{|f|}}{w_t} \bigg\|_{\ell^{p|I^j|}}^{p|I^j|} \,.
\end{equation}
where 
\begin{equation} \label{eq:hatSigma}
	\widehat\Sigma^{(1,2)}_n 
	:=  \sum_{y^1, y^2\in\Z^2} \big( q_n^{|f|}(y^1) \, \rme^{t|y^1|} \big)^{p|I^1|}
	\, \big( q_n^{|f|}(y^2) \, \rme^{t|y^2|} \big)^{p|I^2|}
	\, \rme^{-ps|y^1-y^2|} \,.
\end{equation}
By a uniform bound, 
we can estimate
\begin{equation} \label{eq:S12}
\begin{split}
	\widehat\Sigma^{(1,2)}_n 
	& \le \, \bigg\| \frac{q_n^{|f|}}{w_t} \bigg\|_{\ell^\infty}^{p|I^1|} 
	\sum_{y^1, y^2\in\Z^2}  
	\bigg( \frac{q_n^{|f|}(y^2)}{w_{t}(y^2)} \bigg)^{p|I^2|} \, 
	\rme^{-ps|y^1-y^2|} \\
	& = \, \bigg\| \frac{q_n^{|f|}}{w_t} \bigg\|_{\ell^\infty}^{p|I^1|} 
	\, \bigg\| \frac{q_n^{|f|}}{w_t} \bigg\|_{\ell^p}^{p|I^2|}
	\,\bigg( \sum_{y\in\Z^2} \rme^{-ps|y|} \bigg) \,.
\end{split}
\end{equation}
Since $2|z| \ge |z^1| + |z^2|$ for $z = (z^1, z^2) \in \Z^2$
and $1-\rme^{-x} \ge \frac{2}{3}x$ for $0\le x\le \frac{1}{2}$, we can bound
\begin{equation}\label{eq:easybound}
	\sum_{z\in\Z^2} \rme^{-ps|z|}
	\le \sum_{z\in\Z^2} \rme^{-s|z|}
	\le \bigg( \sum_{x\in\Z} \rme^{-s\frac{|x|}{2}} \bigg)^2
	\le \bigg( \frac{2}{1-\rme^{-\frac{s}{2}}} \bigg)^2 \le \frac{36}{s^2} \,.
\end{equation}
Plugging these estimates into \eqref{eq:overa} and bounding
$\big\| \cdot \big\|_{\ell^{pk}}^{pk}
\le \big\| \cdot \big\|_{\ell^\infty}^{p(k - 1)} \, \big\| \cdot \big\|_{\ell^{p}}^{p}$,
since $\sum_{j=1}^m |I^j| = h$
and $m \le h-1$, we obtain (raising to $1/p$)
\begin{equation*}
	\bigg\| \frac{\sfq_{n}^{|f|,I}}{\cW_t} \cV_s^J \bigg\|_{\ell^p} \le
	\frac{36^{\frac{1}{p}}}{s^{2/p}} \, \bigg\| \frac{q_n^{|f|}}{w_t} \bigg\|_{\ell^\infty}^{h-m+1} 
	\, \bigg\| \frac{q_n^{|f|}}{w_t} \bigg\|_{\ell^{p}}^{m - 1} 
	\le \frac{36^{\frac{1}{p}}}{s^{2/p}} \, \bigg\| \frac{q_n^{|f|}}{w_t} \bigg\|_{\ell^\infty}^{2} 
	\, \bigg\| \frac{q_n^{|f|}}{w_t} \bigg\|_{\ell^{p}}^{h - 2} \,.
\end{equation*}
Applying the estimates in \eqref{eq:reco-ave}, we obtain \eqref{eq:oh2}.
\qed

\smallskip

\subsection{Proof of Proposition~\ref{prop:right}}
The second line of \eqref{eq:est4-new} follows by the first line because
$\|\cdot\|_{\ell^{2q}}^2 \le \|\cdot\|_{\ell^\infty} \, \|\cdot\|_{\ell^q}$.
Let us prove the first line of \eqref{eq:est4-new}.
Writing $J = \{J^1,\ldots, J^m\}$ 
and recalling \eqref{eq:Qright} and \eqref{eq:Qdef2} we can write, as in \eqref{eq:ina00},
\begin{equation*}
	\big\Vert \overline{\sfq}_{L}^{|g|,J} \, \cW_t \big\Vert_{\ell^q}^q
	= \sum_{\bx \in (\Z^2)^h} 
	\overline{\sfq}_{L}^{|g|,J}(\bx)^q \, \cW_t(\bx)^q
	\le 
	\sum_{\by \in (\Z^2)^m}
	\max_{1 \le n \le L} \,
	\prod_{j=1}^m \big( q_n^{|g|}(y_j) \, w_t(y_j)\big)^{ q |J^j|}  \,.
\end{equation*}
We next observe that for $k = |J^j| \ge 1$, arguing as in \eqref{eq:qbound}
with $1/w_t$ replaced by $w_t$, we have
\begin{equation} \label{eq:aarg}
\begin{split}
	\big( q_n^{|f|}(y)\, w_t(y)\big)^{k}
	& \le \big( \sfc \, \rme^{2 \sfc t^2 n} \big)^{k-1} \, 
	\sum_{z\in\Z^2} |f(z)|^k \, w_t(z)^k \, q_n(y-z) \, \frac{w_t(y)}{w_t(z)} \\
	& = \big( \sfc \, \rme^{2 \sfc t^2 n} \big)^{k-1} \, 
	q^{|f|^k w_t^{k-1}}_n(y) \, w_t(y) \,.
\end{split}
\end{equation}
Introducing the function
\begin{equation}\label{eq:Gm}
	G(y_1, \ldots, y_m) :=
	\prod_{j=1}^m | g(y_j) |^{|J^j|} \, w_t(y_j)^{|J^j|-1}  \,,
\end{equation}
and recalling the notation \eqref{eq:kernel1m}, we can thus write
\begin{equation}\label{eq:Gmappl}
\begin{split}
	\big\Vert \overline{\sfq}_{L}^{|g|,J} \, \cW_t \big\Vert_{\ell^q}^q
	& \le \prod_{j=1}^m \big( \sfc \, \rme^{2 \sfc t^2 n} \big)^{q(|J^j|-1)}
	\sum_{\by \in (\Z^2)^m}
	\max_{1 \le n \le L} \, \big( q_n^{\otimes m, G}(\by) \, w_t^{\otimes m}(\by) \big)^q \\
	& \le  \overline{\scrC}^{\,q( h-m)}\,
	\Big\| \max_{1 \le n \le L} \, q_n^{\otimes m, G} \, w_t^{\otimes m}
	\Big\|_{\ell^q}^q \,,
\end{split}
\end{equation}
because $\sfc \, \rme^{2 \sfc t^2 n} \le \overline{\scrC}$,
see \eqref{eq:CbarC}, and $\sum_{j=1}^m |J^j| = h$.
We can now apply \eqref{eq:max-averaged-rw} to get
\begin{equation}\label{eq:Gmconcl}
	\big\Vert \overline{\sfq}_{L}^{|g|,J} \, \cW_t \big\Vert_{\ell^q}
	\le \tfrac{q}{q-1} \, \overline{\scrC}^{\,h} \,
	\big\| G \, w_{t}^{\otimes m} \big\|_{\ell^q} \,.
\end{equation}
It remains to compute
\begin{equation} \label{eq:Gmcomp}
	\big\| G \, w_{t}^{\otimes m} \big\|_{\ell^q}
	= \prod_{j=1}^m  \bigg( \sum_{y_j \in \Z^2} \,
	 | g(y_j) |^{q|J^j|} \, w_t(y_j)^{q|J^j|} \bigg)^{1/q}
	 = \prod_{j=1}^m \big\| g \, w_t \big\|_{\ell^{q|J^j|}}^{|J^j|} \,.
\end{equation}
Since $J \ne *$, we have $|J^j| \ge 2$ for at least one $j$, say for $j=1$, hence
for $k = |J^1|$ we bound
$\big\| \cdot \big\|_{\ell^{qk}}^{k}
\le \big\| \cdot \big\|_{\ell^\infty}^{k-2} \, \big\| \cdot \big\|_{\ell^{2q}}^{2}$,
while for all other $k = |J^j| \ge 1$
we simply bound $\big\| \cdot \big\|_{\ell^{qk}}^{k}
\le \big\| \cdot \big\|_{\ell^\infty}^{k-1} \, \big\| \cdot \big\|_{\ell^{q}}$.
Since $\sum_{j=1}^m |J^j| = h$, and $m \le h-1$ for $J \ne *$, we obtain
\begin{equation*}
	\big\| G \, w_{t}^{\otimes m} \big\|_{\ell^q}
	\le \big\| g \, w_t \big\|_{\ell^\infty}^{h - m - 1}
	\, \big\| g \, w_t \big\|_{\ell^{2q}}^{2}
	\, \big\| g \, w_t \big\|_{\ell^{q}}^{m-1}
	\le \big\| g \, w_t \big\|_{\ell^{2q}}^{2}
	\, \big\| g \, w_t \big\|_{\ell^{q}}^{h-2} \,,
\end{equation*}
because $\|\cdot\|_{\ell^\infty} \le \|\cdot\|_{\ell^q}, \|\cdot\|_{\ell^{2q}}$. This completes
the proof of the first line of \eqref{eq:est4-new}.

We next prove \eqref{eq:est4+-new}.
We may assume that 
$I=\{\{1,2\}, \{3\},\ldots,\{h\}\}$.
Let us fix a partition $J = \{J^1, \ldots, J^{m}\}$ with $J \not\supseteq I$, 
say $1 \in J^1$ and $2 \in J^2$. Then we can write
\begin{equation*}
\begin{split}
	\big\Vert \overline{\sfq}_{L}^{|g|,J} \, \cW_t   \, \cV_s^I \big\Vert_{\ell^q}^q
	\le 
	\sum_{\by \in (\Z^2)^m}
	w_s(y^1-y^2)^q \,
	\max_{1 \le n \le L} \, \bigg\{ &
	\big( q_n^{|g|}(y_1) \, w_t(y_1)\big)^{q |J^1|} \, \\
	& \times \prod_{j=2}^m \big( q_n^{|g|}(y_j) \, w_t(y_j)\big)^{q |J^j|} \bigg\} \,.
\end{split}
\end{equation*}
By \eqref{eq:max-averaged-rw}  for $m=1$, we can bound $q_n^{|g|}(y_1) \, w_t(y_1) \le
\|q_n^{|g|} \, w_t \|_{\ell^\infty} \le \overline{\scrC} \, \|g \, w_t\|_{\ell^\infty}$. 
Then the sum over $y_1 \in \Z^2$
yields $\|w_s\|_{\ell^q}^q$.
If we define $G'$ as $G$ from \eqref{eq:Gm} with the product
ranging from $2$ to~$m$, then
arguing as in \eqref{eq:aarg}-\eqref{eq:Gmappl} we get
\begin{equation*}
	\big\Vert \overline{\sfq}_{L}^{|g|,J} \, \cW_t   \, \cV_s^I \big\Vert_{\ell^q}^q
	\le \overline{\scrC}^{\, q (h-(m-1))} \, \|g \, w_t\|_{\ell^\infty}^{\, q |J^1|} \, \|w_s\|_{\ell^q}^q \,
	\Big\| \max_{1 \le n \le L} \, q_n^{\otimes (m-1), G'} \, w_t^{\otimes (m-1)}
	\Big\|_{\ell^q}^q \,.
\end{equation*}
Applying \eqref{eq:max-averaged-rw}, as in \eqref{eq:Gmconcl}-\eqref{eq:Gmcomp}, we then obtain
\begin{equation*}
\begin{split}
	\big\Vert \overline{\sfq}_{L}^{|g|,J} \, \cW_t   \, \cV_s^I \big\Vert_{\ell^q}
	&\le \overline{\scrC}^{\, h-(m-1)} \,\|g \, w_t\|_{\ell^\infty}^{|J^1|} \, \|w_s\|_{\ell^q} \, 
	\tfrac{q}{q-1} \,\overline{\scrC}^{\,m-1} \, 
	\big\| G' \, w_{t}^{\otimes (m-1)} \big\|_{\ell^q} \\
	&\le \tfrac{q}{q-1} \,
	\overline{\scrC}^{\,h} \, \|w_s\|_{\ell^q} \,\|g \, w_t\|_{\ell^\infty}^{|J^1|} \, 
	\prod_{j=2}^m \big\| g \, w_t \big\|_{\ell^{q|J^j|}}^{|J^j|} \\
	&\le \tfrac{q}{q-1} \,
	\overline{\scrC}^{\,h} \, \|w_s\|_{\ell^q} \,\|g \, w_t\|_{\ell^\infty}^{h - (m-1)} \, 
	\|g \, w_t\|_{\ell^q}^{m-1} \, ,
\end{split}
\end{equation*}
where the last inequality holds by $\big\| \cdot \big\|_{\ell^{qk}}^{k}
\le \big\| \cdot \big\|_{\ell^\infty}^{k-1} \, \big\| \cdot \big\|_{\ell^{q}}$
for $k = |I^j| \ge 1$. Since $m \le h-1$, the proof of \eqref{eq:est4+-new} is complete.
\qed

\smallskip
\subsection{Proof of Proposition~\ref{prop:bulk-rw}}
Let us set for short $p := \frac{q}{q-1}$ (so that $\frac{1}{p}+\frac{1}{q} = 1$).
We are going to use a key functional inequality from
\cite[Lemma~6.8]{CSZ23}, in the improved version
from \cite[eq.~(3.21) in the proof of Proposition~3.3]{LZ21+}:
\begin{equation}\label{eq:HLS}
	\sum_{\bz \in (\Z^2)^h_I, \, \bx \in (\Z^2)^h_J}
	\frac{f(\bz) \, g(\bx)}{(1+|\bx-\bz|^2)^{h-1}} \le C_1
	\, p \, q \, \|f\|_{\ell^p} \, \|g\|_{\ell^q}
	\qquad \text{where} \quad
	C_1 := 2^{2h} (1+\pi)^{h} \,.
\end{equation}
(The value of $C_1$ is extracted
from \cite[proof of Proposition~3.3]{LZ21+} where $C_1 \le 2^{3h+1}
(\frac{c}{2})^{h-1} p q$ with $c \le 1+\pi$ from \cite[proof of Lemma~A.1]{LZ21+},
hence $C_1 \le 2^{2h+2} \, (1+\pi)^{h-1}$.)

We show below the following bound on 
$\widehat \sfQ_{L}^{*,*}(\bz, \bx) = \sum_{n=1}^L \prod_{i=1}^h q_n(x^i - z^i)$:
\begin{equation}\label{eq:Green-function-bound}
	\widehat\sfQ_{L}^{*,*}(\bz, \bx)
	\le \frac{C_2 \, \rme^{-\frac{|\bx-\bz|^2}{16\sfc \, L}}}{(1+|\bx-\bz|^2)^{h-1}}
	\qquad \text{where} \quad C_2 := h! \, (200 \, \sfc^2)^{h}  \,.
\end{equation}
Recalling \eqref{eq:ratiocW}, 
since $\widehat \sfQ_{L}^{I,J}(\bz, \bx) = \widehat \sfQ_{L}^{*,*}(\bz, \bx) \, 
\ind_{\{\bz\sim I, \bx \sim J\}}$, see \eqref{eq:Qdef2}-\eqref{eq:QLap},
we obtain
\begin{equation*}
	\big(\cW_t \, \widehat \sfQ_{L}^{I,J} 
	\tfrac{1}{\cW_t}\big)(\bz,\bx) \le 
	\frac{C_2 \, \ind_{\{\bz\sim I, \bx \sim J\}}}{
	(1+|\bx-\bz|^2)^{h-1}} \prod_{i=1}^h \rme^{t |z^i-x^i| - \frac{|z^i-x^i|^2}{16\sfc \, L}}
	\le 
	\frac{C_2 \, \rme^{8 \sfc h t^2  L}
	\, \ind_{\{\bz\sim I, \bx \sim J\}}}{(1+|\bx-\bz|^2)^{h-1}} \,,
\end{equation*}
because $\max_{a\in\R}\{t a - \frac{a^2}{16\sfc \, L} \} = 8 \sfc \, t^2 L$.
Applying \eqref{eq:HLS}, get \eqref{eq:est2} since $800 (1+\pi) \le 4000$.

\smallskip

We next prove \eqref{eq:est2V}. Let $I, J$ be pairs, say $I = \{\{a,b\}, \{c\} \colon c \ne a, c \ne b\}$
and $J = \{\{\tilde a,\tilde b\}, \{c\} \colon c \ne \tilde a, c \ne \tilde b\}$.
For $\bz \sim I$ and $\bx \sim J$ we have $z^a = z^b$, hence
\begin{equation*}
	\frac{1}{\cV_s^I(\bx)} \le \rme^{s |x^{a}-x^{b}|}
	\le \rme^{s \{|x^{a}-z^{a}| + |z^{a}-z^{b}| + |z^{b}-x^{b}|\}}
	= \rme^{s |x^{a}-z^{a}|} \,
	\rme^{s |z^{b}-x^{b}|} \,,
\end{equation*}
and similarly $\frac{1}{\cV_s^J(\bz)} \le \rme^{s |x^{\tilde a}-z^{\tilde a}|} 
\, \rme^{s |z^{\tilde b}-x^{\tilde b}|}$. Arguing as above, we obtain \eqref{eq:est2V}:
\begin{equation*}
\begin{split}
	\big( \tfrac{\cW_t}{\cV_s^J} \, \widehat \sfQ_{L}^{I,J} 
	\tfrac{1}{\cW_t \, \cV_s^I} \big)(\bz,\bx)
	&\le  \frac{C_2 \, \ind_{\{\bz\sim I, \bx \sim J\}}}{
	(1+|\bx-\bz|^2)^{h-1}} \prod_{i=1}^h \rme^{(t + 2s) |z^i-x^i| - \frac{1}{16\sfc \, L}|z^i-x^i|^2} \\
	&\le 
	\frac{C_2 \, \rme^{8 \sfc h (t+2s)^2 L}
	\, \ind_{\{\bz\sim I, \bx \sim J\}}}{(1+|\bx-\bz|^2)^{h-1}} \,.
\end{split}
\end{equation*}

Let us prove \eqref{eq:Green-function-bound}.
By the bound
$q_n(x) \le \frac{6\sfc}{n} \, \rme^{-\frac{|x|^2}{8\sfc\, n}}$
proved in Lemma~\ref{lem:heat-kernel-bound} we obtain
\begin{equation*}
	\sfQ_{n}^{*,*}(\bz, \bx) = \prod_{i=1}^h q_n(x^i - z^i)
	\le \frac{(6\sfc)^h}{n^h} \, \rme^{-\frac{|\bx - \bz|^2}{8\sfc\, n}} \,,
\end{equation*}
hence for $\bx = \bz$ we get $\widehat\sfQ_{L}^{*,*}(\bx, \bx)
= \sum_{n=1}^L \sfQ_{n}^{*,*}(\bz, \bx)
\le (6\sfc)^h \sum_{n=1}^\infty \frac{1}{n^2}
= (6\sfc)^h \, \frac{\pi^2}{6} \le 2 \, (6\sfc)^h$
which is compatible with \eqref{eq:Green-function-bound}.
We next assume that $\bx \ne \bz$: note that for $A = \frac{|\bx-\bz|^2}{8\sfc} > 0$
\begin{equation*}
	\sum_{n=1}^L \frac{\rme^{-\frac{A}{n}}}{n^h}
	\le  \frac{\rme^{-\frac{A}{2L}}}{A^{h-1}} \,
	\bigg\{ \frac{1}{A} \sum_{n=1}^\infty \phi\big(\tfrac{n}{A}\big) \bigg\}
	\qquad \text{where} \qquad \phi(t) := \frac{\rme^{-\frac{1}{2t}}}{t^h} \,.
\end{equation*}
Since $\phi(\cdot)$ is unimodal, we can bound
$\frac{1}{A} \sum_{n=1}^\infty \phi\big(\tfrac{n}{A}\big)
\le \int_0^\infty \phi(t) \, \dd t + \frac{1}{A} \| \phi \|_\infty$ and note that
$\int_0^\infty \phi(t) = 2^{h-1} \int_0^\infty s^{h-2} \, \rme^{-s} \, \dd s 
= 2^{h-1} \, (h-2)!$ while $\|\phi\|_\infty = (2h)^h \, \rme^{-h}
\le 2^h h! / \sqrt{2\pi h} \le \frac{1}{2} 2^h h!$, therefore for $A \ge 1$ we get
$\frac{1}{A} \sum_{n=1}^\infty \phi\big(\tfrac{n}{A}\big) \le 2^h h!$. Overall,
recalling \eqref{eq:QLap}, we have for $\bx \ne \bz$
\begin{equation*}
	\widehat\sfQ_{L}^{*,*}(\bz, \bx)
	\le \sum_{n=1}^L \sfQ_{n}^{*,*}(\bz, \bx)
	\le \frac{(48 \, \sfc^2)^{h} \, \rme^{-\frac{|\bx-\bz|^2}{16\sfc \, L}}}{|\bx-\bz|^{2(h-1)}}
	\, 2^h \, h!
	\le \frac{h! \, (200 \, \sfc^2)^{h} \, \rme^{-\frac{|\bx-\bz|^2}{16\sfc \, L}}}{(1+|\bx-\bz|^2)^{h-1}} \,,
\end{equation*}
where we last bounded $|\bx-\bz|^2 \ge \frac{1}{2} (1+|\bx-\bz|^2)$ for $\bx \ne \bz$.
We have proved \eqref{eq:Green-function-bound}.
\qed

\smallskip

\subsection{Proof of Proposition~\ref{prop:bulk-interacting}}
Let us define $p := \frac{q}{q-1}$ so that $\frac{1}{p}+\frac{1}{q}=1$.
Since
\begin{equation*}
	\Vert \sfA \Vert_{\ell^q \to \ell^q}
	:= \, \sup_{\substack{\sff, \sfg \colon \|\sff\|_{\ell^p} \le 1 , \,
	\|\sfg\|_{\ell^q} \le 1}} \ \sum_{\bz, \bx \in (\Z^2)^h_I}
	\sff(\bz) \, \sfA(\bz,\bx) \, \sfg(\bx) \,,
\end{equation*}
we can bound $\sum_{\bz,\bx} \sff(\bz) \, |\widehat\sfU|^I(\bz,\bx) \, \sfg(\bx) \le 
\big(\sum_{\bz,\bx} \sff(\bz)^p \, |\widehat\sfU|^I(\bz,\bx)\big)^{1/p}
\big(\sum_{\bz,\bx} |\widehat\sfU|^I(\bz,\bx) \, \sfg(\bx)^q\big)^{1/q}$
by Cauchy-Schwarz, hence we obtain
\begin{equation} \label{eq:inview}
	\Vert \sfA \Vert_{\ell^q \to \ell^q}
	\le \max\Bigg\{ \sup_{\bz\in(\Z^2)^h_I} \,
	\sum_{\bx \in (\Z^2)^h_I} \sfA(\bz, \bx) \,, \,
	\sup_{\bx\in(\Z^2)^h_I} \,
	\sum_{\bz \in (\Z^2)^h_I} \sfA(\bz, \bx) \Bigg\} \,.
\end{equation}
We will prove \eqref{eq:est3} and \eqref{eq:est3+} exploiting this bound.

We recall that $U_{n,\beta}(x)$ is defined in \eqref{eq:Ux} and we define
\begin{equation}\label{eq:U}
	U_{n,\beta} := \sum_{x\in\Z^2} U_{n,\beta}(x)
	= \sum_{k=1}^\infty (\sigma_\beta^2)^{k} \!
	\sum_{0 =: n_0 <n_1<\dots <n_k:= n} \
	\prod_{i=1}^k q_{2(n_i - n_{i-1})}(0) \,.
\end{equation}
When we sum $U_{n,\beta}$ for $n=1,\ldots, L$, if we enlarge
the sum range in \eqref{eq:U} by letting each increment $m_i := n_i - n_{i-1}$
vary freely in $\{1,\ldots, M\}$, recalling \eqref{eq:RNlambda} we obtain
\begin{equation}\label{eq:sumUn}
\begin{split}
	\sum_{n=1}^L \rme^{-\lambda n} \, U_{n,\beta}
	& \le \sum_{k=1}^\infty (\sigma_\beta^2)^{k} 
	\bigg( \sum_{m=1}^L \rme^{-\lambda m} \, q_{2m}(0) \bigg)^k
	= \sum_{k=1}^\infty (\sigma_\beta^2 \, R_L^{(\lambda)})^{k}
	= \frac{\sigma_\beta^2 \, R_L^{(\lambda)}}{1-\sigma_\beta^2 \, R_L^{(\lambda)}} \,.
\end{split}
\end{equation}

We next estimate the exponential spatial moments of $U_{n,\beta}(x)$.
Pluggin the second bound from \eqref{eq:rw-weight}
into \eqref{eq:Ux}, writing $x = (x^1, x^2)$
and $x^a = \sum_{i=1}^k (x^a_i - x^a_{i-1})$, we obtain
\begin{equation*}
	\forall a=1,2: \qquad
	\sum_{x\in\Z^2} \rme^{t x^a} \, U_{n,\beta}(x) \le \rme^{\sfc \, \frac{t^2}{2} n} \, U_{n,\beta} \,.
\end{equation*}
From this, by $|x| \le |x^1| + |x^2|$, Cauchy-Schwarz
and $\rme^{t|x^a|} \le \rme^{tx^a} + \rme^{-tx^a}$,
we deduce that
\begin{equation} \label{eq:momU}
	\sum_{x\in\Z^2} \rme^{t |x|} \, U_{n,\beta}(x) \le 2 \, \rme^{2\sfc \, t^2 n} \, U_{n,\beta} \,.
\end{equation}

\smallskip

We now fix a partition $I  = \{I^1, \ldots, I^m\} \ne *$ and
a \emph{pair} $J = \{\{a,b\}, \{c\}: c \ne a,b\}$.
Our goal is to prove \eqref{eq:est3+}, which also yields
\eqref{eq:est3} for $s=0$.
By \eqref{eq:ratiocW} and \eqref{eq:ratioV} we have the following rough bound,
for any $\tau \in \{-1,+1\}$:
\begin{equation} \label{eq:ratios}
	\frac{\cW_t(\bz) \, \cV_s^J(\bz)^{\tau}}{\cW_t(\bx) \, \cV_s^J(\bx)^{\tau}}
	\le \rme^{2(t + s)|x^a-z^a|}
	\prod_{c \ne a,b} \, \rme^{(t + s)|x^c-z^c|} \,.
\end{equation}
We may order $|I^1| \ge |I^2| \ge \ldots \ge |I^m|$,
so that $|I^1| \ge 2$.
Given $\bz, \bx \in (\Z^2)^h_I$, denoting by $x^{I^j}$
the common value of $x^a$ for $a \in I^j$, by \eqref{eq:Qdef2} we can write
\begin{equation*}
	\sfQ^{I,I}_n(\bz,\bx) 
	\,=\, 
	q_n(x^{I^1} - z^{I^1})^{|I^1|}
	\prod_{j=2}^m
	q_n(x^{I^j} - z^{I^j})^{|I^j|} 
	\,\le\,
	q_n(x^{I^1} - z^{I^1})^2 \,
	\prod_{j=2}^m
	q_n(x^{I^j} - z^{I^j}) \,,
\end{equation*}
because $q_n(\cdot) \le 1$.
Since $|\bbE[\xi_\beta^I]| \le \sigma_\beta^2$ by assumption, from \eqref{eq:sfU} we can bound
\begin{equation*}
	|\sfU|^I_{n,\beta}(\bz,\bx)
	\le U_{n,\beta}(x^{I^1} - z^{I^1}) \, 
	\prod_{j=2}^m q_n(x^{I^j} - z^{I^j}) \,,
\end{equation*}
therefore
by \eqref{eq:momU}, \eqref{eq:ratios} and the first bound in \eqref{eq:rw-weight+}
we obtain
\begin{equation} \label{eq:sommax2}
\begin{split}
	\sum_{\bx \in (\Z^2)^h_I} \bigg( |\sfU|^I_{n,\beta}(\bz,\bx) \,
	\frac{\cW_t(\bz) \, \cV_s(\bz)^{\tau}}{\cW_t(\bx) \, \cV_s(\bx)^{\tau}}
	\bigg)
	\le & \, 
	2^h \, \rme^{4h \sfc \, (t + s)^2 n} \, U_{n,\beta} \,,
\end{split}
\end{equation}
which yields, recalling \eqref{eq:ULap},
\begin{equation} \label{eq:pluin}
	\sup_{\bz\in(\Z^2)^h_I} \,
	\sum_{\bx \in (\Z^2)^h_I} |\widehat \sfU|^{J}_{L, \lambda, \beta}(\bz, \bx) 
	\frac{\cW_t(\bz) \, \cV_s(\bz)^{\tau}}{\cW_t(\bx) \, \cV_s(\bx)^{\tau}}
	\le 1 + 2^h \, 
	\rme^{4h \sfc \, (t + s)^2 L} \sum_{n=1}^L \rme^{-\lambda n} \, U_{n,\beta} \,,
\end{equation}
and the same holds exchanging $\bx$ and $\bz$ by symmetry
(note that the bound \eqref{eq:ratios} is symmetric in $\bx \leftrightarrow \bz$).
Recalling \eqref{eq:inview} and  \eqref{eq:sumUn}, we obtain \eqref{eq:est3+}
(hence  \eqref{eq:est3}).
\qed

\end{document}